\documentclass[]{amsart}

\usepackage[Symbol]{upgreek}

\usepackage{mathdots}

\usepackage{amssymb}
\usepackage{amscd}
\usepackage{epic}
\usepackage{mathrsfs}
\usepackage{accents}
\usepackage{stmaryrd}

\usepackage{pgfplots, pgfplotstable}
\pgfplotsset{compat=1.18}
\usepgfplotslibrary{fillbetween}
\usetikzlibrary{patterns}

\usepackage{hyperref}

\usepackage{xr}

\usepackage{etoolbox} 

\numberwithin{equation}{section}

\def\HLO{\Upl\Upo}

\newcommand{\bbold}{\mathbb}

\renewcommand\Re{\operatorname{Re}}
\renewcommand\Im{\operatorname{Im}}

\def\R { {\bbold R} }
\def\Q { {\bbold Q} }
\def\Z { {\bbold Z} }

\def\N { {\bbold N} }
\def\T { {\bbold T} }

\def \I{\operatorname{I}}

\def \cf{\operatorname{cf}}
\def \ci{\operatorname{ci}}
\def \te{\operatorname{te}}
\def \cH{\mathcal{H}}

\def \order{\operatorname{order}}

\def \ex{\operatorname{e}}

\renewcommand\epsilon{\varepsilon}

\def \d{\operatorname{d}}

\def \ev{\operatorname{e}}

\def \<{\langle}
\def \>{\rangle}

\def \tilde {\widetilde}

\def \hat {\widehat}

\def \supp {\operatorname{supp}}

\def \((  {(\!(}
\def \)) {)\!)}

\def \Li{\operatorname{Li}}
\def \res{\operatorname{res}}

\DeclareMathSymbol{\precequ}{\mathrel}{symbols}{"16}
\DeclareMathSymbol{\succequ}{\mathrel}{symbols}{"17}

\def \nasymp{\not\asymp}

\newcommand{\claim}[2][\!\!]{\medskip\noindent {\it Claim #1}\/: {\it #2}\medskip}
\newcommand{\case}[2][\!\!]{\medskip\noindent {\it Case #1}\/: {\it #2}\/}
\newcommand{\subcase}[2][\!\!]{\medskip\noindent {\it Subcase #1}\/: {\it #2}\/}

\newtheorem{theorem}{Theorem}[section]
\newcounter{tmptheorem}
\newtheorem{lemma}[theorem]{Lemma}
\newtheorem{prop}[theorem]{Proposition}
\newtheorem{cor}[theorem]{Corollary}

\theoremstyle{definition}

\theoremstyle{remark}

\newtheorem*{remark}{Remark}
\newtheorem*{question}{Question}

\newcommand{\abs}[1]{\lvert#1\rvert}
\newcommand{\dabs}[1]{\lVert#1\rVert}

\def \fM {{\mathfrak M}}

\def \fm {{\mathfrak m}}

\def \fv {{\mathfrak v}}
\def \fw {{\mathfrak w}}

\def \No{\text{{\bf No}}}

\def \BM{\operatorname{BM}}

\let\oldi\i
\let\oldj\j
\renewcommand\i{\relax\ifmmode{\boldsymbol{i}}\else\oldi\fi}
\renewcommand\j{\relax\ifmmode{\boldsymbol{j}}\else\oldj\fi}

\renewcommand\leq{\leqslant}
\renewcommand\geq{\geqslant}
\renewcommand\preceq{\preccurlyeq}
\renewcommand\succeq{\succcurlyeq}
\renewcommand\le{\leq}
\renewcommand\ge{\geq}
\renewcommand\frak{\mathfrak}

\DeclareMathAlphabet{\mathbf}{OML}{cmm}{b}{it}

\DeclareFontFamily{U}{fsy}{}
\DeclareFontShape{U}{fsy}{m}{n}{<->s*[.9]psyr}{}
\DeclareSymbolFont{der@m}{U}{fsy}{m}{n}
\DeclareMathSymbol{\der}{\mathord}{der@m}{182}

\DeclareSymbolFont{der@m}{U}{fsy}{m}{n}
\DeclareMathSymbol{\derdelta}{\mathord}{der@m}{100}

\newcommand\ndeg{\operatorname{ndeg}}

\DeclareSymbolFont{imag@m}{OT1}{cmr}{m}{ui}
\DeclareMathSymbol{\imag}{\mathord}{imag@m}{105}

\DeclareFontFamily{OMS}{smallo}{}
\DeclareFontShape{OMS}{smallo}{m}{n}{<->s*[.65]cmsy10}{}
\DeclareSymbolFont{smallo@m}{OMS}{smallo}{m}{n}
\DeclareMathSymbol{\smallo}{\mathord}{smallo@m}{79}

\DeclareFontFamily{OMS}{largerdot}{}
\DeclareFontShape{OMS}{largerdot}{m}{n}{<->s*[.8]cmsy10}{}
\DeclareSymbolFont{largerdot@m}{OMS}{largerdot}{m}{n}
\DeclareMathSymbol{\largerdot}{\mathord}{largerdot@m}{15}

\DeclareMathSymbol{\llambda}{\mathord}{der@m}{108}
\DeclareMathSymbol{\rrho}{\mathord}{der@m}{114}

\def \upg{\upgamma}
\def \Upg{\Upgamma}
\def \upl{\uplambda}
\def \Upl{\Uplambda}
\def \upo{\upomega}
\def \Upo{\Upomega}

\def \Upd{\Updelta}

\newcommand{\equationqed}[1]{\[\pushQED{\qed}#1 \qedhere\popQED\]\let\qed\relax}
\newcommand{\alignqed}[1]{\begin{align*}\pushQED{\qed} #1 \qedhere\popQED\end{align*}\let\qed\relax}

\makeatletter
\newcommand{\dminus}{\mathbin{\text{\@dminus}}}

\newcommand{\@dminus}{%
  \ooalign{\hidewidth\raise1ex\hbox{\bf.}\hidewidth\cr$\m@th-$\cr}%
}
\makeatother

\def\ddeg{\operatorname{ddeg}}

\def \Cc{\mathcal{C}}
\def \C{\mathcal{C}}

\def \inv{\operatorname{inv}}

 \newcommand{\addappendix[1]}{%
  \section*{\appendixname. #1}
  \counterwithin*{figure}{section}
  \stepcounter{section}
  \renewcommand{\thesection}{A}
  \renewcommand{\thefigure}{\thesection.\arabic{figure}}
}

\begin{document}

\title{Analytic Hardy Fields}

\author[Aschenbrenner]{Matthias Aschenbrenner}
\address{Kurt G\"odel Research Center for Mathematical Logic\\
Universit\"at Wien\\
1090 Wien\\ Austria}
\email{matthias.aschenbrenner@univie.ac.at}

\author[van den Dries]{Lou van den Dries}
\address{Department of Mathematics\\
University of Illinois at Urbana-Cham\-paign\\
Urbana, IL 61801\\
U.S.A.}
\email{vddries@illinois.edu}


\begin{abstract}  We show that maximal analytic Hardy fields are $\eta_1$ in the sense of Hausdorff.  We also prove  various embedding theorems about analytic Hardy fields. For example, the ordered differential field $\T$ of transseries is
shown to be isomorphic to an analytic Hardy field. 
\end{abstract}

\date{September 2025}


\maketitle

\tableofcontents

\begingroup
\setcounter{tmptheorem}{\value{theorem}}
\setcounter{theorem}{0} 
\renewcommand\thetheorem{\Alph{theorem}}

\section*{Introduction}\label{intro} 

\noindent
This is a follow-up on \cite{ADHfgh} where the main result is that for any Hardy field $H$ and countable subsets
$A< B$ of $H$ there exists $y$ in a Hardy field extension of $H$ such that $A < y < B$. Equivalently, (the underlying ordered set of) any maximal Hardy field is $\eta_1$ in the sense of Hausdorff. In this result we do not require $H\subseteq \Cc^{\omega}$, and the glueing constructions in \cite{ADHfgh} do not give $y\in \Cc^{\omega}$, even if $H\subseteq \Cc^{\omega}$; see {\em Notations and Conventions}\/ at the end of this introduction for the notation used here.  We call a Hardy field~$H$ {\em smooth\/} if $H\subseteq \Cc^\infty$ and {\em analytic\/} if $H\subseteq \Cc^{\omega}$. By \cite[Corollary~11.20]{ADH4} and [ADH, 16.0.3, 16.6.3], maximal Hardy fields, maximal smooth Hardy fields,  and maximal analytic Hardy fields are all elementarily equivalent to the ordered differential field $\T$ of transseries, and have no proper $\d$-algebraic $H$-field extension with constant field $\R$. We shall tacitly use these facts throughout.

Some may view non-analytic Hardy fields as artificial, since most Hardy fields that occur ``in nature'' are analytic. (But see
\cite{Gokhman,Grelowski} for Hardy fields $H\not\subseteq\Cc^\infty$, and \cite{RSW} for  Hardy fields $H\subseteq\Cc^\infty$, $H\not\subseteq\Cc^\omega$.) To conciliate this view and answer an obvious question we prove in Section~\ref{s4} the analytic version of \cite{ADHfgh}: 

\begin{theorem}\label{aneta} If $H$ is an analytic Hardy field with countable subsets $A<B$, then there exists $y\in \Cc^{\omega}$ in a Hardy field extension of $H$ such that $A < y < B$. 
\end{theorem}

\noindent
Equivalently, all maximal analytic Hardy fields are $\eta_1$. The theorem goes through for smooth Hardy fields  with $y\in \Cc^{\infty}$ in the conclusion; this can be obtained by
refining the glueing constructions from~\cite{ADHfgh} (as was actually done in an early version of that paper at the cost of three extra pages). 
Here we take care of the smooth and analytic versions simultaneously. Compared to \cite{ADHfgh} the new tool we use is a powerful theorem due to Whitney on approximating any $\Cc^n$-function or $\Cc^{\infty}$-function by an analytic function, where the approximation also takes derivatives into account.  From that we obtain an analogue for germs, namely Corollary~\ref{apomega}, which in turn we use to derive Theorem~\ref{aneta} from various results in
the non-analytic setting of \cite{ADHfgh}. 

In the course of establishing Theorem~\ref{aneta} in Sections~\ref{sec:pc} and~\ref{s4} we revisit results on pc-sequences and on extensions of type (b) 
from \cite{ADHfgh}. In Section~\ref{s4} (see Theorem~\ref{thm:dense}) this  also leads to: 

\begin{theorem}\label{thdense} If $H$ is a maximal analytic Hardy field, then $H$ is dense in any Hardy field extension of $H$.
\end{theorem} 

\noindent
If all maximal analytic Hardy fields are maximal Hardy fields, which seems to us implausible, then of course the theorems above would be trivially true. Can a maximal analytic Hardy field ever be a maximal Hardy field? For all we know answering questions of this kind might involve set-theoretic assumptions like CH. In  \cite{ADHfgh} and the present paper we ran into other set-theoretic issues of this kind, and in Section 8 we state some problems that arose this way. 

\medskip\noindent
Sections 5--7 prove embedding theorems about (not necessarily maximal) analytic Hardy fields. A special case of a result in Section~7:  the ordered differential field  $\T$ is isomorphic over $\R$ to an analytic Hardy field extension of $\R$.

\subsection*{Notations and conventions}  We take these from \cite[end of introduction]{ADHfgh}, but for the convenience of the reader we list here what is most needed. 

We let $i$, $j$, $k$, $l$, $m$, $n$ range over $\N=\{0,1,2,\dots\}$. We let $\Cc$ be the ring of germs at $+\infty$ of continuous functions~${(a,+\infty)\to \R}$, ${a\in\R}$.
Let $f$, $g$ range over~$\Cc$, with representatives~$(a,+\infty)\to\R$ ($a\in\R$) of $f$, $g$ also denoted by $f$, $g$. 
Then on $\Cc$ we have   binary relations $\le$, $<_{\ex}$ given by $f\leq g :\Leftrightarrow f(t)\le g(t)$, eventually, and~$f<_{\ex} g :\Leftrightarrow f(t) < g(t)$, eventually, as well as~$\preceq$,~$\prec$,~$\asymp$,~$\sim$   defined as follows:   
\begin{align*} f\preceq g\quad &:\Longleftrightarrow\quad \text{$|f|\le c|g|$ for some  $c\in \R^{>}$,}\\
f\prec g\quad &:\Longleftrightarrow\quad   \text{$g\in \Cc^\times$ and $\abs{f}\leq c\abs{g}$ for all $c\in\R^>$},\\
f\asymp g \quad &:\Longleftrightarrow\quad  \text{$f\preceq g$ and $g\preceq f$,}\\
f\sim g\quad &:\Longleftrightarrow\quad   f-g\prec g.
\end{align*}
For $r\in \N\cup\{\infty\}$ we let
$\Cc^r$ be the subring of $\Cc$ consisting of the germs of $r$ times continuously differentiable functions $(a,+\infty)\to \R$, $a\in \R$. Thus $\Cc^{<\infty} :=  \bigcap_{n}\Cc^n$ is a differential ring with the obvious derivation, 
and has $\Cc^{\infty}$ as a differential subring. 
 We let $\Cc^{\omega}$ be the differential subring of $\Cc^{\infty}$ consisting of the germs of real analytic functions~$(a,+\infty)\to \R$, $a\in \R$. 
A {\it Hausdorff field}\/ is a subfield $H$ of $\Cc$; it is naturally also an ordered and valued field (see~\cite[Section~2]{ADH5}), with the relations~$\le$,~$\preceq$ on~$\Cc$ restricting to the ordering of $H$ and the dominance relation associated to the valuation of $H$, respectively.
 A {\em Hardy field\/}  is a differential subfield  of $\Cc^{<\infty}$. (So every Hardy field is a Hausdorff field.)

 The prefix~``$\d$'' abbreviates ``differentially''; for example, ``$\d$-algebraic'' means ``differentially algebraic''. 

\subsection*{Acknowledgements} We thank the anonymous referee for suggestions as to how to improve readability
of the paper.

\endgroup

\section{Whitney's Approximation Theorem}

\noindent
In this section we let $r\in\N\cup\{\infty\}$ and $a,b\in \R$. 
We shall use the one-variable case of an approximation theorem due to Whitney~\cite[Lem\-ma~6]{Whitney} to upgrade various constructions of smooth functions  to analytic functions. To formulate this theorem we introduce some notation. 
Let~${U\subseteq\R}$ be  open. Then $\Cc^m(U)$ denotes the $\R$-algebra of $\Cc^m$-functions $U\to\R$, 
with $\Cc(U):=\Cc^0(U)$ and
 $\C^\infty(U):=\bigcap_m \C^m(U)$,
and~$\C^\omega(U)$ denotes the $\R$-algebra of analytic functions $U\to\R$, so $\C^\omega(U)\subseteq\C^\infty(U)$.
Let~$S\subseteq U$ be nonempty.
For $f$ in $\Cc(U)$ we set
$$\dabs{f}_S\ :=\ \sup\big\{ \abs{f(s)}:\, s\in S\big\} \in [0,\infty], $$
so for $f,g\in\Cc(U)$ and $\lambda\in \R$ (and the convention $0\cdot\infty=\infty\cdot 0=0$) we have 
$$\dabs{f+g}_S\le \dabs{f}_S+\dabs{g}_S,\quad \dabs{\lambda f}_S=|\lambda|\cdot\dabs{f}_S,\quad\text{and}\quad
\dabs{fg}_S\leq\dabs{f}_S\dabs{g}_S.$$  
If $\emptyset\neq S'\subseteq S$ then 
$\dabs{f}_{S'}\leq\dabs{f}_S$.
Next,
let~$f\in\Cc^m(U)$. 
We then put
$$\dabs{f}_{S;\,m}\ :=\ \max\big\{ \dabs{f}_S,\dots,\dabs{f^{(m)}}_S\big\}\in [0,\infty].$$
Then again for  $f,g\in\Cc(U)$ and $\lambda\in \R$  we have 
$$\dabs{f+g}_{S;m}\le \dabs{f}_{S;m}+\dabs{g}_{S;m},\quad \dabs{\lambda f}_{S;m}=|\lambda|\cdot\dabs{f}_{S;m},$$
and 
\begin{equation}\label{eq:prod norm}
\dabs{fg}_{S;\,m}\  \leq\ 2^m \dabs{f}_{S;\,m}\dabs{g}_{S;\,m}
\end{equation}
Let $f\in \Cc(U)$. For $U=\R$ we set $\dabs{f}_{m}:=\dabs{f}_{\R;\,m}$.  
For $k\leq m$ and $\emptyset\neq S'\subseteq S\subseteq U$ we have
$\dabs{f}_{S';\,k} \leq \dabs{f}_{S;\,m}$. Moreover, $\dabs{f}_{S;\,m}$ does not change if
$S$ is replaced by its closure in $U$.
 
\begin{theorem}[Whitney]\label{thm:WAP dim 1, general} Let $(a_n)$, $(b_n)$, $(\varepsilon_n)$ be sequences in $\R$ and $(r_n)$
in $\N$ such that $a_0= b_0$, $(a_n)$ is strictly decreasing, $(b_n)$ is strictly increasing, and $\varepsilon_n>0$, $r_n\le r$ for all $n$. Set $I:=\bigcup_n K_n$, where $K_n:=[a_n, b_n]$.   Then, for any $f\in \Cc^r(I)$, there exists $g\in\Cc^\omega(I)$ such that for all $n$ we have
$\dabs{{f-g}}_{K_{n+1}\setminus K_n;\,r_n}<\varepsilon_n$.
\end{theorem}

\noindent
For a self-contained proof of Theorem~\ref{thm:WAP dim 1, general}, see the appendix to this paper.

\medskip
\noindent
We let $\Cc^m_a$ be the $\R$-algebra of functions $f\colon[a,+\infty)\to\R$ which extend to a function in $\Cc^m(U)$
for some open neighborhood  $U\subseteq\R$ of $[a,+\infty)$. Likewise we define $\Cc^\infty_a$ and~$\Cc^\omega_a$, and $\Cc_a:=\Cc_a^0$;
see~\cite[Section~3]{ADH5}.
For $f\in\Cc_a^m$ and nonempty~${S\subseteq  [a,+\infty)}$   we   put
$\dabs{f}_{S;\,m}:=\dabs{g}_{S;\,m}$ where~$g\in\Cc^m(U)$ is any extension of~$f$ to an open neighborhood $U\subseteq\R$ of $[a,+\infty)$. 
We shall use the following special case of Theorem~\ref{thm:WAP dim 1, general}:

\begin{cor}\label{cor:WAP dim 1} Let $f\in \Cc^r_b$, and 
let   $(b_n)$ be a strictly increasing sequence
in $\R$ such that $b_0=b$ and~${b_n\to\infty}$ as $n\to\infty$, and let $(\varepsilon_n)$ be a sequence in $\R^>$  and
$(r_n)$ be a sequence in~$\N$ with~$r_n\leq r$ for all $n$.   Then there exists $g\in\Cc_b^\omega$ such that for all $n$ we have   $\dabs{{f-g}}_{[b_n,b_{n+1}];\,r_n}<\varepsilon_n$.
\end{cor}
\begin{proof} 
Extend $f$ to a function in~$\Cc^r(I)$, also denoted by $f$,  where $I:=(a,+\infty)$, ${a<b}$, and take
a strictly decreasing sequence $(a_n)$ in $\R$ with $a_0=b_0$ and~${a_n\to a}$ as $n\to+\infty$.
 Now apply Theorem~\ref{thm:WAP dim 1, general}.
\end{proof}

\noindent
Here is a useful  reformulation of Corollary~\ref{cor:WAP dim 1}: 

\begin{cor}\label{cor:WAP dim 1, 1}
Let     $f\in \Cc^r_b$ and $\varepsilon\in\Cc_b$ be such that $\varepsilon>0$ on $[b,+\infty)$. Then there exists $g\in\Cc_b^\omega$ such that  $\abs{(f-g)^{(k)}(t)} < \varepsilon(t)$  for all $t\ge b$ and
$k\leq \min\!\big\{r,1/\varepsilon(t)\big\}$.
\end{cor}
\begin{proof}
Take a strictly increasing sequence $(b_n)$ in $\R$ with $b_0=b$ and $b_n\to \infty$ as~$n\to\infty$, and
for each $n$, set 
$$\varepsilon_n:=\min\big\{\varepsilon(t): t\in[b_n,b_{n+1}]\big\}\in\R^>,\quad
r_n:= \min\left\{ r, \big\lfloor  \dabs{1/\varepsilon}_{[b_n,b_{n+1}]}\big\rfloor\right\}\in\N.$$
Corollary~\ref{cor:WAP dim 1} yields  $g\in\Cc^\omega_b$ such that $\dabs{f-g}_{[b_n,b_{n+1}];\,r_n}<\varepsilon_n$ for all~$n$. Then for~$t\in [b_n,b_{n+1}]$ and $k\leq \min\!\big\{r,1/\varepsilon(t)\big\}$ we have
$k\leq r_n$ and so
\[\abs{(f-g)^{(k)}(t)} \leq \dabs{f-g}_{[b_n,b_{n+1}];\,r_n}<\varepsilon_n\leq \varepsilon(t).\qedhere\]
\end{proof}

\noindent
This leads to an improved version of \cite[Lemma~2.5]{ADHfgh}:

\begin{lemma}\label{smooth->analytic}
Let $f, g\in\Cc_b$ be such that $f < g$ on $[b,+\infty)$. Then there exists $y\in\Cc^\omega_b$   such that  $f < y < g$ on $[b,+\infty)$. 
\end{lemma}
\begin{proof}
Let $z:=\frac{1}{2}(f+g)\in\Cc_b$ and $\varepsilon:=\frac{1}{2}(g-f)\in\Cc_b$.
Corollary~\ref{cor:WAP dim 1, 1} (with $r=0$) then yields
 $y\in\Cc_{b}^\omega$ such that $\abs{y-z}<\varepsilon$ on $[b,+\infty)$, so
  $f< y < g$ on $[b,+\infty)$. 
\end{proof}

\noindent
Thus we can replace ``$\phi\in \mathcal{C}^{\infty}$''   by  ``$\phi\in\Cc^\omega$'' in the statements of Lemma~2.7 and   Corollary~2.8 in \cite{ADHfgh}. 
Here is another consequence of Corollary~\ref{cor:WAP dim 1, 1}:

\begin{cor}\label{cor:WAP dim 1, 2}
Let   $f\in\Cc^r$  and $\varepsilon\in\Cc$, $\varepsilon>_{\ex} 0$. Then there exists $g\in\Cc^\omega$ such that for all $k\le r$ we have
$\abs{{(f-g)^{(k)}}} <_{\ex}  \varepsilon$.
\end{cor}
\begin{proof}
Pick $a$ and representatives of $f$ in $\Cc_a^r$ and of $\varepsilon$  in $\Cc_a$, also denoted by~$f$,~$\varepsilon$, with
$\varepsilon>0$ on $[a,+\infty)$.  
Take $\varepsilon^*\in\Cc_a$ with $0<\varepsilon^*\leq \varepsilon$ on $[a,+\infty)$ and $\varepsilon^*\prec 1$.
Corollary~\ref{cor:WAP dim 1, 1} applied to $\varepsilon^*$ in place of $\varepsilon$ yields $g\in \Cc_a^\omega$ such that~$\abs{(f-g)^{(k)}(t)} < \varepsilon^*(t)$ for all $t\ge a$ and $k\leq \min\!\big\{r,1/\varepsilon^*(t)\big\}$.
Given $k\leq r$, take $b\geq a$ such that  $k\le 1/\varepsilon^*(t)$ for all~$t\geq b$; then
$\abs{(f-g)^{(k)}(t)} < \varepsilon(t)$ for such $t$.
\end{proof}


\noindent
Our next goal is to prove a version of Corollary~\ref{cor:WAP dim 1, 2} for  approximating germs in~$\Cc^{<\infty}$ by germs in $\Cc^{\omega}$: see Corollary~\ref{apomega} below.
First a lemma about glueing two approximations $g_{-}$ and $g_{+}$ to a function $f$ to make a single approximation $g$ to $f$ that combines properties of $g_{-}$ and $g_{+}$:
 
\begin{lemma}
Let $f\in \Cc_{a_0}$ and $a_0 \le a < b$. Suppose $f$ is of class $\Cc^n$ on $[a,+\infty)$ and of class $\Cc^{n+1}$ on
$[b,+\infty)$. Let also functions $\varepsilon\in \Cc_{a_0}$    
and  $g_{-}, g_{+}\in \Cc^{\infty}_{a_0}$ be given such that 
\begin{itemize}
\item $\varepsilon >0$ on $[a_0,+\infty)$;
\item $|(f-g_{-})^{(j)}| < \varepsilon$   on $[a,+\infty)$ for $j=0,\dots,n$; and
\item $|(f-g_{+})^{(j)}| < \varepsilon$ on $[b,+\infty)$ for $j=0,\dots,n+1$.
\end{itemize}
 Then, for any $\delta\in \R^{>}$, there is a function $g\in \Cc^{\infty}_{a_0}$ and a $b'> b$ such that: \begin{enumerate}
 \item[(i)] $g=g_{-}$ on $[a_0,b]$ and $g=g_{+}$ on $[b',+\infty)$;
\item[(ii)]  $|(f-g)^{(j)}| < (1+\delta)\varepsilon$ on $[a,+\infty)$ for $j=0,\dots,n$; and
\item[(iii)]  $|(f-g)^{(j)}| < \varepsilon$ on $[b',+\infty)$ for $j=0,\dots,n+1$. 
\end{enumerate}
\end{lemma} 
\begin{proof} Let $b'>b$, set $\beta:= \alpha_{b,b'}$ as in \cite[(3.4)]{ADHfgh}, and $g:= (1-\beta) g_{-}+\beta g_{+}$ on~$[a_0,+\infty)$, so $g\in \Cc_{a_0}^{\infty}$. Let $\delta>0$; we show that if $b'-b$ is sufficiently large, then $g$ satisfies~(i), (ii), (iii). It is clear that (i) holds, and so (iii) as well. Then the inequality in (ii) holds on~$[a,b]$ and on $[b',+\infty)$, so it suffices to consider what happens  on $[b,b']$. There we have for $j=0,\dots,n$:
$$(f-g)^{(j)}\ =
\ f^{(j)}-\big((1-\beta) g_{-}^{(j)}+\beta g_{+}^{(j)}\big)-\sum_{i=0}^{j-1}\binom{j}{i}\beta^{(j-i)}\big(g_{+}^{(i)}-g_{-}^{(i)}\big),$$
and $$f^{(j)}-\big((1-\beta) g_{-}^{(j)}+\beta g_{+}^{(j)}\big)=(1-\beta)\big(f-g_{-})^{(j)} + \beta\big(f-g_{+})^{(j)},$$ so
$$\big|f^{(j)}-\big((1-\beta) g_{-}^{(j)}+\beta g_{+}^{(j)}\big)\big|\ \le\ \max\big\{\big|(f-g_{-})^{(j)}\big|, \big|(f-g_{+})^{(j)}\big|\big\}\ <\ \varepsilon  \text{ on $[b,b']$.}$$
By \cite[(3.5)]{ADHfgh} we have reals $C_m\geq 1$ (independent of $b'$) with $\abs{\beta^{(m)}}\leq C_m/(b'-b)^m$. 
Hence for $j=0,\dots,n$ we have on $[b,b']$: 
$$ \left|\sum_{i=0}^{j-1}\binom{j}{i}\beta^{(j-i)}\big(g_{+}^{(i)}-g_{-}^{(i)}\big)\right|\ \le\ 
\sum_{i=0}^{j-1}\binom{j}{i}\frac{C_{j-i}}{(b'-b)^{j-i}}\,\big|g_{+}^{(i)}-g_{-}^{(i)}\big|$$
and $\big|g_{+}^{(i)}-g_{-}^{(i)}\big|<2\varepsilon$ for $i=0,\dots, n$. So for $b'-b$  so large that 
$$\sum_{i=0}^{j-1}\binom{j}{i}\frac{C_{j-i}}{(b'-b)^{j-i}}\ <\  \delta/2,$$
condition (ii) is satisfied. (See also Figure~\ref{fig:gplusgminus}.)
\end{proof} 

\begin{figure}[h]
\def\betafn{10*(0.05*tanh(3+6*(x-1)) + 0.05)}
\def\gplusfn{0.28*8^(-x)+(x-0.2)^(4)+0.3}
\def\gminusfn{(\gplusfn)+0.25*x^3-0.4*x^2-0.3*x+0.3*sin(500*x)+0.3}
\begin{tikzpicture}
  \begin{axis} [axis lines=center, axis y line=none, xmin=-0.5, xmax=1.3, ymin = -0.02, ymax = 1.25, width=0.8\textwidth, height = 0.5\textwidth, xlabel={$t$}, xtick={-0.35,0,1}, xticklabels={\strut $a$, \strut $b$, \strut $b'$}, ytick=1, yticklabels={$1$},   legend style={font=\small},
 legend cell align=left, legend style={at={(1.1,0.75)},anchor=west}]
  
    \addplot [domain=-0.45:1.1, smooth, dotted, very thick] { \gplusfn};
    \addplot [domain=-0.45:1.1, smooth, dashed, very thick] {\gminusfn};
 \addlegendentry{$g_+$};
    \addlegendentry{$g_-$};

  \addplot [domain=-0.45:1.1, smooth, very thick] { (1-(\betafn))*(\gminusfn)+(\betafn)*(\gplusfn) };

    \addlegendentry{$g$};
\draw [dashed] (0,-0.02) -- (0,1.25);
\draw [dashed] (-0.35,-0.02) -- (-0.35,1.25);
\draw [dashed] (1,-0.02) -- (1,1.25);

  \end{axis}
\end{tikzpicture}
\caption{}\label{fig:gplusgminus}
\end{figure}

\begin{prop}\label{apinf} Suppose $f\in \Cc^{<\infty}$ and $\varepsilon\in\Cc$,  $\varepsilon>_{\ex} 0$. Then there exists $g\in \Cc^{\infty}$ such that $|(f-g)^{(n)}| <_{\ex} \varepsilon$ for all $n$.
\end{prop} 
\begin{proof}  Represent $f$ and $\varepsilon$ by continuous functions $ [a_0,+\infty)\to \R$ ($a_0\in \R$), also denoted by
$f$ and $\varepsilon$, such that $\varepsilon>0$ on $[a_0,+\infty)$. Next, take a strictly increasing sequence $(a_n)$ of real numbers starting with the already given $a_0$, such that $a_n\to \infty$ as $n\to \infty$, and $f$  is of class $\Cc^n$ on
$[a_n,+\infty)$, for each $n$.  Then Corollary~\ref{cor:WAP dim 1, 1} gives for each $n$ a function $g_n\in \Cc_{a_0}^\infty$ such that
$|(f-g_n)^{(j)}|< \varepsilon/2$ on $[a_n,+\infty)$ for~$j=0,\dots,n$.  All this remains true when increasing each  $a_n$ while keeping $a_0$ fixed and maintaining that $(a_n)$ is strictly increasing. Now use the lemma above to construct $g$ as required:  first glue
 $g_0$ and $g_1$ and increase the $a_n$ for $n\ge 1$, then glue the resulting function with $g_2$ and increase the $a_n$ for $n\ge 2$, and so on, and arrange the product of the $(1+\delta)$-factors to be $<2$.
\end{proof} 

\noindent
Now  Corollary~\ref{cor:WAP dim 1, 2} (for $r=\infty$) and  Proposition~\ref{apinf}    yield:

\begin{cor}\label{apomega} For any germs $f\in \Cc^{<\infty}$ and  $\varepsilon\in \Cc$ with $\varepsilon>_{\ex} 0$, there exists a germ $g\in \Cc^{\omega}$ such that $|(f-g)^{(n)}| <_{\ex} \varepsilon$ for all $n$.
\end{cor}

\noindent
In the next section we apply Corollary~\ref{apomega} to bounded Hardy fields.

\section{Bounded Hardy Fields}

\noindent
As in   \cite[Section~5]{ADH5}  a set $H\subseteq \mathcal C$ is called
{\it bounded}\/ if for some $\phi\in\mathcal C$ we have~$h\leq \phi$ for all~${h\in H}$, and {\it unbounded}\/ otherwise. 
Every countable subset of $\mathcal C$ is bounded, cf.~\cite[remarks after Lemma~5.17]{ADH5}.
As a consequence, the union of  countably many bounded subsets of $\mathcal C$ is also bounded.

In this section we first establish a few general facts about the class of bounded Hardy fields, notably an ``analytification'' result (Corollary~\ref{cor:analytify}) 
needed for the proof of Proposition~\ref{prop:smanpc}. We then focus on the subclass of Hardy fields with
countable cofinality, and show   it to be closed under natural differential-algebraic Hardy field extensions (Theorem~\ref{thm:cf}). 
Some auxiliary results from this subsection (e.g., \ref{lem:cf(K)}, \ref{lem:cof 1}, \ref{lem:H(x)}) are also used later, notably in Section~\ref{sec:ctbl cf}, where we continue our study of   Hardy fields of countable cofinality.

\subsection*{Observations on bounded Hardy fields}
{\it In the rest of this section $H$ is a Hardy field.}\/
If~$H$ is bounded, then there is a~$\phi\in\mathcal C$ with~$\phi>_{\ex}0$ and $g\prec\phi$ for all~$g\in H$, so
$\varepsilon:=1/\phi\in\mathcal C^\times$ satisfies~$\varepsilon>_{\ex} 0$ and~$\varepsilon\prec h$ for all $h\in H^\times$. 
A germ~$y\in\Cc$ is said to be {\it $H$-hardian}\/ if it lies in a Hardy field extension  of $H$, and {\it hardian}\/ if
it lies in some Hardy field (equivalently, it is $\Q$-hardian). For  $r\in\{\infty,\omega\}$, if~$H\subseteq\Cc^r$ and $y\in\Cc^r$ is $H$-hardian, then $H\langle y\rangle\subseteq\Cc^r$; see~\cite[Section~4]{ADH5}.
By \cite[Lem\-mas~5.18, 5.19]{ADH5} we have:

\begin{lemma}\label{lem:5.4.19} If $H$ is bounded, then any $\d$-algebraic Hardy field extension of $H$ is bounded, and
for any $H$-hardian $f\in \Cc^{<\infty}$, the Hardy field $H\<f\>$ is bounded. 
\end{lemma}

\begin{cor}\label{lem:analytify, 1}
If $H$ is bounded  and $F$ is a Hardy field extension of $H$  and $\d$-algebraic over $H\<S\>$ for some countable $S\subseteq F$, 
then~$F$ is bounded.
\end{cor}


 
\begin{lemma}\label{lem:analytify, 2}
Let   $f,g\in\Cc^{<\infty}$  be such that $f$ is $H$-hardian, $\d$-transcendental over~$H$, and
$(f-g)^{(n)}\prec h$ for all $h\in H\langle f \rangle^{\times}$ and all~$n$.
Then $g$ is $H$-hardian, and there is a unique isomorphism $H\langle f\rangle\to H\langle g\rangle$
of Hardy fields   over~$H$   sending $f$ to $g$.
\end{lemma}
\begin{proof}
Let $P\in H\{Y\}^{\neq}$,  $r:=\order P$, so  $P(f)\in H\langle f\rangle^\times$. It suffices to show that then $P(f) \sim P(g)$. By Taylor expansion [ADH, p.~210], with $\i$ ranging over $\N^{1+r}$:
$$P(g) - P(f) =  \sum_{\abs{\i}\geq 1} P_{(\i)}(f)(g-f)^{\i}\quad\text{where $P_{(\i)}=\frac{P^{(\i)}}{\i!}\in H\{Y\}$.}$$
If $\abs{\i}\geq 1$, then $(g-f)^{\i}\prec h$ for all $h\in H\langle f\rangle^{\times}$, and hence
$P_{(\i)}(f)(g-f)^{\i} \prec P(f)$. Thus~$P(g)-P(f)\prec P(f)$ as required.
\end{proof}

\noindent
With Corollary~\ref{apomega} we now obtain analytic ``copies'' of certain $H$-hardian germs:

\begin{cor}\label{cor:analytify}
Suppose $H$ is bounded  and $f$ in a Hardy field extension of $H$ is $\d$-transcendental over $H$.
Then there is an $H$-hardian $g\in\Cc^\omega$
and an isomorphism~$H\langle f\rangle\to H\langle g\rangle$
of Hardy fields over $H$   sending $f$ to $g$.
\end{cor}
\begin{proof}
By Lemma~\ref{lem:5.4.19},   the Hardy field $H\langle f\rangle$ is bounded, so we can
take~$\varepsilon\in\Cc^\times$ with $\varepsilon>_{\ex} 0$ and~$\varepsilon \prec h$   for all $h\in H\langle f\rangle^\times$.
Corollary~\ref{apomega} yields a $g\in\Cc^\omega$ such that~$\abs{(f-g)^{(n)}}\leq \varepsilon$
for all $n$, and so it remains to appeal to  Lemma~\ref{lem:analytify, 2}.
\end{proof}

\noindent
Recall from [ADH, 10.6] that an $H$-field $L$ is said to be {\it Liouville closed}\/ if it is real closed and for all~$f,g\in L$ there exists $y\in L^\times$ with $y'+fy=g$.
If $H\supseteq\R$, then our Hardy field~$H$ is an $H$-field, and $H$ has  a smallest Liouville closed Hardy field extension~$\Li(H)$. (See \cite[Section~4]{ADH5}.)
We can now also strengthen  \cite[Theorem~5.1]{ADHfgh}:

\begin{cor}\label{cor:5.1ana}
Suppose $H\supseteq\R$ is Liouville closed, and $\phi\in\Cc$, $\phi>_{\ex}H$. Then there is
an $H$-hardian $z\in\Cc^\omega$ with $z>_{\ex}\phi$.
\end{cor}
\begin{proof}
By  \cite[Theorem~5.1]{ADHfgh} we have an $H$-hardian $y\in\Cc^\infty$ with~$y>_{\ex}\phi +1$.
Then~$y$ is $\d$-transcendental over $H$ and $H\langle y \rangle$ is bounded, by~\cite[Lemma~5.1]{ADH5} and
Lemma~\ref{lem:5.4.19}. 
This yields $\varepsilon\in\Cc$ such that $\varepsilon>_{\ex} 0$ and~$\varepsilon \prec h$   for all $h\in H\langle y\rangle^\times$.
Now Corollary~\ref{apomega}  gives~$z\in\Cc^\omega$ with~$\abs{y^{(n)}-z^{(n)}}<_{\ex}\varepsilon$ for all $n$.
Then $z$ is  $H$-hardian by Lemma~\ref{lem:analytify, 2}, and~$z=y+(z-y)>_{\ex}\phi$.
\end{proof}

\noindent
Thus maximal Hardy fields,  maximal $\Cc^\infty$-Hardy fields, and maximal $\Cc^{\omega}$-Hardy fields are unbounded; see~also~\cite[Corollary~5.23 and succeeding remarks]{ADH5}.) The {\it cofinality}\/ of a totally ordered set $S$ (that is,
the smallest ordinal isomorphic to a  cofinal subset of~$S$) is denoted by~$\cf(S)$; likewise~$\ci(S)$ denotes the {\it coinitiality}\/ of~$S$; cf.~[ADH, 2.1].
As~\cite[Theorem~5.1]{ADHfgh} gave rise to~\cite[Co\-rol\-la\-ry~5.2]{ADHfgh}, so Corollary~\ref{cor:5.1ana} yields:

\begin{cor}\label{corsjoan} If $H$ is a maximal analytic Hardy field, then $\cf(H) > \omega$, and thus
$$\ci(H)\ =\ \cf(H^{<a})\ =\ \ci(H^{>a})\ >\ \omega\  \text{ for all $a\in H$.}$$
Likewise with ``smooth'' in place of ``analytic''.
\end{cor}

\noindent
Call a subset $F$ of $\Cc$ {\bf cofinal} if for each $\phi\in\mathcal C$ there exists
$f\in F$ with~$\phi\leq f$. 
If~$F_1,F_2\subseteq\Cc$ and for all $f_1\in F_1$ there is an $f_2\in F_2$ with~$f_1\leq f_2$,
and $F_1$ is cofinal, then $F_2$ is cofinal.
Clearly each cofinal subset of $\Cc$ is unbounded.
The following strengthens \cite[Theorem~7]{S}:
 
\begin{cor}\label{CHcof}
Assume the Continuum Hypothesis {\em CH}: $2^{\aleph_0}=\aleph_1$. Then
there is a cofinal analytic Hardy field.
\end{cor}
\begin{proof}
Put $\mathfrak c:=2^{\aleph_0}$, 
and let $\alpha$, $\alpha'$, $\beta$  range over ordinals~$<\mathfrak c$.
Choose an enumeration~$(\phi_\alpha)_{\alpha<\mathfrak c}$ of $\Cc$.
Suppose   $\big((H_\alpha,h_\alpha)\big)_{\alpha<\beta}$ is a family of bounded analytic Hardy fields  $H_\alpha$, each with an element $h_\alpha\in H_\alpha$,   such that
\begin{equation}\label{eq:cof}
\alpha<\alpha'<\beta\ \Rightarrow\ H_{\alpha}\subseteq H_{\alpha'}\qquad\text{and}\qquad
\alpha<\beta\ \Rightarrow\ \phi_\alpha <_{\ex} h_\alpha.
\end{equation}
Then $H:=\bigcup_{\alpha<\beta}H_\alpha$ is an analytic Hardy field, and $H$ is bounded, as the
union of countably many bounded subsets of $\mathcal C$.
By Lemma~\ref{lem:5.4.19},  $H^*:=\Li\!\big(H(\R)\big)$ is also bounded.
Take $\phi\in\Cc$ with $\phi>_{\ex} H^*$ and $\phi\geq\phi_\beta$.
  Corollary~\ref{cor:5.1ana} yields an $H^*$-hardian   $h_\beta\in\Cc^\omega$ with  
$h_\beta>_{\ex}\phi$. Then the analytic Hardy field~$H_\beta:=H^*\langle h_\beta\rangle$ is bounded by Lemma~\ref{lem:5.4.19}, contains $H_\alpha$ for all~$\alpha<\beta$,  
and $\phi_\beta<_{\ex}h_\beta$. 

Now  transfinite recursion yields a family $\big((H_\alpha,h_\alpha)\big)_{\alpha<\mathfrak c}$
where   $H_\alpha$ is a bounded analytic   Hardy field and $h_\alpha\in H_\alpha$   such that
\eqref{eq:cof} holds with $\mathfrak c$ in place of $\beta$.
Then~$\bigcup_{\alpha<\mathfrak c} H_\alpha$ is a cofinal analytic Hardy field.
\end{proof}

\noindent
See Corollary~\ref{CHcof strengthened} below for a strengthening of  Corollary~\ref{CHcof}.

\begin{remark} Vera Fischer suggested replacing
CH  in Corollary~\ref{CHcof} by $\mathfrak{b}=\mathfrak{d}$, which is strictly weaker than CH (provided of course that our base theory ZFC is consistent). 
 Here 
$\mathfrak{b}$ and $\mathfrak{d}$ are so-called {\em cardinal characteristics of the continuum}.  See~\cite[2.1, 2.2]{Blass} for their definitions, and~\cite[2.4]{Blass} for the inequalities $\aleph_1\leq\mathfrak{b}\leq\mathfrak{d}\leq\mathfrak{c}$. 
\textit{Martin's Axiom}\/ (MA) implies $\mathfrak{b}=\mathfrak{d}=\mathfrak{c}$, see \cite[6.8, 6.9]{Blass} and~\cite[Corollary~8]{Rudin}.  If ZFC is consistent, then MA is strictly weaker than CH by~\cite{ST}. It is easy to check from their definitions that~$\mathfrak{b}$ is the least cardinality of an unbounded subset of $\C$, and $\mathfrak{d}$ is the least cardinality of a cofinal subset of~$\C$.
Replacing in the proof above~$\mathfrak{c}$ by~$\mathfrak{d}$ and taking
$(\phi_{\alpha})_{\alpha< \mathfrak{d}}$ to enumerate a cofinal subset of $\C$, the proof does indeed go through with
$\mathfrak{b}=\mathfrak{d}$ instead of CH. 
\end{remark}

\subsection*{Hardy fields of countable cofinality}
Hardy fields of countable cofinality are bounded.
For later use we study here such Hardy fields in more detail. Given a valued differential field $K$, let  $C$ denote  its constant field 
and $\Gamma$  its value group. 

\begin{lemma}\label{lem:cf(K)}
Let $K$ be a pre-$H$-field with $\Gamma\neq\{0\}$. Then $\operatorname{cf}(K)=\operatorname{cf}(\Gamma)$.
\end{lemma}
\begin{proof}
Apply [ADH,  2.1.4]   to the increasing surjection~$f\mapsto -vf\colon K^>\to\Gamma$. 
\end{proof}

 
\noindent
In the next two lemmas $\Gamma$ is an ordered abelian group. Recall from [ADH, 2.4]  that  the archimedean class of $\alpha\in\Gamma$ is
$$[\alpha]\ :=\ \big\{ \beta\in\Gamma:\  \text{$\abs{\alpha}\le n\abs{\beta}$ and  $\abs{\beta}\le n\abs{\alpha}$ for some $n\ge 1$} \big\}.$$
We write $[\alpha]_\Gamma$ instead of $[\alpha]$ if we want to stress the dependence on $\Gamma$. We equip~$[\Gamma]=\big\{[\alpha]:\alpha\in\Gamma\big\}$ with the ordering satisfying
$[\alpha] \le [\beta]$ iff $\abs{\alpha}\leq n\abs{\beta}$ for some $n \ge 1$.
 If~$\Delta$ is an ordered subgroup of $\Gamma$, then for each $\delta\in\Delta$ we have $[\delta]_\Delta=[\delta]_\Gamma\cap\Delta$, 
 and we have an embedding $[\delta]_\Delta\mapsto[\delta]_\Gamma\colon [\Delta] \to [\Gamma]$ of ordered sets 
via which we identify~$[\Delta]$ with an ordered subset of $[\Gamma]$.

\begin{lemma}\label{lem:cof oags, 1}
Suppose $\Gamma\neq\{0\}$. If
$[\Gamma]$ has no largest element, then $\operatorname{cf}(\Gamma)=\operatorname{cf}([\Gamma])$;
otherwise  $\operatorname{cf}(\Gamma)=\omega$. 
\end{lemma}
\begin{proof}
If $[\Gamma]$ has no largest element, then 
[ADH,  2.1.4] applied to the increasing surjection $\gamma\mapsto[\gamma]\colon\Gamma^{\geq}\to[\Gamma]$
yields  $\operatorname{cf}(\Gamma)=\operatorname{cf}([\Gamma])$.
If $\gamma\in\Gamma^{>}$ is such that $[\gamma]$ is the largest element of~$[\Gamma]$,  then
$\N\gamma$ is cofinal in $\Gamma$.
\end{proof}

\noindent
Let $G$ be an abelian group, with divisible hull $\Q G=\Q\otimes_{\Z}G$. Then $\operatorname{rank}_{\Q} G:=\dim_{\Q}\Q G$ (a cardinal) is the {\it rational rank}\/ of $G$.
(NB: in [ADH, 1.7] we defined the rational rank of $G$ to be~$\infty$ if the $\Q$-linear space $\Q G$ is not finitely generated.)

\begin{lemma}\label{lem:cof oags, 2}
Let $\Delta\ne\{0\}$ be an ordered subgroup of $\Gamma$ with~${\operatorname{rank}_{\Q}(\Gamma/\Delta)\le \aleph_0}$. Then $\operatorname{cf}(\Gamma)\leq\operatorname{cf}(\Delta)$.
\end{lemma}
\begin{proof}
 By Lemma~\ref{lem:cof oags, 1} we may assume that $[\Gamma]$
has no maximum. Let $S$ be a well-ordered cofinal subset of $[\Delta]$ of order type $\operatorname{cf}([\Delta])$, so $\abs{S}=\operatorname{cf}([\Delta])$. Then~$\tilde{S}:=S\cup\big( [\Gamma]\setminus[\Delta] \big)$ is cofinal in
$[\Gamma]$, so
$\operatorname{cf}(\Gamma)=\operatorname{cf}([\Gamma])=\operatorname{cf}(\tilde S)\leq\abs{\tilde{S}}$
by Lemma~\ref{lem:cof oags, 1} and [ADH,  2.1.2].
Since~$[\Gamma]\setminus [\Delta]$ is countable  by [ADH, 2.3.9], we have~$\abs{\tilde{S}}
\leq\max\!\big\{\abs{S},\omega\big\}=\max\!\big\{\!\operatorname{cf}([\Delta]),\omega\big\}$. Now apply Lemma~\ref{lem:cof oags, 1}   to $\Delta$ in place of~$\Gamma$.
\end{proof}

\noindent
A valued differential field $K$ has {\it small derivation}\/ if for all $f\in K$: $f\prec 1\Rightarrow f'\prec 1$, and
{\it very small derivation}\/ if for all $f\in K$: $f\preceq 1\Rightarrow f'\prec 1$.
For  more on this, see~[ADH, 4.4] and~\cite[Section~13]{ADH4}, respectively. 

\begin{lemma}\label{cofoaglem}
Let $K\subseteq L$ be an extension of pre-$H$-fields where 
$\operatorname{rank}_{\Q}(\Gamma_L/\Gamma)$ is countable, and suppose
$\Gamma\neq\{0\}$ or $L$ has very small derivation and archimedean residue field.
Then~$\operatorname{cf}(L)\leq\operatorname{cf}(K)$.
\end{lemma}
\begin{proof}
If $\Gamma\neq\{0\}$,   then by Lemmas~\ref{lem:cf(K)} and~\ref{lem:cof oags, 2}
we have $\operatorname{cf}(L)=\operatorname{cf}(\Gamma_L)\leq \operatorname{cf}(\Gamma)=\operatorname{cf}(K)$.
Suppose $\Gamma=\{0\}$, so $L$ has very small derivation and archimedean residue field. Then  $K$ is archimedean, so
$\operatorname{cf}(K)=\omega$,
and $\Gamma_L$ is countable, hence~$\operatorname{cf}(\Gamma_L)\leq\omega$.
Therefore, if $\Gamma_L\neq\{0\}$, then~$\operatorname{cf}(L)=\operatorname{cf}(\Gamma_L)\leq\omega=\operatorname{cf}(K)$ by Lemma~\ref{lem:cf(K)} applied to~$L$ in place of $K$, and if $\Gamma_L=\{0\}$, then $\operatorname{cf}(L)=\omega=\operatorname{cf}(K)$.
\end{proof}

\noindent
By [ADH, 3.1.10] the hypothesis on $\operatorname{rank}_{\Q}(\Gamma_L/\Gamma)$ in Lemma~\ref{cofoaglem} is satisfied if 
$\operatorname{trdeg}(L|K)$ is countable. Hence this lemma yields:
  
\begin{cor}
If   $F$ is a Hardy field extension of $H$ such that
 $\operatorname{trdeg}(F|H)$ is countable, then $\operatorname{cf}(F)\leq\operatorname{cf}(H)$.
 Hence if $\operatorname{trdeg}(H|C_H)$ is countable, then~$\operatorname{cf}(H)=\omega$ \textup{(}and so $H$ is bounded\textup{)}.
\end{cor}

\noindent
In \cite[Corollary~3.13]{ADHfgh} we showed that if~$H\supseteq\R$ and~$H^{>\R}$ has countable coinitiality, 
and $H^{\operatorname{da}}$ is the
$\d$-closure of~$H$ in  a maximal Hardy field extension  of $H$, 
then   $(H^{\operatorname{da}})^{>\R}$ also has countable coinitiality. The 
property of~$H$ having countable cofinality  is equally robust:

\begin{theorem}\label{thm:cf}
Let $E$ be a differentially algebraic Hardy field extension of $H$ such that~$\exp\!\big(E(x)\big)\subseteq E(x)$. Then~$\operatorname{cf}(E) \leq \operatorname{cf}(H)$,
with equality if $\exp\!\big(H(x)\big)\subseteq H(x)$. \end{theorem}

\noindent
Here is an immediate consequence:

\begin{cor}\label{cor:cf lc}
If $H$ has countable cofinality, then so does $\Li\!\big(H(\R)\big)$ as well as the $\d$-closure of $H$ in any  maximal Hardy field extension  of $H$. 
\end{cor}

\noindent
We precede the proof of Theorem~\ref{thm:cf} by a few lemmas. In Lemmas~\ref{lem:cof 1} and~\ref{lem:H(x)} we let $K$ be a
pre-$\d$-valued field of $H$-type with asymptotic couple $(\Gamma,\psi)$ where~${\Gamma\ne\{0\}}$.
By [ADH, 10.3.1], $K$ has a $\d$-valued extension
$\operatorname{dv}(K)$ of $H$-type, the {\it $\d$-valued hull}\/ of $K$, such that any embedding of $K$ into any $\d$-valued field $L$ of
$H$-type extends uniquely to an embedding   $\operatorname{dv}(K) \to L$.
  
\begin{lemma}\label{lem:cof 1}
$\Gamma$ is cofinal in  $\Gamma_{\operatorname{dv}(K)}$.
\end{lemma}
\begin{proof}
This is clear if $\Gamma=\Gamma_{\operatorname{dv}(K)}$.
Otherwise
$\Gamma_{\operatorname{dv}(K)} = \Gamma+\Z\alpha$ where~$0<n\alpha<\Gamma^>$ for all $n\geq 1$, by [ADH, 10.3.2], so
$\Gamma$ is cofinal in $\Gamma_{\operatorname{dv}(K)}$.
\end{proof}


\noindent
For the proof of the next lemma we recall that  ${\abs{\Gamma\setminus (\Gamma^{\ne})'}\le 1}$, and 
for $\beta\in\Gamma$ we have~$\beta\in \Gamma\setminus (\Gamma^{\ne})'$ iff  $\beta=\max\Psi$  or $\Psi<\beta<(\Gamma^>)'$, by [ADH, 9.2.1, 9.2.16]. We say that $K$ has {\it asymptotic integration}\/ if $\Gamma=(\Gamma^{\ne})'$ and
$K$ is {\it grounded}\/ if $\Psi$ has a largest element.
A {\it gap}\/ in~$K$ is a~$\beta\in\Gamma$ such that~$\Psi<\beta<(\Gamma^>)'$. 
(So there is at most one gap in~$K$, and   $K$ has asymptotic integration or is grounded  iff  it has no gap.) For all this, see [ADH, 9.1, 9.2].

\begin{lemma}\label{lem:H(x)}
Let $s\in K$ and~$y'=s$, $y$ in a pre-$\d$-valued extension of $K$ of $H$-type. Then $\Gamma$ is cofinal~in~$\Gamma_{K(y)}$.
\end{lemma}
\begin{proof} 
Set $L:=K(y)$ and $M:=\operatorname{dv}(L)$.
Lemma~\ref{lem:cof 1} allows us to replace $K$ by its $\d$-valued hull inside 
$M$
to arrange that~$K$ is $\d$-valued.
Using [ADH, 10.5.15 and remark preceding 4.6.16] we replace $K$ by $K(C_M)$ to arrange also $L$ to be $\d$-valued with $C=C_L$.
Finally, replacing $K$ by its algebraic closure inside an algebraic closure of $L$
we arrange $K$ to be algebraically closed. 
We may assume~$y\notin K$, so $y$ is transcendental over $K$. Then
$$S\ :=\ \big\{ v(s-a') : a\in K\big\}\ \subseteq\ \Gamma.$$
Assume for now that $S$ has  a maximum $\beta$. Then $\beta\notin (\Gamma^{\neq})'$ by [ADH, 10.2.5(i)], so~$\beta=\max\Psi$ or $\beta$ is a gap in $K$. If $\beta=\max\Psi$, then $\Gamma_{L}=\Gamma+\Z\alpha$ with~$\Gamma^{<} < n\alpha < 0$ for all $n\geq 1$, so~$\Gamma$ is cofinal in $\Gamma_{L}$. 
 Suppose $\beta$ is a gap in $K$. Take~$a\in K$ with~$\beta=v(s-a')$ and set $z:= y-a$, so $z'=s-a'$. We arrange $z\nasymp 1$
 by replacing~$a$ with $a+c$ for suitable $c\in C_L=C$.
If~$z\prec 1$, then~[ADH, 10.2.1 and its proof] gives $\Gamma_{L}=\Gamma+\Z\alpha$ with $0 < n\alpha < \Gamma^{>}$ for all $n\ge 1$, so  $\Gamma$ is cofinal in~$\Gamma_{L}$. 
 If~$z\succ 1$, then [ADH, 10.2.2 and its proof] gives likewise that
 $\Gamma$ is cofinal in~$\Gamma_{L}$.

 If $S$ does not have a largest element, 
then $L$ is an immediate extension of $K$: this holds by [ADH, 10.2.6]
if~$S<(\Gamma^>)'$; otherwise take~$a\in K$ with $v(s-a')\in(\Gamma^>)'$ and $y-a\nasymp 1$, and apply [ADH, 10.2.4 and 10.2.5(iii)] to $s-a'$, $y-a$ in place of~$s$,~$y$, respectively.
\end{proof}

\noindent
Lemmas~\ref{lem:cf(K)},~\ref{lem:H(x)}, and  \cite[Proposition 4.2(iv)]{ADH5} yield:

\begin{lemma}
If $H\subseteq\R$, then $x^\N$ is cofinal in $H(x)$, and if $H\not\subseteq\R$, then $H$ is cofinal in $H(x)$.
Hence $\operatorname{cf}(H)=\operatorname{cf}(H(x))$.
\end{lemma}

\noindent
We can now give the proof of Theorem~\ref{thm:cf}.
First, replacing $E$, $H$ by $E(x)$, $H(x)$, respectively, and using the last lemma, we arrange $x\in H$.
Let $S$ be a well-ordered cofinal subset of $H$ of order type $\operatorname{cf}(H)$. For $\phi\in \Cc$, define   $\exp_n(\phi)\in\Cc$ by recursion: $\exp_0(\phi):=\phi$ and
$\exp_{n+1}(\phi):=\ex^{\exp_n(\phi)}$. 
Then $\abs{S}=\operatorname{cf}(H)$,
and $\tilde S:=\bigcup_n \exp_n(S)$ is a cofinal subset of $E$, by~\cite[Lem\-ma~5.1]{ADH5}.
Thus $\operatorname{cf}(E)=\operatorname{cf}(\tilde S)\leq\abs{\tilde S}=\abs{S}=\operatorname{cf}(H)$
 as claimed.
If $\exp(H)\subseteq H$, then $\tilde S\subseteq H$, hence~$\operatorname{cf}(E)=\operatorname{cf}(H)$.
 \qed
 
\medskip
\noindent
The following corollary of Lemma~\ref{lem:cof 1} is not used later.
If $K$ is a pre-$H$-field, then by~[ADH, 10.5.13] there is a unique field ordering on  $\operatorname{dv}(K)$ making it a pre-$H$-field
extension of $K$. Equipped with this ordering,  $\operatorname{dv}(K)$  is an
$H$-field,    the {\it $H$-field hull}\/ of $K$, which embeds uniquely over $K$ into any $H$-field extension of~$K$;
notation:~$H(K)$ (not to be confused with the Hardy field $H(\R)$ generated over the Hardy field $H$ by $\R$).

\begin{cor}
$H$ is cofinal in $H(\R)$.
\end{cor}
\begin{proof}
This is clear if   $H\subseteq\R$; assume $H\not\subseteq\R$.
Let $E$ be the $H$-field hull of~$H$, taken as an $H$-subfield of the Hardy field extension $H(\R)$ of~$H$.
Then $H$ is cofinal in $E$ (Lemma~\ref{lem:cof 1}), so replacing $H$ by $E$ we arrange that $H$ is an $H$-field. Now use that
$\Gamma_{H(\R)}=\Gamma_H\neq\{0\}$ by 
[ADH, 10.5.15 and remark preceding 4.6.16].
\end{proof}

\noindent
For use in Section~\ref{sec:ctbl cf} we   include the following cofinality result, which is immediate from Lemma~\ref{lem:cof oags, 1} and [ADH, 10.4.5(i)]:

\begin{lemma}\label{lem:10.4.5}
Let $K$ be a $\d$-valued field of $H$-type with divisible asymptotic couple~$(\Gamma,\psi)$, $\Gamma\neq\{0\}$,
and let $s\in K$ be such that 
$$S:=\big\{ v(s-a^\dagger): a\in K^\times\big\} < (\Gamma^>)'$$
and $S$ has no largest element. Let $f$ be an element of an $H$-asymptotic field extension of $K$,
transcendental over $K$, with $f^\dagger = s$. Then $[\Gamma]=\big[\Gamma_{K(f)}\big]$, so $\Gamma$ is cofinal in $\Gamma_{K(f)}$.
\end{lemma}

\section{Pseudoconvergence in Analytic Hardy Fields}\label{sec:pc}

\noindent
We   complement the material on pc-sequences
 from \cite[Sections~3,~4]{ADHfgh} by criteria for germs in $\Cc^{<\infty}$ to be pseudolimits of pc-sequences in Hardy fields, and then use this to show that each pc-sequence of countable length in an analytic Hardy field has an analytic pseudolimit. The main results to this effect are Propositions~\ref{propimmy1},~\ref{propimmy2}, and~\ref{prop:smanpc}.  
{\it In this section $H$ is a Hardy field.}\/ 

\subsection*{Revisiting pseudoconvergence in Hardy fields} 
Let $H\supseteq \R(x)$ be real closed with asymptotic integration and $(f_\rho)$ a pc-sequence in $H$ of $\d$-transcendental type over $H$ (cf.~[ADH, 4.4]) with pseudolimit $f$ in a Hardy field extension of $H$. 
Then
the valued field extension $H\langle f\rangle\supseteq H$ is immediate by [ADH, 11.4.7, 11.4.13].  In \cite{ADHfgh} we only considered pc-sequences of countable length, but here we do not assume~$(f_{\rho})$ has countable length (to be exploited in the proof of Theorem~\ref{thdense}). 
We begin by deriving a sufficient condition on $y\in \Cc^{<\infty}$ to be $H$-hardian with~${f_\rho\leadsto  y}$.
This will enable us to find such $y$ in $\mathcal{C}^{\omega}$. (Another possible use is to find such~$y$ with oscillating~$y-f$, so that $H\<y\>$ and~$H\<f\>$ are ``incompatible'' Hardy field extensions of $H$.)
To simplify notation, set $t:= x^{-1}$.  Let $\phi\in H^\times$. Recall from [ADH, 11.1] that   $\phi$ is said to be {\it active}\/ in~$H$ if $\phi\succeq h^\dagger$ for some $h\in H^\times$, $h\nasymp 1$.  Denoting by~$\der$ the derivation of the differential ring $\Cc^{<\infty}$, we let
  $(\Cc^{<\infty})^\phi$ be the   ring $\Cc^{<\infty}$ equipped with the derivation~$\derdelta:=\phi^{-1}\der$
  and $H^\phi$ be the ordered valued field~$H$ equipped with the restriction of $\derdelta$ to $H$; we
then have a ring isomorphism  $P\mapsto P^\phi\colon \Cc^{<\infty}\{Y\}\to(\Cc^{<\infty})^\phi\{Y\}$ with $P(y)=P^\phi(y)$ for
each $y\in\Cc^{<\infty}$.
We first observe:

\begin{lemma}\label{lem:immy}
Let $\phi$ be active in $H$, $0< \phi\prec 1$, let $\derdelta:=\phi^{-1}\der$ be the derivation of~$(\Cc^{<\infty})^\phi$, and
let  $z\in\Cc^{<\infty}$ satisfy $z^{(i)}\prec t^j$ for all $i$, $j$. Then  
$\derdelta^k(z)\prec 1$ for all $k$.
\end{lemma}
\begin{proof}
This  is clear for $k=0$.  Suppose $k\geq 1$.
The identity (3.1) in \cite{ADHfgh} gives~$\derdelta^k(z)=\phi^{-k}\sum_{j=1}^kR^k_j(-\phi^\dagger)z^{(j)}$
where  $R^k_j(Z)\in \Q\{Z\}$ for $j=1,\dots,k$.
This yields $\derdelta^k(z)\prec 1$  in view of $\phi^\dagger\preceq 1$ and $z^{(j)}\prec t^{2k}\prec \phi^k$
for $j=1,\dots,k$.
\end{proof}

\noindent
The proof of the next result uses various items from [ADH]: for $P_{+h}$, $P_{\times h}$ see 4.3,  for~$\ddeg_{\prec \fv}P$, see 6.6, for
$\ndeg_{\prec \fv}P$, see 11.1, and for $Z(H,f)$, see 11.4.  

\begin{lemma}\label{lem:6.7.12}
Let $y\in\Cc^{<\infty}$ be such that for all~$h\in H$, $\fm\in H^\times$ with $f-h\preceq\fm$
and all $n$ there is an active $\phi_0$ in $H$ such that for all active $\phi>0$ in $H$ with $\phi\preceq \phi_0$ we have 
$\derdelta^n\!\big(\frac{y-h}{\fm}\big)\preceq 1$ for~${\derdelta=\phi^{-1}\der}$.
Then~$y$ is $H$-hardian  and there is a Hardy field
isomorphism ${H\langle y\rangle \to H\langle f\rangle}$  over $H$
sending~$y$ to $f$.
\end{lemma}
\begin{proof} First, $Z(H,f)=\emptyset$ by [ADH, 11.4.13], since $(f_{\rho})$ is of $\d$-transcendental type over $H$ and $f_{\rho}\leadsto f$. 
It is enough to show that $Q(y) \sim Q(f)$ for all~${Q\in H\{Y\}\setminus H}$. Let $Q\in H\{Y\}\setminus H$. 
Then $Q\notin Z(H,f)$, so we have $h\in H$ and $\fv\in H^\times$ such that~${h-f\prec \fv}$ and
$\ndeg_{\prec \fv}Q_{+h}=0$. 
Since $H\<f\>\supseteq H$ is immediate, we have~${\fm\in H^\times}$ with~$f-h\asymp \fm$.
Let $r:=\order Q$ and choose active
$\phi_0$ in $H$ such that for all active $\phi>0$ in $H$ with $\phi\preceq \phi_0$ we have 
$\derdelta^j\!\big(\frac{y-h}{\fm}\big)\preceq 1$ for~${\derdelta=\phi^{-1}\der}$ and~$j=0,\dots,r$.
Now take
any $\fw\in H^\times$ with $\fm\prec\fw\prec\fv$. Then $\ndeg Q_{+h,\times\fw}=0$, so
we can choose an active $\phi>0$ in $H$ with~$\phi\preceq\phi_0$
and $\ddeg Q^\phi_{+h,\times\fw}=0$. Then~$\ddeg_{\prec\fw} Q^\phi_{+h}=0$, so
renaming $\fw$ as $\fv$ we arrange $\ddeg_{\prec\fv} Q^\phi_{+h}=0$.

Using~$(\phantom{-})^\circ$ as explained in \cite[Section~8]{ADH4}, we have the  Hardy field~$H\langle f\rangle^\circ$ and the  $H$-field isomorphism $h\mapsto h^\circ\colon H\langle f\rangle^\phi\to H\langle f\rangle^\circ$. 
Put~$u:=(y-h)/\fm\in \Cc^{<\infty}$. Then~$\ddeg_{\prec\fv^\circ} Q^{\phi\circ}_{+h^\circ}=0$ and  $(u^\circ)^{(j)}\preceq 1$ for~${j=0,\dots,r}$,
hence~$Q^{\phi\circ}(y^\circ)\sim  Q^{\phi\circ}(f^\circ)$ by \cite[Lemma~11.7]{ADH4} with 
$H^\circ$, $H\langle f\rangle^\circ$, $h^\circ$, $f^\circ$, $\fm^\circ$, $Q^{\phi\circ}$, $\fv^\circ$, $y^\circ$ in place of 
$H$, $\hat H$, $h$,  $\hat h$, $\fm$, $Q$,  $\fv$, $y$, respectively. 
This yields $Q(y) \sim Q(f)$, since~$Q^{\phi\circ}(g^\circ)=Q(g)^\circ$ for~$g\in\Cc^{<\infty}$.
\end{proof}
 

\begin{prop} \label{propimmy1} Suppose $0\in v(f-H)$ and  $y\in \Cc^{<\infty}$ is such that for all $\fm\in H^\times$ with $v\fm\in v(f-H)$ and all $i$, $j$, $k$ we have
$$y^{(i)}-f^{(i)}\ \prec\  \fm^j t^k\ \text{ in $\Cc^{<\infty}$.}$$
Then $y$ is $H$-hardian and there is an isomorphism $H\langle y\rangle \to H\langle f\rangle$ of Hardy fields over~$H$ sending $y$ to $f$ \textup{(}and thus $f_\rho\leadsto y$\textup{)}.
\end{prop} 

 \begin{proof}  
Let $\phi$ be active in $H$, $0<\phi\prec 1$, and $\derdelta=\phi^{-1}\der$ the derivation of $(\Cc^{<\infty})^\phi$.
Let also~$h\in H$ and $\fm\in H^\times$ with $f-h\preceq \fm$, and put $z:=\frac{y-f}{\fm}$.
By  Lemma~\ref{lem:6.7.12} 
it suffices to show that then
$\derdelta^k\big( \frac{y-h}{\fm} \big)\preceq 1$ for all $k$; equivalently,
 $\derdelta^k (z)\preceq  1$ for all $k$ (thanks to
$\frac{y-h}{\fm}-z = \frac{f-h}{\fm}\preceq 1$ and smallness of the derivation of $H\langle f\rangle^\phi$).

\claim[1]{Suppose $\fm\preceq 1$. Then $z^{(n)} \prec  \fm^jt^k$  for all $n$, $j$, $k$.}

\noindent
This holds for $n=0$ because $\fm z\prec \fm^{j+1}t^k$ for all $j$, $k$. 
Let $n\ge 1$ and assume inductively that $z^{(i)}\prec \fm^jt^k$ for $i=0,\dots, n-1$ and all $j$, $k$. 
Now $(\fm z)^{(n)}= y^{(n)}-f^{(n)}\prec \fm^jt^k$ for all $j$, $k$, and
$$(\fm z)^{(n)}\  =\ \ \fm^{(n)}z + \cdots + \fm z^{(n)}.$$ 
Since $\fm\preceq 1$, the smallness of the derivation of $H$ and the inductive assumption gives $\fm^{(n-i)}z^{(i)}\preceq z^{(i)}\prec\fm^jt^k$ for $i=0,\dots,n-1$ and all $j$, $k$, so
$\fm z^{(n)} \prec \fm^jt^k$  for all $j$, $k$, and thus $z^{(n)}\prec \fm^jt^k$  for all $j$, $k$.

\claim[2]{$\derdelta^k (z)\prec  1$ for all $k$.}

\noindent
If $\fm\preceq 1$, then this holds by Claim~1 and Lemma~\ref{lem:immy}. 
In general, take $h_1\in H$ with~$f-h_1\prec f-h$ and $f-h_1\preceq 1$, and then $\fm_1\in H^\times$ with $f-h_1\asymp\fm_1$. 
By the special case just proved with $h_1$, $\fm_1$ in place of $h$, $\fm$ we have~$\derdelta^k \big( \frac{y-f}{\fm_1} \big)\prec  1$ for all~$k$.
Now
$z = \big(\frac{y-f}{\fm_1}\big)\big(\frac{\fm_1}{\fm}\big)$ and $\frac{\fm_1}{\fm}  \prec 1$ (in $H$), so the claim follows using
the  Product Rule for the derivation   of $(\Cc^{<\infty})^\phi$ and smallness of the derivation of~$H^\phi$.
\end{proof} 

\noindent
Here is a more useful variant for the case $0\notin v(f-H)$:

\begin{prop} \label{propimmy2} Suppose $0\notin v(f-H)$, and $y\in  \Cc^{<\infty}$ is such that
$$y^{(i)}-f^{(i)}\ \prec\   t^k \text{ for all $i$, $k$.}$$
Then $y$ is $H$-hardian and there is an isomorphism $H\langle y\rangle \to H\langle f\rangle$ of Hardy fields over~$H$ sending $y$ to $f$.
\end{prop} 
\begin{proof} 
Let $\phi$, $h$, $\fm$, $z$ be as in the proof of Proposition~\ref{propimmy1}; as in that proof
it suffices to show that $\derdelta^k (z)\preceq  1$ for all $k$. Now $v(f-H)$ is downward closed, so~$v(f-H)<0$, which gives $1\prec f-h\preceq\fm$. Thus
$\fm\succ 1$, 
hence with $\frak{n}:= \fm^{-1}\in H^\times$  we have~$z=\frak{n}(y-f)$ and $\frak{n}\prec 1$, so in
view of  $z^{(i)}=\frak{n}^{(i)}(y-f)+\cdots + \frak{n}(y-f)^{(i)}$ we obtain
 $z^{(i)}\prec t^j$ for all $i$, $j$, and Lemma~\ref{lem:immy} then yields
   $\derdelta^k (z)\prec  1$ for each $k$.
\end{proof} 

\noindent
Multiplicative conjugation gives a reduction to Proposition~\ref{propimmy2}, except 
when $(f_{\rho})$ is a cauchy sequence, not just a pc-sequence. We shall exploit this several times.

\subsection*{Constructing analytic pseudolimits in Hardy field extensions}  
Let   $(f_{\rho})$ be a pc-sequence in $H$. Corollary~3.2  from \cite{ADHfgh}  says:
if $(f_\rho)$ has countable length, then~$(f_{\rho})$ pseudoconverges in some Hardy field extension of $H$. Using  Corollary~\ref{apomega} and Proposition~\ref{propimmy2}  we now deduce smooth and analytic versions of this key fact. 
We say that an $H$-field with  real closed constant field is {\it closed}\/ if  it has no proper $\d$-algebraic $H$-field extension with the same
constant field. (This is not how ``closed'' was introduced in \cite{ADHfgh}, but it is equivalent to it in view of [ADH, 16.0.3 and proof of 16.4.8].)
By \cite[Corollary~11.20]{ADH4}, every maximal Hardy field is a closed $H$-field; likewise with ``maximal smooth'' or ``maximal analytic'' in place of ``maximal''.
Every closed $H$-field is Liouville closed, by [ADH, 10.6.13, 10.6.14], and  every divergent pc-sequence in a closed $H$-field is of $\d$-transcendental type over it,  by [ADH, 11.4.8, 11.4.13].
 
\begin{prop}\label{prop:smanpc}
Suppose $H$ is an analytic Hardy field and
$(f_\rho)$   pseudoconverges in some Hardy
field extension of $H$. Suppose also that $H$ is bounded or $(f_\rho)$ does not have width~$\{\infty\}$ in the valued field $H$. Then $(f_\rho)$ pseudoconverges in an analytic Hardy field extension of $H$. 
Likewise with ``smooth'' in place of ``analytic''.
\end{prop}

\begin{proof}
 Assume $H$ is analytic; the smooth case goes the same way. As in \cite{ADHfgh} we can pass from $H$ to an extension of $H$ and reduce to the case that $H\supseteq \R$, $H$ is closed, and $(f_{\rho})$ has no pseudolimit in $H$. Take $f$ in a Hardy field extension of $H$ such that $f_{\rho}\leadsto f$. 
Then $f\notin H$, so $f$ is $\d$-transcendental over $H$. If $H$ is bounded, then
 Corollary~\ref{cor:analytify} yields an $H$-hardian $y\in \Cc^{\omega}$ with $f_{\rho}\leadsto y$.

Suppose $(f_\rho)$ does not have width $\{\infty\}$. Then take $h\in H^\times$ with ${v(hf-hf_{\rho})} < 0$ for all $\rho$. Take $\varepsilon\in \Cc$ such that $\varepsilon>_{\ex} 0$ and $\varepsilon \prec t^k$
for all $k$, for example, $\varepsilon= \ex^{-x}$. Now Corollary~\ref{apomega} gives $y\in \Cc^{\omega}$ such that $(hf)^{(i)}-y^{(i)}\prec t^k$ for all $k$. Then
$y$ is $H$-hardian and $hf_{\rho}\leadsto y$ by Proposition~\ref{propimmy2}, hence
$h^{-1}y\in \Cc^{\omega}$ is $H$-hardian and~$f_{\rho}\leadsto h^{-1}y$. 
\end{proof}

\begin{cor}\label{smanpc} If $H$ is an analytic Hardy field, then 
every pc-sequence in $H$ of countable length
pseudoconverges in an analytic Hardy field extension of $H$.
Likewise with ``smooth'' in place of ``analytic''.
\end{cor}
\begin{proof}
Suppose   $(f_\rho)$ has countable length. Then $(f_\rho)$ pseudoconverges
in a Hardy field extension of $H$, by  \cite[Corollary~3.2]{ADHfgh}. Moreover, if $(f_\rho)$
has width $\{\infty\}$, then~$\big(v(f-f_\rho)\big)$ is cofinal in $\Gamma_H$, so $\operatorname{cf}(H)=\operatorname{cf}(\Gamma_H)=\omega$ by Lemma~\ref{lem:cf(K)}, hence $H$ is bounded. Now use  Proposition~\ref{prop:smanpc}.
\end{proof}

\noindent
Arguing as in the proof of \cite[Corollary~4.8]{ADHfgh}, using Corollary~\ref{smanpc}
instead of \cite[Corol\-lary~3.2]{ADHfgh}, yields:

\begin{cor}\label{cormhcian}
If  $H$ is a maximal analytic or maximal smooth Hardy field, then $\ci(H^{>\R})>\omega$. 
\end{cor}

 \noindent
Recall from \cite[Section~6]{ADHfgh} that the $H$-couple $(\Gamma,\psi)$ of $H$ is said to be {\it countably spherically complete}\/ if 
in the valued abelian group $(\Gamma,\psi)$,
every 
pc-sequence   of
length $\omega$ in it pseudoconverges in it.
In view of \cite[Remark preceding Corollary~8.1]{ADHfgh}, Corollaries~\ref{corsjoan},~\ref{smanpc},~\ref{cormhcian} yield  a version of \cite[Corollary~8.1]{ADHfgh} for maximal analytic Hardy fields:

\begin{cor}\label{hcoan} If $H$ is a maximal analytic or maximal smooth Hardy field, then  its $H$-couple~$(\Gamma,\psi)$   is countably spherically complete and 
$$\cf(\Gamma^{<})\ =\ \ci(\Gamma^{>})\ >\ \omega, \qquad \ci(\Gamma)\ =\ \cf(\Gamma)\ >\ \omega.$$
\end{cor}

\section{Proofs of Theorems A and B} \label{s4}

\noindent
We begin with revisiting Case (\rm{b}) extensions, then prove Theorem A and use it to characterize the possible gaps in maximal analytic Hardy fields. We also determine the number of maximal analytic Hardy fields. Next we prove Theorem~B, and finish this section with two subsections on dense pairs
of closed $H$-fields. 

\subsection*{Case {(\rm{b})} extensions}
Let $H\supseteq \R$ be a Liouville closed Hardy field with $H$-couple~$(\Gamma,\psi)$ over $\R$. 
Suppose~$\beta$ in an $H$-couple $(\Gamma^*,\psi^*)$ over $\R$ extending $(\Gamma,\psi)$ falls under Case~(b), that is, with~$\big(\Gamma\<\beta\>, \psi_{\beta}\big)$   the $H$-couple over
$\R$ generated by~$\beta$ over $(\Gamma, \psi)$ in~$(\Gamma^*, \psi^*)$:

\begin{itemize}\item[(b)] We have a sequence $(\alpha_i)$ in 
$\Gamma$ and a sequence $(\beta_i)$ in $\Gamma^*$ that is $\R$-linearly independent
over $\Gamma$, such that $\beta_0=\beta-\alpha_0$ and $\beta_{i+1}=\beta_i^\dagger-\alpha_{i+1}$ for all $i$, and such that
$\Gamma\<\beta\>=\Gamma \oplus \bigoplus_{i=0}^\infty \R\beta_i$. 
\end{itemize}

\noindent
Unlike in key parts of \cite[Section~9]{ADHfgh}
we do not assume~$\beta$ is of countable type over~$\Gamma$, and this will be exploited in the proof of Theorem~\ref{thdense}. 
(Recall from \cite[Section~8]{ADHfgh}  that an element $\gamma$ of  an ordered vector space over $\R$ extending $\Gamma$
has {\it countable type over $\Gamma$}\/ if $\gamma\notin\Gamma$ and $\operatorname{cf}(\Gamma^{<\gamma}),\operatorname{ci}(\Gamma^{>\gamma})\le\omega$.)
By~\cite[Corollary~8.15]{ADHfgh}, $\beta$ also falls under Case~(b) with the same sequences $(\alpha_i)$, $(\beta_i)$ when
 $(\Gamma,\psi)$ and~$(\Gamma^*,\psi^*)$ are viewed as $H$-couples over $\Q$; see \cite[Section~8]{ADHfgh} for the relevant definitions. 

In the next proposition and its corollary we assume $y\in\Cc^{<\infty}$ is $H$-hardian,  $y>0$, and $vy$ realizes the same cut in~$\Gamma$ as~$\beta$. Then by~\cite[Remark~8.21]{ADHfgh} we have a unique isomorphism over $\Gamma$ of 
the $H$-couple over $\Q$ generated by $\Gamma\cup\{\beta\}$  in~$(\Gamma^*, \psi^*)$ with 
the $H$-couple of the Hardy field~$H\<y\>^{\operatorname{rc}}$ over $\Q$ sending $\beta$ to $vy$.
Moreover, if~$z\in\Cc^{<\infty}$ is also $H$-hardian with~$z>0$ and~$vz$ realizes the same cut in $\Gamma$ as $\beta$, then we have a unique Hardy field isomorphism $H\<y\>\to H\<z\>$ over $H$ sending $y$ to~$z$, by \cite[Proposition~8.20]{ADHfgh}. The problem here is to find
 such~$z$ in $\Cc^{\omega}$ (in which case~$H\<z\>$ is smooth, respectively analytic, if $H$ is). This can always be done, in view of Corollary~\ref{apomega} and the following:
 
 \begin{prop}\label{revb} There exists $\varepsilon\in \C$ such that $\varepsilon >_{\ex} 0$ and for all
$z\in \Cc^{<\infty}$, if~$(z-y)^{(i)}\prec \varepsilon$ for all $i$, then $z$ is $H$-hardian and
  $vz$ realizes the same cut in~$\Gamma$~as~$\beta$. 
 \end{prop}
\begin{proof}
Let the sequences $(\alpha_i)$, $(\beta_i)$ be as in (b). Take $f_i\in H^{>}$ with~$vf_i=\alpha_i$, and recursively we set $y_0:= y/f_0$, $y_{i+1}:=y_i^\dagger/f_{i+1}$. Then $vy_i$ realizes the same cut in $\Gamma$ as $\beta_i$, and
 $H\<y\>=H(y_0, y_1, y_2,\dots)$, 
 by \cite[proof of Proposition~8.20]{ADHfgh}.
 Next set~$K:=\R\< f_0, f_1, f_2,\dots\>$ and note that
 $K\<y\>=K(y_0, y_1, y_2,\dots)$, using that the right hand side contains all $y_n'$. Suppose~$z\in \Cc^{<\infty}$ is such that $(z-y)^{(i)}\prec f$ for all $i$  and all~$f\in K\langle y\rangle^{\times}$. Then~$z$ is $K$-hardian  and $\d$-transcendental over~$K$ by Lemma~\ref{lem:analytify, 2}, and the proof of that lemma also shows that $P(y)\sim P(z)$ for all
 $P\in K\{Y\}^{\ne}$. Thus we have elements~$z_i\in K\<z\>$ defined recursively by  $z_0:= z/f_0$, $z_{i+1}:=z_i^\dagger/f_{i+1}$, and then~${y_i\sim z_i}$ for all $i$. 
 For the Hausdorff field~$H_n:=H(y_0,\dots,y_n)$ we have~$v(H_n^\times)=\Gamma\oplus\Z vy_0\oplus\cdots\oplus\Z vy_n$
 by \cite[proof of Proposition~8.20]{ADHfgh},
 and for~$h\in H^\times$ and~$i_0,\dots,i_n\in\N$ we have~$hy_0^{i_0}\cdots y_n^{i_n} \sim h  z_0^{i_0}\cdots z_n^{i_n}$~(in~$\Cc$).
 Hence $z_0,\dots,z_n$ generate a Hausdorff field over $H$ with an isomorphism~$H_n\to H(z_0,\dots,z_n)$ over $H$ sending $y_i$ to $z_i$ for $i=0,\dots,n$.  These isomorphisms have therefore a common extension to an isomorphism
 $H\<y\>\to H\<z\>$ of Hardy fields over $H$. In particular, $z$ is $H$-hardian, and $vz$ realizes the same cut in $\Gamma$ as $\beta$.  Now by Lemma~\ref{lem:analytify, 1} and the remarks preceding it there exists $\varepsilon\in \C$ such that $\varepsilon >_{\ex} 0$ and
 $\varepsilon \prec f$ for all~$f\in K\<y\>^\times$, so any such $\varepsilon$ has the desired property. 
 \end{proof}

\noindent
Combining Corollary~\ref{apomega} with Proposition~\ref{revb} yields:

\begin{cor}\label{cor:revb}
There exists an $H$-hardian $z\in\Cc^\omega$ such that $z>0$ and $vz$ realizes the same cut in $\Gamma$ as $\beta$. 
\end{cor}

\noindent
We can now use  \cite[Theorem~9.2]{ADHfgh} to obtain an analytic strengthening of it:

\begin{cor} \label{bcon} 
Suppose $\beta$ is of countable type over $\Gamma$
and $\beta_i^\dagger < 0$ for all $i$.
Then for some $H$-hardian $z\in \Cc^{\omega}$:  $z > 0$ and $vz$ realizes the same cut in $\Gamma$ as $\beta$.
\end{cor} 
\begin{proof} \cite[Theorem~9.2]{ADHfgh} gives $H$-hardian $y>0$ such that $vy$ realizes the same cut in $\Gamma$ as $\beta$. Then Corollary~\ref{cor:revb}  gives a $z$ as required.
\end{proof} 

\subsection*{Proof of Theorem A} First  an analytic/smooth version of \cite[Lemma~9.1]{ADHfgh}:

\begin{lemma}\label{lemth01an}
 Let $H$ be a maximal analytic or maximal smooth Hardy field with $H$-couple $(\Gamma,\psi)$ over $\R$. Then no element in any
 $H$-couple over $\R$ extending $(\Gamma,\psi)$ has countable type over~$\Gamma$. 
\end{lemma}
\begin{proof}
By Corollary~\ref{hcoan}, $(\Gamma,\psi)$ is countably spherically complete, and both $\Gamma$ and~$\Gamma^<$   have
uncountably cofinality. By \cite[Lemma~8.11]{ADHfgh}, 
any element of any $H$-couple over $\R$ extending $(\Gamma,\psi)$ and of countable type over $\Gamma$
falls under Case~(b).  
Now argue as in the proof of \cite[Lemma~9.1]{ADHfgh}, using
 Corollary~\ref{bcon} in place of  \cite[Theorem~9.2]{ADHfgh}, that there are no such elements. 
\end{proof}

\noindent
Theorem~\ref{aneta} from the introduction  and its smooth version follow from 
Corollary~\ref{smanpc} and Lemma~\ref{lemth01an}, just as the main theorem in \cite{ADHfgh}  is derived in the beginning of \cite[Section~9]{ADHfgh} from the non-smooth analogues of that corollary and lemma. 
We now use this to characterize gaps in maximal analytic Hardy fields: Corollary~\ref{gapsmax} below. 

\subsection*{Characters of gaps in maximal Hardy fields}
Let~$S$ be an ordered set (as in~\cite{ADHfgh} this means {\em linearly ordered set}) and $C$  a cut in $S$, that is a downward closed subset of $S$. We define the {\bf character} of~$C$ (in $S$) to be the pair $(\alpha,\beta^*)$ where~$\alpha:=\cf(C)$ and $\beta^*$ is the set $\beta:=\ci(S\setminus C)$ equipped with the reversed ordering. We then also call
$C$ an {\bf $(\alpha,\beta^*)$-cut} (in $S$); see \cite[\S{}3.2]{Harzheim}.
The characters of the cuts~$\emptyset$ and $S$ in~$S$ are~$\big(0,\ci(S)^*\big)$ and $\big(\!\cf(S),0\big)$, respectively. Note that $S$ is $\eta_1$ iff no cut in~$S$ has character $(\alpha,\beta^*)$  with~${\alpha,\beta\leq\omega}$. 
A {\bf gap $A<B$ in $S$} is a pair $(A,B)$ of subsets of~$S$ such that $A<B$ and there is no $s\in S$ with $A<s<B$.
The character of such a gap~$A<B$ is defined to be the character~$(\alpha,\beta^*)$ of the cut $A^{\downarrow}=S\setminus B^\uparrow$ in $S$, and
then~$A<B$ is also called an $(\alpha,\beta^*)$-gap in $S$.

Let $G$ be an ordered abelian group.
If $v\colon G\to S_\infty$ is a surjective convex valuation on $G$ ([ADH, p.~99]) and $A<B$ is an $(\alpha,\beta^*)$-gap in $S$ where $\alpha,\beta\geq\omega$, then~$\big(v^{-1}(A)\cap G^<, v^{-1}(B)\cap G^{<}\big)$ is a  $(\alpha,\beta^*)$-gap in $G$.
If $H$ is an ordered field, then $\cf(H^<)=\ci(H^>)=\cf(H)$, and the cuts~$H^{<h}$ and $H^{\leq h}$ ($h\in H$) in~$H$ have character~$\big(\!\cf(H),1\big)$ and $\big(1,\cf(H)^*\big)$, respectively.

\begin{cor}\label{cor:gaps 1}
Let $H$ be a maximal Hardy field, or a maximal analytic Hardy field, or a maximal smooth Hardy field. Set $\kappa:=\ci(H^{>\R})$. Then $\omega<\kappa\leq\mathfrak c$, and~$H$ has gaps of character~$(\omega,\kappa^*)$, $(\kappa,\omega^*)$, and $(\kappa,\kappa^*)$.
\end{cor}
\begin{proof}  Corollary~\ref{cormhcian} and \cite[Corollary~4.8]{ADHfgh} give $\omega<\kappa\leq\mathfrak c$.
The gaps~${\R < H^{>\R}}$ and~${H^{<\R} < \R}$ in $H$ have character $(\omega,\kappa^*)$ and $(\kappa,\omega^*)$, respectively.
To obtain a~$(\kappa,\kappa^*)$-gap in $H$, take a coinitial sequence $(\ell_\rho)_{\rho<\kappa}$   in $H^{>\R}$ with $\ell_{\rho}\succ \ell_{\rho'}$
for all~${\rho < \rho'<\kappa}$. 
Put~$\upg_\rho:=\ell_\rho^\dagger > 0$, so $(1/\ell_\rho)' = (1/\ell_\rho)^\dagger/\ell_\rho = -\upg_\rho/\ell_\rho < 0$. 
Set~$A:=\{\upg_\rho/\ell_\rho:\rho<\kappa\}$, $B:=\{\upg_\rho:\rho<\kappa\}$. With $(\Gamma,\psi)$ the asymptotic couple of $H$, $v(A)$ is coinitial in $(\Gamma^>)'$ and has no smallest element, and
$v(B)$ is cofinal in $\Psi=(\Gamma^{\ne})^\dagger$ and has no largest element.
Now $H$ has asymptotic integration, so there is no $\gamma\in\Gamma$ with~$\Psi<\gamma<(\Gamma^>)'$.
Hence $A<B$  is a $(\kappa,\kappa^*)$-gap in $H$.
\end{proof}

\noindent
Let now $G$ be an ordered abelian group. Assume $G\ne \{0\}$ and $G^>$ has no smallest element, so~$\ci(G^>)\geq\omega$.
A gap $A<B$ in $G$ is said to be {\bf cauchy} if $A,B\neq\emptyset$, $A$ has no largest element, $B$ has no smallest element,
and for each~$\varepsilon\in G^>$ there are~$a\in A$, $b\in B$ with $b-a<\varepsilon$.
If $A<B$ is a cauchy gap in $G$, then so is $-B<-A$.

\begin{lemma}\label{cau}
Let $A<B$ be a cauchy gap in $G$ and  $(a_\rho)$ be an   increasing cofinal well-indexed sequence in $A$.
Then $(a_\rho)$ is a divergent c-sequence in $G$.
\end{lemma}
\begin{proof}
Let $\varepsilon\in G^>$ and take  $a\in A$, $b\in B$ with $b-a<\varepsilon$. Take $\rho_0$ such that $a\leq a_\rho$ for all $\rho>\rho_0$. Then $0< a_{\rho'}-a_\rho<b-a<\varepsilon$  for 
$\rho_0<\rho<\rho'$. Hence $(a_\rho)$ is a c-sequence in $G$,
and  there is no $a\in G$ with $a_\rho\to a$.
\end{proof}

\begin{lemma}\label{lem:cauchy gap}
Every cauchy gap in $G$ has character $\big(\!\cf(G^<),\cf(G^<)^*\big)$. Moreover,  $G$ is complete iff $G$ has no cauchy gap.
\end{lemma}
\begin{proof}
The first claim follows from Lemma~\ref{cau} and [ADH, 2.4.11].
It also follows from this lemma that if $G$ is complete, then $G$ has no cauchy gap.
Conversely, suppose $G$ has no cauchy gap. Let $(a_\rho)$ be a c-sequence in $G$. For each $\varepsilon\in G^>$, take
$\rho_\varepsilon$ such that $\abs{a_\rho-a_{\rho'}}<\varepsilon$ for all $\rho,\rho'\geq\rho_\varepsilon$, and set
$$A\ :=\ \{ a_{\rho_\varepsilon}-\varepsilon:\  \varepsilon\in G^>\}, \qquad 
B\ :=\ \{ a_{\rho_\delta}+\delta:\  \delta\in G^>\}.$$
Then $A<B$ and for all $\varepsilon\in G^{>}$ there are
$a\in A$ and $b\in B$ with $b-a<\varepsilon$. But~$A< B$ is no cauchy gap, so we have $g\in G$ with
$A\leq g\leq B$.  Then $a_{\rho} \to g$. 
\end{proof}

\begin{cor}\label{cor:gaps 2}
Let $H$ be a maximal, or maximal analytic,  or maximal smooth Hardy field, and $\lambda:=\cf(H)$. Then $\omega<\lambda\leq\mathfrak c$, and $H$ has gaps of character~$(0,\lambda^*)$, $(\lambda,0)$, $(1,\lambda^*)$, $(\lambda,1)$, and if $H$ is not complete, then
$H$ has a $(\lambda,\lambda^*)$-gap.
\end{cor}
\begin{proof}
Lemma~\ref{lem:cf(K)} yields $\lambda=\cf(\Gamma)$, so $\omega<\lambda\leq\mathfrak c$ by
\cite[Corollary~8.1]{ADHfgh} and Corollary~\ref{hcoan}. For the rest use  the remarks before Corollary~\ref{cor:gaps 1} and Lemma~\ref{lem:cauchy gap}.
\end{proof}

\noindent
The main result of \cite{ADHfgh}, Theorem~\ref{aneta}, and Corollaries~\ref{cor:gaps 1},~\ref{cor:gaps 2} now give:

\begin{cor}\label{gapsmax} 
Assume  \textup{CH}. If $H$ is a maximal Hardy field, or a maximal analytic Hardy field, or a maximal smooth Hardy field, then the characters of gaps  in $H$ are $$\text{$(0,\omega_1^*)$, $(\omega_1,0)$,  $(1,\omega_1^*)$, $(\omega_1,1)$, $(\omega,\omega_1^*)$, $(\omega_1,\omega^*)$, and $(\omega_1,\omega_1^*)$.}$$
\end{cor}

\subsection*{The number of maximal analytic Hardy fields} 
We recall some definitions from \cite{ADHfgh}. A germ $\phi\in\Cc$ is said to be {\it overhardian}\/ if $\phi$ is hardian and $\phi>_{\ex} \exp_n(x)$ for all $n$;
see \cite[Corollary~5.11]{ADHfgh}. Let $H\supseteq\R$ be a Hardy field. Then $$H^{\te}\ :=\ \big\{f\in H:\text{$f>\exp_n(x)$ for each $n$}\big\}$$ denotes the set of overhardian (or transexponential) elements of $H$. We let $*H^{\te}$ be the set of equivalence classes of the equivalence relation $\sim_{\exp}$ on $H^{\te}$ given by 
$$f\sim_{\exp} g\quad:\Longleftrightarrow\quad
\text{$f\leq\exp_n(g)$ and $g\leq\exp_n(f)$ for some $n$}\qquad (f,g\in H^{\te}).$$
Denoting the equivalence class of $f\in H^{\te}$ by $*f$,
 we linearly order $H^{\te}$ by
$$*f< *g\quad:\Longleftrightarrow\quad \text{$\exp_n(f)<g$ for all $n$} \qquad (f,g\in H^{\te}).$$
We now establish analytic  and smooth versions of Theorem~7.1 from \cite{ADHfgh}:

\begin{cor}\label{cor:nomaxan}
The number of maximal analytic Hardy fields is $2^{\mathfrak c}$ where $\mathfrak c=2^{\aleph_0}$. Likewise with ``smooth'' in place of ``analytic''.
\end{cor}
\begin{proof}
We treat the number of maximal analytic Hardy fields; the smooth case is similar, using the smooth version
of Theorem~\ref{aneta}.
In the argument following the statement of \cite[Proposition~7.4]{ADHfgh} we replace
$\mathcal H$ by the set of all analytic Hardy fields~$H\supseteq\R$ with~$\abs{*H^{\te}} < \frak{c}$.
Thus modified, this argument shows that it is enough to prove that in \cite[Proposition~7.4]{ADHfgh} we can choose $f_0$, $f_1$ to be analytic whenever the Hardy field $H$ is analytic. For this we first note that
if $H$ in \cite[Lem\-ma~7.7]{ADHfgh} is analytic, then we can take $y$ there to be analytic, by appealing to Theorem~\ref{aneta} instead of \cite[Section~5 and Corollary~6.7]{ADHfgh}.
Now argue as in the remarks following~\cite[Lem\-ma~7.10]{ADHfgh} 
using this analytic version of \cite[Lemma~7.7]{ADHfgh}.
\end{proof}

\noindent
Corollary~7.8 of \cite{ADHfgh} has an analytic version with a similar proof:

\begin{cor}\label{anteta} If $H$ is a maximal analytic Hardy field, then  the ordered set $*H^{\te}$ is $\eta_1$, and $\abs{*H^{\te}}=\frak{c}$.
\end{cor}

\noindent 
We now improve Corollary~\ref{CHcof}:  assuming CH,
there are as many cofinal maximal analytic Hardy fields as there are maximal analytic Hardy fields,
by Corollary~\ref{CHcof strengthened}.


\begin{lemma}\label{lem:CHcof strengthened, 1}
Let $\phi\in\Cc^\omega$ be overhardian. Then there 
is a set $\mathcal H_\phi$ of  analytic Hardy field extensions of~$\R\langle\phi\rangle$ with 
$\abs{\mathcal H_\phi}=2^{\mathfrak c}$ such that  for each $H\in\mathcal H_\phi$, 
$*\phi$ is the largest element of  $*H^{\te}$, and each Hardy field 
contains at most one $H\in\mathcal H_\phi$.
\end{lemma}
\begin{proof} By \cite[Lemma~7.7]{ADHfgh} we have $*\R\langle\phi\rangle^{\te}=\{*\phi\}$. 
Let $H\supseteq\R\langle\phi\rangle$ be an analytic Hardy field
with $*\phi=\max *H^{\te}$, and
let~$P<Q$ be a countable gap in~$*H^{\te}$ with~$Q<*\phi$. 
Then   \cite[Proposition~7.4]{ADHfgh} and the argument in the proof of Corollary~\ref{cor:nomaxan} yields analytic Hardy fields~$H_0=H\langle f_0\rangle$ and~$H_1=H\langle f_1\rangle$ without a common Hardy field extension 
such that for $j=0,1$, we have $f_j\in H_j^{\te}$, $P<*f_j<Q\cup\{*\phi\}$, and~$*H_j^{\te}=*H^{\te}\cup\{*f_j\}$ (thus~$*\phi=\max *H_j^{\te}$).

We now follow the argument after the statement of \cite[Proposition~7.4]{ADHfgh}, with~$\mathcal H$ now the set of all analytic Hardy fields~$H\supseteq\R\langle\phi\rangle$ such that~$\abs{*H^{\te}} < \frak{c}$
and~$*\phi=\max*H^{\te}$. For an ordinal $\lambda$ we let $2^{\lambda}$ be the set of functions~${\lambda\to\{0,1\}}$.
With $s$ ranging over $\bigcup_{\lambda<\frak c} 2^\lambda$, we construct
a tree~$(H_s)$ in $\mathcal H$ with~$\abs{*H^{\te}}\leq\abs{{\lambda+1}}$ for~$s\in 2^\lambda$, as follows.
For $\lambda=0$ the function~$s$ has empty domain and we take~$H_s=\R\langle\phi\rangle$.
If~$s\in 2^\lambda$ ($\lambda<\frak c$) and~$H_s\in \cH$ are given with~$\abs{*H_s^{\te}}\le \abs{\lambda+1}$, then \cite[Lemma~7.2]{ADHfgh}  provides a countable gap~$P$,~$Q$ in~${*H_s^{\te}\setminus\{*\phi\}}$, and we let~$H_{s0}, H_{s1}\in \cH$ be obtained from $H_s$ as~$H_0$,~$H_1$ are obtained from $H$ in the remark above.
Suppose~$\lambda< \frak{c}$ is an infinite limit ordinal,
$s\in 2^\lambda$, and that for every $\alpha< \lambda$ we are given~$H_{s|\alpha}\in \cH$ with~$H_{s|\alpha}\subseteq H_{s|\beta}$ whenever
$\alpha\le \beta < \lambda$.
Then we set~$H_s:= \bigcup_{\alpha<\lambda} H_{s|\alpha}\in\cH$. Assuming also inductively that~$\abs{*H_{s|\alpha}^{\te}}\le  \abs{\alpha+1}$ for all~$\alpha<\lambda$,
we have~$\abs{*H_s^{\te}}\le \abs{\lambda}\cdot \abs{\lambda+1}=\abs{\lambda+1}$, as desired.
This finishes the construction of our tree. Then for each~$s\in 2^{\frak c}$   we have 
an analytic Hardy field~$H_s:=\bigcup_{\lambda<\frak{c}} H_{s|\lambda}$ such that if~$s, s'\in 2^{\frak c}$ are different, then~$H_s$,~$H_{s'}$ have no common Hardy field extension.
Hence $\mathcal H_\phi:=\{H_s:s\in 2^{\mathfrak c}\}$ has the required properties.
\end{proof}

\begin{lemma}\label{lem:CHcof strengthened, 2}
Assume \textup{CH}. 
Let $H$ be a bounded analytic Hardy field. Then 
$H$ extends to a cofinal analytic Hardy field.  
\end{lemma}
\begin{proof} By Lemma~\ref{lem:5.4.19} we can
replace~$H$ by $\Li\!\big(H(\R)\big)$
to arrange that $H\supseteq\R$  and~$H$ is Liouville closed. Next, 
take  an enumeration    $(\phi_\alpha)_{\alpha<\mathfrak c}$   of~$\Cc$ 
with $\phi_0>_{\ev} H$.   
 Corollary~\ref{cor:5.1ana} yields an $H$-hardian~$h_0\in\Cc^\omega$ with $h_0>_{\ev} \phi_0$, and 
 then the analytic Hardy field $H_0:=H\langle h_0\rangle$ is bounded by Lemma~\ref{lem:5.4.19}.
 Now a transfinite recursion as in the proof of  Corollary~\ref{CHcof}, beginning with $(H_0,h_0)$,  yields a 
 cofinal analytic Hardy field extension of~$H_0$ and thus of $H$.
\end{proof}

\noindent
Corollary~\ref{cor:5.1ana} gives an overhardian $\phi\in\Cc^\omega$. For such $\phi$ and $\mathcal H_\phi$ as in Lemma~\ref{lem:CHcof strengthened, 1}, all $H\in\mathcal H_\phi$ are bounded.
With  Lemma~\ref{lem:CHcof strengthened, 2} we can now improve Corollary~\ref{CHcof}:

\begin{cor}\label{CHcof strengthened}
Assuming \textup{CH}, there are $2^{\mathfrak c}$  cofinal maximal analytic Hardy fields.
\end{cor}

\subsection*{Maximal analytic Hardy fields approximate maximal Hardy fields}  A maximal analytic Hardy field is
an $\infty\omega$-elementary substructure of any maximal Hardy field extension, by Corollary~\ref{cor:inftyomegaelemsub} below. Maximal analytic Hardy fields are also very close to maximal Hardy fields in another way:

\begin{theorem}\label{thm:dense} Let $H$ be a maximal analytic Hardy field or a maximal smooth Hardy field. Then $H$ is dense in any Hardy field extension of $H$.
\end{theorem}
 
\begin{proof}  We establish two claims:

\claim[1]{If $f\in \Cc^{<\infty}$ is $H$-hardian and $H\<f\>$ is an immediate extension of $H$, then~$H$ is dense in $H\<f\>$.} 

\noindent
To prove this, assume $f\in  \Cc^{<\infty}$ is $H$-hardian, $(f_{\rho})$ is a divergent pc-sequence in~$H$, and $f_{\rho}\leadsto f$. By  [ADH, 16.0.3, Section 11.4],  $(f_{\rho})$ is of $\d$-transcendental type over $H$. 
If the sequence $\big(v(f-f_{\rho})\big)$ is cofinal in $\Gamma:=v(H^\times)$, then~$(f_{\rho})$ 
is a cauchy sequence, and so $H$ is indeed dense in $H\<f\>$, by \cite[Corollary 4.1.6]{ADH6}. 
Suppose~$\big(v(f-f_{\rho})\big)$ is not cofinal in $\Gamma$. Then we have $h\in H^\times$ such that $0\notin v({hf-H})$, so 
Corollary~\ref{apomega} and Proposition~\ref{propimmy2}   yield an
$H$-hardian pseudolimit of $(hf_{\rho})$ in~$\Cc^{\omega}$, contradicting the maximality of $H$. This proves Claim~1.

\claim[2]{For any Hardy field extension $K$ of $H$ we have $\Gamma_K=\Gamma$.} 

\noindent
Towards a contradiction, suppose
$K$ is a Hardy field extension of $H$ and~${\beta\in \Gamma_K\setminus\Gamma}$. We arrange that $K$ is Liouville closed.
Let $(\Gamma,\psi)$ and $(\Gamma_K, \psi_K)$ be the $H$-couples of~$H$ and
$K$ over $\R$, respectively, and let~$\big(\Gamma\<\beta\>,\psi_{\beta}\big)$ be the $H$-couple over
$\R$ generated by $\beta$ over $(\Gamma,\psi)$ in $(\Gamma_K, \psi_K)$. There are several cases to consider, and we show that each is impossible. For {\it closed $H$-couples}\/ and {\it $H$-couples of Hahn type}\/ mentioned below, see~\cite[p.~536]{ADH3}.

First the case that $\big(\Gamma\<\beta\>,\psi_{\beta}\big)$ is an immediate extension of $(\Gamma,\psi)$. Then we have a 
divergent pc-sequence $(\gamma_{\rho})$ in $(\Gamma,\psi)$ with $\gamma_{\rho}\leadsto \beta$. As in the beginning of~\cite[Section~8]{ADHfgh} we take $g_{\rho}\in H$ with $vg_{\rho}=\gamma_{\rho}$ so that $(g_{\rho}^\dagger)$ is a pc-sequence in~$H$, and arguing as in~loc.~cit.~(using $H^\dagger=H$) we see that $(g_{\rho}^\dagger)$ has no pseudolimit in $H$ (because then
$(\gamma_{\rho})$ would have one in $\Gamma$). Now take $g\in K$ with $vg=\beta$. Then~$v(g^\dagger-g_{\rho}^\dagger)=(\beta-\gamma_{\rho})^\dagger$, and the latter is eventually strictly increasing as a function of 
$\rho$, and so $g_{\rho}^\dagger\leadsto g^\dagger$. Moreover, $\big(v(g^\dagger-g_{\rho}^\dagger)\big)$ is not cofinal in
$\Gamma$.  As at the end of the proof of Claim 1, with $g^\dagger$ and $(g_{\rho}^\dagger)$ in the role of $f$ and  $(f_{\rho})$, this contradicts the maximality assumption on $H$. 

Since the $H$-field $H$ is   Liouville closed with constant field $\R$, its $H$-couple $(\Gamma,\psi)$ over $\R$ is closed. 
Hence by \cite[Proposition~4.1]{ADH3} and the remark following its proof, the vector $\beta$ falls under Case (a), or Case (b), or Case
${\rm(c)}_n$ for a certain $n$. In Case~(a) we have~$(\Gamma+\R\beta)^\dagger=\Gamma^\dagger$ and so $\Gamma\<\beta\>=\Gamma+\R\beta$; but $(\Gamma_K, \psi_K)$ is of Hahn type, hence~$\big(\Gamma\<\beta\>,\psi_{\beta}\big)$ is an immediate extension of $(\Gamma,\psi)$, and we have just excluded that possibility. Case ${\rm(c)}_n$ gives an element $\beta_n\in \Gamma\<\beta\>$ with
$\beta_n^\dagger\notin \Gamma$ and~$\beta_n^\dagger$ falling under Case~(a), and so this is also impossible.  

Finally, suppose $\beta$ falls under Case (b). Take $y\in K^{>}$ with $vy=\beta$. Then  Corollary~\ref{cor:revb}  gives an $H$-hardian $z\in \Cc^{\omega}$ with $vz$ realizing the same cut in $\Gamma$ as $\beta$,  contradicting the maximality assumption on $H$. This finishes the proof of Claim 2. 

\medskip\noindent
To finish the proof of the theorem, let $f\in \Cc^{<\infty}$ be $H$-hardian; it suffices to show that then $H$ is dense in $H\<f\>$. 
Now by Claim 2, $H\<f\>$ is an immediate extension of $H$, and hence $H$ is indeed dense in $H\<f\>$ by Claim 1. 
\end{proof}

\begin{question}
Is every maximal Hardy field dense in every Hausdorff field extension?
\end{question}

\subsection*{Dense pairs of closed $H$-fields} Let~$\mathcal L=\{0,1,{-},{+},{\,\cdot\,},{\der},{\leq},{\preceq}\}$ be the language of ordered valued differential rings; cf.~[ADH, p.~678]. We view 
each ordered valued differential field as an $\mathcal L$-structure in the natural way.  
We let
$\mathcal L^2$ extend $\mathcal L$ by a new unary predicate symbol~$U$. 
The $\mathcal L^2$-structures
 are presented as pairs~$(K,F)$ where~$K$ is an $\mathcal L$-structure and $U$ names the subset $F$ of $K$. 
 Let $T$ be the $\mathcal L$-theory of closed $H$-fields with small derivation. 
 Recall from [ADH]  that $T$ is complete and model-complete. Here we announce: 
 
\begin{theorem}\label{densepairs} 
The following requirements on  $\mathcal L^2$-structures $(K,F)$ 
axiomatize a complete $\mathcal L^2$-theory $T^{\operatorname{d}}$: 
\begin{enumerate}
\item[\textup{(1)}] $K\models T$, that is, $K$ is a closed $H$-field with small derivation; 
\item[\textup{(2)}] $F$ is the underlying set of a closed $H$-subfield of $K$; and
\item[\textup{(3)}] $F\neq K$ and $F$ is dense in the ordered field $K$.
\end{enumerate}
Moreover, each $\mathcal L^2$-formula $\varphi(x)$ where $x=(x_1,\dots,x_m)$ is $T^{\operatorname{d}}$-equivalent
to a boolean combination of formulas of the form
\begin{equation}\label{eq:special fm}
\exists y_1\cdots\exists y_n\big( U(y_1)\ \&\ \cdots\ \&\ U(y_n)\ \&\ \psi(x,y)\big)
\end{equation}
where $\psi(x,y)$ with $y=(y_1,\dots,y_n)$ is an $\mathcal L$-formula.
\end{theorem}

\noindent
This follows from Fornasiero's \cite[Theorems~8.3 and 8.5]{Forn}, with  details of how it follows to appear in~\cite{ADHdim+}.
Note that by this theorem the $\mathcal L^2$-theory~$T^{\operatorname{d}}$ is decidable. Moreover,  no pair $(K,F)\models T^{\operatorname{d}}$ induces ``new   structure'' on~$F$:

\begin{cor}\label{cor:forn}
Let $(K,F)\models T^{\operatorname{d}}$, and
let $S\subseteq K^m$ be $A$-definable in $(K,F)$, where~$A\subseteq F$.
Then $S\cap F^m$ is $A$-definable in the $\mathcal{L}$-substructure $F$ of $K$.
\end{cor}
\begin{proof}
By the theorem this reduces to the case where $S$ is defined in $(K,F)$
by a formula 
as in \eqref{eq:special fm} where however $\psi(x,y)$ is now an $\mathcal L_A$-formula.
Then $S\cap F^m$ is defined in $F$ by the $\mathcal L_A$-formula $\exists y\psi(x,y)$.
\end{proof}

\noindent
Note that if $M$ is a maximal analytic or maximal smooth Hardy field and $N$ a maximal Hardy field with $M\subseteq N$, $M\neq N$, then $(N,M)\models T^{\operatorname{d}}$ by Theorem~\ref{thm:dense}. But strictly speaking, we do not know whether there exist such $M$, $N$.

To secure a model of the complete theory $T^{\operatorname{d}}$ we proceed as follows. Let $F$ be an $H$-field. Then the completion $F^{\operatorname{c}}$  of the ordered valued differential field~$F$ 
 is an $H$-field extension of $F$, and
$F$ is dense in~$F^{\operatorname{c}}$; see [ADH, 10.5.9]. 
If $F$ is  closed and of countable cofinality, then  $F^{\operatorname{c}}$ is closed, by~[ADH, 14.1.6], so
 if in addition~$F$  has small derivation and~$F\neq F^{\operatorname{c}}$, then~$(F^{\operatorname{c}},F)\models T^{\operatorname{d}}$.
Now $\T$ is not complete: set~$\ex_0=x$ and $\ex_{i+1}=\exp\ex_i$ for all~$i$; then $\big(\sum_{i=0}^n 1/{\ex_i}\big)_{n=0}^\infty$ is a cauchy sequence in~$\T$ but has no limit in $\T$. 
Therefore $(\T^{\operatorname{c}},\T)\models T^{\operatorname{d}}$.

\section{Analytic Hardy Fields of Countable Cofinality}\label{sec:ctbl cf}

\noindent
Generalizing terminology introduced in \cite[Section~8]{ADHfgh},  call a valued abelian group
{\bf countably spherically complete} if every pc-sequence in it of length~$\omega$ pseudoconverges in it. 
Any $\eta_1$-ordered abelian group with a convex valuation is countably spherically complete, by [ADH, 2.4.2].
Thus maximal analytic  and maximal smooth Hardy fields are countably spherically complete.
In the  first subsection we use this fact to realize the completion of an analytic
Hardy field of countable cofinality as an analytic Hardy field: Corollary~\ref{cor:completion Hardy field}.
Another main result of this section is a realization of the $H$-field $\T_{\log}$ of {\it logarithmic transseries}\/
from [ADH, Appendix~A] as an analytic Hardy field. This is obtained in Corollary~\ref{cor:Tlog}, preceded by some
observations on short ordered sets.

\subsection*{Completing analytic Hardy fields of countable cofinality} 
Lemma~\ref{lem:Kc->M} below concerns $H$-asymptotic fields, and we recall from 
  [ADH, Ch 9]  the definition: an {\it asymptotic field}\/ is a valued differential field $K$ such that for all $f,g\in K^\times$ with~$f,g \prec~1$ we have: ${f\prec g \Leftrightarrow f'\prec g'}$;
 an {\it $H$-asymptotic field}\/ is an asymptotic field $K$ such that for all $f,g\in K^\times$ with $f,g \prec~1$ we have: $f\prec g \Rightarrow f^\dagger\succeq g^\dagger$.
Every pre-$H$-field is an $H$-asymptotic field, by [ADH, 9.1, 10.5]. We shall also mention certain properties
an $H$-asymptotic field may have: being, respectively, {\it $\upl$-free}\/,
{\it $\upo$-free}\/, {\it newtonian}\/,   {\it asymptotically $\d$-algebraically maximal.}\/  For these, see Sections~11.6--11.7  and Chapter~14 of~[ADH], or
the summary in the introduction of~\cite{ADH6}.  For  $H$-fields, being Liouville closed, $\upo$-free, and newtonian is equivalent to being closed. 

Let now $K$ be an asymptotic field. Equip  the completion $K^{\operatorname{c}}$ of the valued field~$K$ with the unique extension of the derivation of $K$
to a continuous derivation on~$K^{\operatorname{c}}$; cf.~[ADH, 4.4.11, 9.1.5].
Then~$K^{\operatorname{c}}$ is asymptotic by [ADH,  9.1.6], and if $K$ is a pre-$H$-field ($H$-field, respectively), then so is~$K^{\operatorname{c}}$ by [ADH, 10.5.9].
Let $L$ be an asymptotic field  extension of $K$  such that
$\Gamma$ is cofinal in $\Gamma_L$. By [ADH, 3.2.20], the natural inclusion~${K\to L}$ extends uniquely to an embedding $K^{\operatorname{c}}\to L^{\operatorname{c}}$
of valued fields, and it is easily checked that this  is an embedding of valued {\it differential}\/ fields.
If $K$ is dense in $L$, then there is a unique valued field embedding $L\to K^{\operatorname{c}}$
over $K$, by [ADH, 3.2.13], and this is also an embedding of  valued  differential  fields.

Whenever in sections 5,6,7 we are given valued differential fields $K$ and $L$ (for example, asymptotic fields), an
{\em embedding $K\to L$} means: an embedding of valued differential fields. If in addition $K$ and $L$ are given as pre-$H$-fields (for example, Hardy fields) such an embedding should also preserve the ordering, that is, be an embedding of ordered valued differential fields.

\begin{lemma}\label{lem:Kc->M}
Let $K$ be an $\upo$-free  $H$-asymptotic field whose value group  $\Gamma_K$
has countable cofinality. Let $M$ be a newtonian  $H$-asymptotic field with asymptotic integration, and suppose $M$ is countably spherically complete. Then  any embedding~$K\to M$ extends to an embedding $K^{\operatorname{c}}\to M$.
\end{lemma}
\begin{proof}
Let $\iota\colon K\to M$ be an embedding; we need to extend $\iota$ to an embedding~${K^{\operatorname{c}}\to M}$.
The $\d$-valued hull  $L:=\operatorname{dv}(K)$ of $K$  is $\upo$-free by [ADH, remark after 13.6.1], and $\Gamma_L=\Gamma_K$
by [ADH, 10.3.2(i)]. By [ADH, 14.2.5], $M$ is $\d$-valued; let
$\iota_L$ be the extension of $\iota$ to an embedding $L\to M$.
Using a remark before the lemma we see that it is enough to show that 
$\iota_L$ extends to an embedding $L^{\operatorname{c}}\to M$.
Hence
replacing $K$, $\iota$ by $L$, $\iota_L$,  we arrange   $K$ is $\d$-valued. Take an immediate asymptotically $\d$-algebraically maximal $\d$-algebraic extension $L$ of $K$; by a remark following the statement of
[ADH, Theorem 14.0.1] such $L$ exists and is $\upo$-free and newtonian.  Then by~\cite[Theorem~3.5]{Nigel19}, $L$ is a newtonization of $K$ (as defined on~[ADH, p.~643]), so embeds into~$M$ over $K$. Passing to this newtonization 
we  arrange that $K$ is 
newtonian.
Then $K^{\operatorname{c}}$ is 
$\upo$-free by [ADH, 11.7.20] and 
newtonian by [ADH, 14.1.5]. 


Suppose $f\in K^{\operatorname{c}}\setminus K$. It suffices to show that then $\iota$ extends to an embedding~$\iota_f\colon K\<f\>\to M$. Here is why:  $K\<f\>$ is $\upo$-free by the remark before~[ADH, 11.7.20], so $K\<f\>$ has a newtonization $E$ in $K^{\operatorname{c}}$ by \cite[Theorem~B]{Nigel19}; by the same remark $E$ is $\upo$-free; moreover, $\iota_f$ extends to an embedding~$E\to M$.  
Hence we can transfinitely iterate this extension process to obtain an embedding~$K^{\operatorname{c}}\to M$ extending $\iota$.

To construct $\iota_f$, 
take a c-sequence~$(f_\rho)$ in~$K$ with $f_\rho\to f$ (in $K^{\operatorname{c}}$).
By [ADH, 2.2.25] the index set of $(f_\rho)$ has cofinality~$\omega$, so 
by passing to a cofinal subsequence  we arrange~$(f_\rho)$ is a divergent pc-sequence in $K$ of length $\omega$ and width $\{\infty\}$ such that~$f_\rho\leadsto f$. Take $g\in M$ such that $\iota(f_\rho)\leadsto g$.  
Now $K$ is asymptotically $\d$-algebraically maximal by \cite[Theorem~A]{Nigel19}, so 
$(f_\rho)$ is of $\d$-transcendental type over~$K$ by~[ADH, 11.4.8, 11.4.13], hence
[ADH, 11.4.7] yields  an embedding~$K\langle f\rangle\to M$ extending $\iota$ and sending
$f$ to $g$. 
\end{proof}


\noindent
Lemma~\ref{lem:Kc->M} yields a   pre-$H$-field version  of it  without the $\upo$-free hypothesis
on $K$:

\begin{prop}\label{prop:completion Hardy field}
Let $K$ be a pre-$H$-field with $\operatorname{cf}(\Gamma_K)=\omega$ and $M$ a  countably spherically complete closed $H$-field.
Then every embedding $K\to M$ extends to  an embedding~$K^{\operatorname{c}}\to M$. 
\end{prop}

\noindent
Before we begin the proof, from [ADH, 16.3.21] we recall that a {\it pre-$\HLO$-field}\/ $\mathbf K= (K, I, \Lambda, \Omega)$  is
a pre-$H$-field $K$ equipped with a {\it $\HLO$-cut}\/ $(I, \Lambda, \Omega)$ of $K$ as defined on~[ADH, p.~691]. A {\it $\HLO$-field}\/ is a pre-$\HLO$-field $\mathbf K = (K; \dots)$ where $K$ is an $H$-field. 
If $\mathbf M = (M; \dots)$ is a pre-$\HLO$-field and $K$
is a pre-$H$-subfield of $M$, then~$K$ has a unique expansion to a pre-$\HLO$-field~$\mathbf K$
such that $\mathbf K \subseteq \mathbf M$. 
Given a  pre-$\HLO$-field~$\mathbf K=(K,\dots)$, we denote the value group and residue field of $K$ by~$\Gamma_{\mathbf K}$,~$\res\mathbf K$, and
 $\mathbf K$  is said to have some given property   of pre-$H$-fields if its underlying pre-$H$-field $K$ does.
 Given pre-$\HLO$-fields $\boldsymbol K$ and $\boldsymbol L$, an {\em embedding\/}
$\boldsymbol K \to \boldsymbol L$ is an embedding in the usual model-theoretic sense.

To show Proposition~\ref{prop:completion Hardy field}, let $K$, $M$ be as in the proposition. We arrange that~$M$ extends $K$  and then have to find an 
embedding~$K^{\operatorname{c}}\to M$ over $K$.
Take any expansion $\mathbf M$ of $M$ to a $\HLO$-field
and expand $K$ to  a pre-$\HLO$-field~$\mathbf K$ such that~${\mathbf K\subseteq\mathbf M}$. Then the proposition below applied
 to $\mathbf M$ in place of $\mathbf L$ yields an $\upo$-free $H$-field extension $K^*$ of~$K$ such that~$K$ is cofinal in $K^*$ and an embedding $\iota^*\colon K^*\to M$ over $K$.
    Lemma~\ref{lem:Kc->M} gives an extension of $\iota^*$ to an embedding $(K^*)^{\operatorname{c}}\to M$,
  and by a remark before that lemma this yields an embedding $K^{\operatorname{c}}\to M$ as required.

\medskip
\noindent  
It remains to establish the following ``cofinality'' refinement of [ADH, 16.4.1]. Here we recall that a Liouville closed $H$-field $K$ is said to be {\it Schwarz closed}\/ if for all~$a\in K$ the linear differential operator $\der^2-a$ splits over the algebraic closure~$K[\imag]$ of~$K$,
and for all $a,b\in K$, if $a\leq b$, and $\der^2-a$ splits over $K$, then so does $\der^2-b$; cf.~[ADH, 5.2, 11.8].
Every closed $H$-field is Schwarz closed [ADH, 14.2.20].

\begin{prop}\label{1641}
Let $\mathbf K$ be a pre-$\HLO$-field with $\Gamma_{\mathbf K}\neq\{0\}$. Then there exists an $\upo$-free  $\HLO$-field extension $\mathbf K^*$
of $\mathbf K$ such that;
\begin{enumerate}
\item[(i)] $\res\mathbf K^*$ is algebraic over $\res\mathbf K$;
\item[(ii)] $\mathbf K$ is cofinal in $\mathbf K^*$; and
\item[(iii)] any embedding of $\mathbf K$ into a Schwarz closed $\HLO$-field $\mathbf L$ extends to
an embedding $\mathbf K^*\to\mathbf L$.
\end{enumerate}
\end{prop} 

\noindent
We revisit the proof of [ADH, 16.4.1], which consists of several lemmas and a corollary.  
Recall: a differential field $F$ is said to be {\it closed under logarithms}\/ if for all $f\in F$ there is a~$y\in F^\times$ such that~$y^\dagger=f'$,
and $F$ is {\it closed under integration}\/ if for all~$g\in F$ there is a~$z\in F$ such that~$z'=g$.
Let $\mathbf K=(K,I,\Lambda,\Omega)$ be a pre-$\HLO$-field with~$\Gamma:=\Gamma_K\neq\{0\}$.

\begin{lemma}
Suppose $K$ is grounded, or there exists $b\asymp 1$ in $K$ such that $v(b')$ is a gap in $K$. Then $\mathbf K$
has an $\upo$-free $\HLO$-field extension $\mathbf K^*$ such that $\res\mathbf K=\res\mathbf K^*$, $\mathbf K$ is cofinal in $\mathbf K^*$, and any embedding
of $\mathbf K$ into a $\HLO$-field $\mathbf L$ closed under logarithms extends to an embedding $\mathbf K^*\to\mathbf L$.
\end{lemma} 
\begin{proof}
By Lemma~\ref{lem:cof 1}, $K$ is cofinal in the $H$-field hull $F:=H(K)$ of $K$,
and hence by Lemma~\ref{lem:H(x)}, $K$ is also cofinal in the $H$-field extension $F_{\upo}$ of $F$ constructed in~[ADH, 11.7].
Thus the lemma follows from the proof of [ADH, 16.4.2].
\end{proof}

\begin{lemma}
Suppose $K$ has gap $\beta$ and $v(b')\neq\beta$ for all $b\asymp 1$ in $K$. Then there exists a grounded pre-$\HLO$-field
extension $\mathbf K_1$ of $\mathbf K$ such that $\res\mathbf K=\res\mathbf K^*$, $\mathbf K$ is cofinal in $\mathbf K_1$, and any embedding of $\mathbf K$ into a $\HLO$-field $\mathbf L$ closed under
integration extends to an embedding $\mathbf K_1\to\mathbf L$.
\end{lemma}
\begin{proof}
Take $s\in K$ such that $vs=\beta$.  
Recall from [ADH,~14.2] that $\I(K)$ denotes the
$\mathcal O$-submodule of $K$ generated by $\der\mathcal O$.
Following the proof of [ADH, 16.4.3],
suppose~${s\notin \I(K)}$, and take $K_1$ as in Case~1 of that proof, so 
$K_1=H(K)(y)$ where~${y'=s}$. Now use that $\Gamma_{H(K)}=\Gamma$, and that $\Gamma_{H(K)}$ is cofinal in $\Gamma_{H_1}$ by Lem\-ma~\ref{lem:H(x)}. 
If~${s\in \I(K)}$ and~$K_1$ is as in Case~2, then $K_1=K(y)$ where $y'=s$, so again  $K$ is cofinal in $K_1$ by
Lemma~\ref{lem:H(x)}.
\end{proof}

\noindent
These two lemmas yield a ``cofinality'' refinement of [ADH, 16.4.4]:

\begin{cor}\label{cor:16.4.4}
Suppose $K$ does not have asymptotic integration. Then $\mathbf K$ has
an $\upo$-free $\HLO$-field extension $\mathbf K^*$ such that 
$\res\mathbf K^* = \res\mathbf K$, $\mathbf K$ is cofinal in $\mathbf K^*$, and   any embedding of $\mathbf K$
into a $\HLO$-field $\mathbf L$ closed under integration extends to an embedding $\mathbf K^*\to\mathbf L$.
\end{cor}

\noindent
The next three lemmas are ``cofinality''  refinements of [ADH, 16.4.5, 16.4.6, 16.4.7] and take care of the case where $K$ has asymptotic integration.

\begin{lemma}\label{lem:16.4.5}
Assume $K$ has asymptotic integration and is not $\upl$-free. Then~$\mathbf K$
extends to an $\upo$-free $\HLO$-field $\mathbf K^*$ such that
$\res\mathbf K^* = (\res\mathbf K)^{\operatorname{rc}}$,  $\mathbf K$ is cofinal in~$\mathbf K^*$, and any embedding
of $\mathbf K$ into a Liouville closed $\HLO$-field $\mathbf L$ extends to an embedding~$\mathbf K^*\to\mathbf L$.
\end{lemma}
\begin{proof}
As in the proof of [ADH, 16.4.5], it is enough, by Corollary~\ref{cor:16.4.4},  to show that~$\mathbf K$ has a $\HLO$-field extension
$\mathbf K_1=(K_1,\dots)$ with a gap such that  $\res\mathbf K_1 = (\res\mathbf K)^{\operatorname{rc}}$,  $K$ is cofinal in $K_1$, and any embedding
of $\mathbf K$ into a Liouville closed $\HLO$-field $\mathbf L$ extends to an embedding $\mathbf K_1\to\mathbf L$.
Take $\mathbf K_1$ as in the proof of [ADH, 16.4.5].
Put $E:=H(K)^{\operatorname{rc}}$. Then $\Gamma_{E}=\Q\Gamma$, so $K$ is cofinal in $E$.
If $E$ has a gap, then~$K_1=E$, and we are done.
Suppose $E$ has no gap. Then $K_1=E(f)$ where~$f\in K_1^\times$ and $\upl:=-f^\dagger\in K$, and $s:=-\upl$ creates a gap over $E$ (as defined in [ADH, p.~503]).
By the proof of Case~2 in~[ADH, 16.4.5], the hypothesis of
Lemma~\ref{lem:10.4.5} holds for $E$ in place of~$K$, so~$E$ is cofinal in $K_1$, and hence so is $K$.
\end{proof}

\begin{lemma}
Suppose $K$ is $\upl$-free but not $\upo$-free. Then $\mathbf K$ has an $\upo$-free $\HLO$-field
extension $\mathbf K^*$ such that $\res\mathbf K^*$ is algebraic over $\res\mathbf K$, $\mathbf K$ is cofinal in $\mathbf K^*$,
 and any embedding of $\mathbf K$ into
a Schwarz closed $\HLO$-field $\mathbf L$ extends to an embedding of~$\mathbf K^*$ into $\mathbf L$.
\end{lemma}
\begin{proof}
For the definition 
of the pc-sequence $(\upo_\rho)$ in $K$ and the $\d$-rational functions~$\omega$,~$\sigma$ and their  role in $\upo$-freeness as
used in this proof, see [ADH, 11.7, 11.8].
Take $\upo\in K$ with $\upo_\rho\leadsto\upo$. Let $\mathbf K^*$ be as in the proof of [ADH, 16.4.6].
That proof shows that  $\Omega=\omega(K)^\downarrow$ or $\Omega=K\setminus\sigma\big(\Upg(K)\big){}^\uparrow$.
Suppose first that~$\Omega=\omega(K)^\downarrow$. With $\mathbf K_{\upg}=(K_{\upg},\dots)$ as in Case~1 of that proof, we have
  $K_{\upg}=K\langle\upg\rangle$ where~$\upg\neq 0$, $\sigma(\upg)=\upo$, and $v\upg$ is a gap in $K_{\upg}$.
The remarks before [ADH,  13.7.7] give $[\Gamma]=\big[\Gamma_{K_{\upg}}\big]$, so $K$ is cofinal
in $K_{\upg}$. Now follow the argument in Case~1 of loc.~cit., using Corollary~\ref{cor:16.4.4} instead of [ADH, 16.4.4].
If~$\Omega=K\setminus\sigma\big(\Upg(K)\big){}^\uparrow$, then we argue as in Case~2 of loc.~cit.,
using Lemma~\ref{lem:16.4.5} instead of [ADH, 16.4.5].
\end{proof}

\begin{lemma}
Suppose $K$ is $\upo$-free. Then $\mathbf K$ has an $\upo$-free $\HLO$-field extension $\mathbf K^*$ such that
$\res\mathbf K^* = \res\mathbf K$, $\mathbf K$ is cofinal in $\mathbf K^*$, and any embedding of $\mathbf K$ into a $\HLO$-field $\mathbf L$ extends to an embedding of $\mathbf K^*$ into $\mathbf L$.
\end{lemma}
\begin{proof}
Take $\mathbf K^*=(K^*,\dots)$ as in the proof of [ADH, 16.4.7]. Then $K^*=H(K)$, and
by Lem\-ma~\ref{lem:cof 1}, $K$ is cofinal in $H(K)$.
\end{proof}

\noindent
 This concludes the proof of  Proposition~\ref{1641} and of Proposition~\ref{prop:completion Hardy field}. Combining the latter
  with  \cite[Corollary~3.2]{ADHfgh}  and Corollary~\ref{smanpc} yields:

\begin{cor}\label{cor:completion Hardy field}
Let $H$ be a  Hardy field of countable cofinality and $M\supseteq H$ a maximal Hardy field.
Then there is an embedding $H^{\operatorname{c}}\to M$ over~$H$. Likewise if~$M\supseteq H$ is a maximal analytic Hardy field or a maximal smooth Hardy field.
\end{cor}

\noindent
As a consequence of Corollary~\ref{cor:completion Hardy field}, with $t:=x^{-1}$,
each maximal analytic Hardy field contains a Hardy field extending $\R(t)$ and isomorphic over~$\R(t)$ to the ordered   field~$\R(\!(t)\!)$ of Laurent series over $\R$ equipped with the continuous $\R$-linear derivation given by $t'=-t^2$.
(This may be viewed as a Hardy field version of Besicovitch's strengthening \cite{Besi} of
Borel's theorem on $\Cc^\infty$-functions with prescribed Taylor series~\cite{Borel}.)
In Corollary~\ref{cor:embed T}   we show that even the ordered differential field~$\T$ of transseries, which vastly extends $\R(\!(t)\!)$, embeds into any given maximal analytic Hardy field. As a first step
we accomplish this below for
the $H$-subfield~$\T_{\log}$ of~$\T$ of logarithmic transseries (cf.~[ADH, p.~722]). For this it is useful to
have available some facts about short ordered sets, also needed in Section~\ref{eie}. 

\subsection*{Short ordered sets} 
Let $S$ be an ordered set. (As in \cite{ADHfgh}, this means ``linearly ordered set''.)
Let~$S^*$ denote
$S$ equipped with the reversed ordering.
Then the following are equivalent:
\begin{enumerate}
\item[(S1)]   all well-ordered subsets of $S$ 
and of $S^*$ are countable;
\item[(S2)]  there are no embeddings of $\omega_1$ into $S$ or $S^*$;
\item[(S3)] $\operatorname{cf}(A),\operatorname{ci}(A) \leq\omega$ for all ordered subsets $A$ of $S$. 
\end{enumerate}
Call  $S$ {\bf short} if any of the equivalent conditions (S1)--(S3) holds; cf.~\cite[1.7(i)]{DW}  and~\cite[pp.~88, 170--171]{Rosenstein}. 
If  $S$ is short, then so are~$S^*$ and every ordered subset of~$S$. If~$S$ is countable, then it is short; more generally, if
 $S$ is a union of countably many short ordered subsets, then $S$ is short.
If $S\to S'$ is a surjective increasing map between ordered sets and $S$ is short, then so is $S'$; similarly with ``decreasing'' instead of ``increasing''.  Shortness   enters our story via the following observation:

\begin{lemma}\label{lem:short}
Let $(G,S,v)$ be a valued abelian group where $S$ is short, and let $(a_\rho)$ be
a  pc-sequence in $(G,S,v)$. Then some final segment of $(a_\rho)$ has countable length.
\end{lemma}
\begin{proof}
Put $s_\rho:=v(a_{\rho+1}-a_\rho)$, where~$\rho+1$ is the successor of $\rho$.
After deleting an initial segment of  $(a_\rho)$  we arrange that the sequence~$(s_\rho)$ in $S$ 
is strictly increasing. Then the  image of the index set of $(a_\rho)$ under the embedding $\rho\mapsto s_\rho$ of ordered sets is a well-ordered subset of $S$ and hence countable.
\end{proof}

\begin{lemma}
If the order topology of $S$ is second countable, then $S$ is short.
\end{lemma}
\begin{proof}
Suppose $i\colon \omega_1\to S$ is strictly increasing.
With $\lambda$ ranging over the limit ordinals~$<\omega_1$ we then have uncountably many nonempty pairwise disjoint open intervals $\big(i(\lambda),i(\lambda+2)\big)$
in $S$, so $S$ is not second countable. An embedding $\omega_1\to S^*$ yields the same conclusion. 
\end{proof}

\noindent
In particular, the real line (the ordered set of real numbers) is short.
(In fact, 
by~\cite[Theorem~2]{HarringtonShelah},   each Borel ordered set is short.)
The following observation  is due to
Hausdorff~{\cite[p.~133]{Hau07}} and Urysohn~\cite{Urysohn24}. 

\begin{lemma}\label{lem:embed short into eta1}
Suppose $S$ is short and $T$ is an $\eta_1$-ordered set.  Then any embedding of an ordered subset of $S$ into $T$ extends to an embedding~${S\to T}$. In particular, there exists an embedding $S\to T$.
\end{lemma}
\begin{proof}
Let $A$ be an ordered subset of $S$ and $i\colon A\to T$ an embedding. Suppose~$s\in S\setminus A$. 
Then $\operatorname{cf}(A^{<s}),\operatorname{ci}(A^{>s})\leq\omega$, so
we have $t\in T$ with $i(A^{<s})<t<i(A^{>s})$. Thus~$i$ extends to an
embedding $A\cup\{s\}\to T$ sending $s$ to $t$. Zorn does the rest. 
\end{proof}

\begin{cor}\label{lem:eta1 size}
Every $\eta_1$-ordered set has cardinality $\geq\frak c$. 
There is an~$\eta_1$-ordered set of cardinality $\frak c$.
\end{cor}
\begin{proof} For the first claim, apply Lemma~\ref{lem:embed short into eta1} to $S=\text{the real line}$. 
The second claim follows from the first together with
 [ADH, B.9.6]. 
\end{proof}

\noindent
Combining Lemma~\ref{lem:embed short into eta1} and Corollary~\ref{lem:eta1 size} we obtain:

\begin{cor}[Urysohn~\cite{Urysohn23,Urysohn24}]
Every short ordered set has cardinality $\leq\frak c$.
\end{cor}

\noindent
For (ordered)  Hahn products, see [ADH, 2.2, 2.4]. Shortness of $\R$ is at the root of the following result due  to Esterle
 \cite[Lemme~2.2 and the remark after it]{Esterle}:

\begin{lemma}\label{lem:Esterle} 
If $S$ is short, then so is the Hahn product $H[S,\R]$. 
\end{lemma}

\noindent
From Lemma~\ref{lem:Esterle} and the Hahn Embedding Theorem  [ADH, 2.4.19] we obtain a characterization of short ordered abelian groups:

\begin{lemma}\label{lem:short char}
For an ordered abelian group $\Gamma$, the following are equivalent:
\begin{enumerate}
\item[\textup{(i)}] $\Gamma$ is short;
\item[\textup{(ii)}] the ordered set $[\Gamma]$ is short;
\item[\textup{(iii)}] $\Gamma$ embeds into $H[S,\R]$ for some short $S$.
\end{enumerate}
\end{lemma}

\begin{cor}\label{cor:short oag}
Let $\Delta\subseteq\Gamma$ be an extension of ordered abelian groups. Then
$$\text{$\Gamma$ is short} \quad\Longleftrightarrow\quad \text{$\Delta$ and $[\Gamma]\setminus[\Delta]$ are short.}$$
In particular,  if $\operatorname{rank}_{\Q}(\Gamma/\Delta)\le \aleph_0$, then  $\Gamma$ is short iff $\Delta$ is short,
and if $\Delta$ is convex, then $\Gamma$ is short iff $\Delta$ and $\Gamma/\Delta$ are short.
\end{cor}
\begin{proof} The direction $\Rightarrow$ is clear from Lemma~\ref{lem:short char}. For the converse, note that if
$\Delta$ and ${[\Gamma]\setminus [\Delta]}$ are short, then so is $[\Gamma]$, and hence $\Gamma$ as well by  Lemma~\ref{lem:short char}. Next, use that if~$\operatorname{rank}_{\Q}(\Gamma/\Delta)\le \aleph_0$, then $[\Gamma]\setminus [\Delta]$ is countable by [ADH, 2.3.9]. For convex~$\Delta$, see [ADH, p.~102].
\end{proof}

\begin{lemma}\label{lem:short ofield}
Let $K$ be an ordered field equipped with a convex valuation whose residue field is archimedean. 
Then $K$ is short iff its value group $\Gamma$ is short.
\end{lemma}
\begin{proof}
Suppose $\Gamma$ is short. Then $\Q\Gamma$ is also short, by the previous corollary,
and the real closure of $\res K$ remains archimedean;
hence to show that $K$ is short we may replace $K$ by its real closure to arrange that $K$ is real closed.
Using [ADH, 3.3.32, 3.3.42, 3.5.1, 3.5.12] we obtain an ordered  field embedding of $K$ into the ordered Hahn field $\R(\!(t^\Gamma)\!)$.
The underlying ordered additive group of $\R(\!(t^\Gamma)\!)$ is isomorphic with the ordered
Hahn product~$H[t^{\Gamma},\R]$; see~[ADH, p.~114].
Hence   $K$ is short by Lemma~\ref{lem:short char}. Conversely, if $K$ is short then so is its ordered subset $K^>$
and then also the image $\Gamma$ of $K^>$ under the decreasing map~$f\mapsto vf\colon K^>\to\Gamma$.
\end{proof}

\noindent
Hence if an ordered field is short, then so is its real closure.
If $K$  as in Lemma~\ref{lem:short ofield} is short and $L$ is an ordered field extension of $K$ with a convex valuation 
that makes it an immediate  extension of $K$, then $L$ is short.
(NB: the ordered fraction field of a short ordered integral 
domain may fail to be short \cite{Ci1,Ci2}.)
The following is from  \cite[\S{}2.10]{vdDMM}:

\begin{cor}\label{cor:T short}
$\T$ is short.
\end{cor}
\begin{proof}
We recall some features of the construction of $\T$ from [ADH, Appendix~A].
We have the ordered subfield 
$\T_{\exp}=\bigcup_m E_m$ of $\T$ where $E_m=\R[[G_m]]$ for certain ordered subgroups $G_m$ of $\T^>$,
with~$G_0=x^{\R}$ and $G_{m+1}=G_m\exp(A_m)$ for some
subgroup $A_m$ of the additive group of~$E_m$, with $G_m$ a convex subgroup of $G_{m+1}$.
An easy induction on  $m$   shows that each $E_m$ is short, and
thus $\T_{\exp}$ is short.
Now $\T=\bigcup_n (\T_{\exp}){\downarrow^n}$ where $f\mapsto f{\downarrow^n}$ is the $n$th compositional iterate of
the automorphism $f\mapsto f{\downarrow}=f\circ\log x$ of the ordered field $\T$, hence
$\T$ is also short.
\end{proof}

\begin{question}
Are $\d$-algebraic Hardy field extensions of short Hardy fields also short?  
\end{question}

\noindent
Next two algebraic variants of Lemma~\ref{lem:embed short into eta1}, attributed to Esterle in \cite[2.37]{DW}:

\begin{lemma}\label{shortetaab}
Let $\Delta$ be a short ordered abelian group and $\Gamma$ a divisible $\eta_1$-ordered abelian group.  
Then any  embedding of an ordered subgroup of $\Delta$ into $\Gamma$ extends to an embedding~${\Delta\to \Gamma}$. 
\end{lemma}
\begin{proof} Let $\Delta_0$ be an ordered subgroup of $\Delta$ and $i\colon \Delta_0\to \Gamma$
 an embedding.
The divisible hull $\Q\Delta\subseteq \Gamma$ of $\Delta$ is   short, by Corollary~\ref{cor:short oag}. Replace~$\Delta_0$, $\Delta$ by~$\Q\Delta_0$,~$\Q\Delta$ (and $i$ accordingly) to arrange   $\Delta_0$, $\Delta$ to be divisible.
Given~${\delta\in \Delta\setminus \Delta_0}$,
Lemma~\ref{lem:embed short into eta1} yields $\gamma\in\Gamma$ 
with~$i(\Delta_0^{<\delta})<\gamma<i(\Delta_0^{>\delta})$, and
then $i$ extends to an embedding of the ordered subgroup $\Delta_0\oplus\Q\delta$ of $\Delta$ into $\Gamma$  sending $\delta$ to $\gamma$. Zorn does the rest. 
\end{proof}

\noindent
In the same way, taking real closures instead of divisible hulls in the proof:

\begin{lemma}\label{lem:embed short ordered fields} Any embedding of an ordered subfield of a short ordered field $K$ into
a real closed $\eta_1$-ordered field $L$ extends to an embedding $K\to L$.
\end{lemma}

\noindent
Combining Corollary~\ref{cor:T short} and the previous lemma yields:

\begin{cor}
The ordered field $\T$ embeds into each  real closed $\eta_1$-ordered field.
\end{cor}

\noindent
Lemma~\ref{lem:embed short H-fields} below is an analogue of Lemma~\ref{lem:embed short ordered fields} for $H$-fields   with small derivation.



\subsection*{Realizing $\T_{\log}$ as an analytic Hardy field}
This uses the following variant of Lem\-ma~\ref{lem:Kc->M} for embedding $\upo$-free immediate extensions:

\begin{lemma}\label{lem:L->M}  
Let $K$ be an $H$-asymptotic field with short value group,
$L$ an $\upo$-free   immediate extension of~$K$, and $M$ a
newtonian  $H$-asymptotic field with asymptotic integration. Suppose $M$ is countably spherically complete.
Then any embedding~$K\to M$ extends to an embedding~$L\to M$.
\end{lemma}
\begin{proof}
Let $\iota\colon K\to M$ be an embedding; we shall extend $\iota$ to an embedding ${L\to M}$. Now $L$ is pre-$\d$-valued by [ADH, 10.1.3], and as $\operatorname{dv}(L)$ is $\upo$-free by [ADH, remark after 13.6.1] and $\Gamma_{\operatorname{dv}(L)}=\Gamma$ by [ADH, 10.3.2(i)], we can replace $L$ by $\operatorname{dv}(L)$ to arrange that $L$ is $\d$-valued. Then $L$ has an immediate $\d$-algebraic newtonian
$\upo$-free extension by [ADH, remark after 14.0.1], which is then a newtonization of~$L$ by~\cite[Theorem~3.5]{Nigel19}. Replacing $L$ by this newtonization we also arrange that~$L$ is newtonian. 
Using Zorn we further arrange that $\iota$ does not extend to any embedding into~$M$ of any valued differential subfield of $L$
properly containing $K$. Note that~$K$ is $\upo$-free by [ADH, remark preceding 11.7.20]. 
Now $M$ is $\d$-valued by [ADH, 14.2.5], 
hence so is $K$ by the universal property of $\operatorname{dv}(K)$. Likewise, $K$ is newtonian, by the semiuniversal property
of the newtonization of $K$ (which exists by the same arguments as we used for $L$). 
Hence  $K$ is asymptotically $\d$-algebraically maximal by~\cite[Theorem~A]{Nigel19}. 
It remains to show that $K=L$. Suppose towards a contradiction that 
$f\in L\setminus K$. Take a divergent pc-sequence
$(f_\rho)$ in~$K$ with pseudolimit~$f$. By Lemma~\ref{lem:short} we arrange that $(f_\rho)$ has
length~$\omega$, and hence we can take $g\in M$ with~$\iota(f_\rho)\leadsto g$.
As in the proof of Lemma~\ref{lem:Kc->M} we then obtain an  embedding~$K\langle f\rangle\to M$ extending $\iota$
and sending $f$ to $g$, a contradiction.
\end{proof}


\noindent
Set $\ell_0:=x\in\T$ and $\ell_{n+1}:=\log\ell_n$.
Recall that~$\T_{\log}=\bigcup_n \R[[\mathfrak L_n]]$ where $\mathfrak L_n:=\ell_0^{\R}\cdots\ell_n^{\R}$
is the subgroup of the monomial group $G^{\operatorname{LE}}$ of~$\T$ generated by the real powers of
the $\ell_i$ ($i=0,\dots,n$). 
The ordered subgroup~$\mathfrak L:=\bigcup_n\mathfrak L_n$ of $G^{\operatorname{LE}}$  is divisible and short,   
$\T_{\log}$ is real closed, $\upo$-free and an immediate $H$-field extension of its $H$-subfield $\R(\mathfrak L)$.  
Identify   $\R(\mathfrak L)$  with an $H$-subfield of the analytic Hardy field~$\Li\!\big(\R(x)\big)$ 
in the obvious way.
From  \cite[Corollary~3.2]{ADHfgh},
Corollary~\ref{smanpc}, and Lemma~\ref{lem:L->M}, we obtain:

\begin{cor}\label{cor:Tlog}
The $H$-field $\T_{\log}$ embeds over $\R(\mathfrak{L})$ into any maximal Hardy field. 
Likewise with {\em maximal analytic\/} and with {\em maximal smooth\/} in place of~{\em maximal}.
\end{cor}

\noindent
By Corollary~\ref{cor:completion Hardy field}, every embedding $i\colon \T_{\log}\to M$ as in Corollary~\ref{cor:Tlog}
extends to an embedding of the completion of $\T_{\log}$ into $M$. In the next section we  
show that~$i$ even extends to an embedding of {\it every}\/  immediate $H$-field extension of $\T_{\log}$
into~$M$.

\subsection*{Implications between ``short'', ``countable cofinality'', and ``bounded''} The first two notions are defined
for ordered sets, and ``bounded'' is defined for subsets of $\Cc$. It is clear that for ordered sets, 
$$\text{short}\ \Longrightarrow\  \text{countable cofinality},$$
 and that for Hausdorff fields,
 $$\text{ countable cofinality}\ \Longrightarrow\  \text{bounded}.$$
These implications cannot be reversed for analytic Hardy fields: Let $H$ be a maximal analytic Hardy field. Corollary~\ref{anteta}  gives a sequence $(h_{\lambda})$ in $H$, indexed
by the ordinals $\lambda \leq \omega_1$, such that all $h_\lambda$ are transexponential and $*h_\lambda < *h_\mu$ for all~${\lambda < \mu\leq \omega_1}$.
It follows that $ \R\<h_\lambda:\,  \lambda< \omega_1\>$ is a bounded analytic Hardy field (bounded by $h_{\omega_1}$) with cofinality 
$\omega_1$, and so $ \R\<h_\lambda:\, \lambda\leq \omega_1\>$  is an analytic Hardy field of cofinality $\omega$ that is not short. (We thank Philip Ehrlich and Elliot Kaplan for a useful email discussion on this topic.)

\section{Embeddings of Immediate Extensions}\label{eie}

 \noindent
The goal of this section is to prove the following theorem, which partly generalizes Lemma~\ref{lem:L->M} beyond the $\upo$-free setting:

\begin{theorem}\label{thm:embedd imm}
Let $K$ be a short pre-$H$-field with archimedean residue field, and 
suppose $K$ is $\upo$-free or not $\upl$-free.
Let
$\hat{K}$ be an  immediate pre-$H$-field extension of~$K$ 
and let $M$ be a countably spherically complete closed $H$-field. Then every embedding~$K\to M$   extends to an embedding~$\hat{K}\to   M$.
\end{theorem}

\noindent
Using also 
\cite[Corollary~3.2]{ADHfgh} and
Corollary~\ref{smanpc}, this yields:

\begin{cor}
Let $K$ be a short Hardy field which is $\upo$-free or not $\upl$-free, and let~$M$ be a 
maximal Hardy field extending $K$. Then every immediate Hardy field extension of $K$ embeds into~$M$ over $K$. Likewise with ``maximal analytic'' as well as with ``maximal smooth'' in place of ``maximal''.
\end{cor} 

\noindent
The main steps towards the proof of Theorem~\ref{thm:embedd imm} are Propositions~\ref{prop:upo-free imm} and~\ref{prop:HLOimm} below. This requires us to revisit the topic of pre-$\HLO$-fields once again.

We note also that by \cite{VDF} and [ADH, 10.5.8], each pre-$H$-field has an immediate strict pre-$H$-field extension that is
spherically complete.

\subsection*{Immediate pairs of pre-$\HLO$-fields}
Here we generalize [ADH, 16.4.1] to certain pairs of pre-$\HLO$-fields. A {\bf pre-$\HLO$-pair}
is a pair~$(\mathbf K,\hat{\mathbf K})$ of pre-$\HLO$-fields
with~${\mathbf K \subseteq \hat{\mathbf K}}$. Let 
 $(\mathbf K,\hat{\mathbf K})$ be a pre-$\HLO$-pair, with $\mathbf K=(K,\dots)$ and $\hat{\mathbf K}=(\hat{K},\dots)$. We call $(\mathbf K,\hat{\mathbf K})$ a {\bf $\HLO$-pair} if both~$\mathbf K$,~$\hat{\mathbf K}$ are $\HLO$-fields,
and we say that~$(\mathbf K,\hat{\mathbf K})$ is {\bf immediate} if the valued field extension $K\subseteq \hat K$ is immediate.
 We also call~$(\mathbf K,\hat{\mathbf K})$ {\bf $\upo$-free} if both $K$, $\hat{K}$ are $\upo$-free, and similarly for
other properties of pre-$H$-fields.
A pre-$\HLO$-pair~$(\mathbf K^*,\hat{\mathbf K}^*)$
{\bf extends}~$(\mathbf K,\hat{\mathbf K})$ if~$\mathbf K\subseteq\mathbf K^*$ and
$\hat{\mathbf K}\subseteq\hat{\mathbf K}^*$.

\begin{prop}\label{prop:upo-free imm}
Suppose $(\mathbf K,\hat{\mathbf K})$ is an immediate pre-$\HLO$-pair such that if $K$ is $\upo$-free \textup{(}$\upl$-free, respectively\textup{)}, then so is~$\hat K$.
Then $(\mathbf K,\hat{\mathbf K})$ extends to an immediate $\upo$-free $\HLO$-pair $(\mathbf K^*,\hat{\mathbf K}^*)$ 
such that $\res{\mathbf K}^*$ is algebraic over~$\res{\mathbf K}$ and   any embedding of~$\mathbf K$ into a Schwarz closed $\HLO$-field $\mathbf L$ extends to
an embedding~$\mathbf K^*\to\mathbf L$. 
$$
\begin{CD}
\mathbf K^* @>\subseteq>{\textup{\tiny\ immediate\ }}>  \hat{\mathbf K}^* \\ 
@AA{\subseteq}A @A{\subseteq}AA  \\
\mathbf K   @>\subseteq>{\textup{\tiny\ immediate\ }}>  \hat{\mathbf K}  
\end{CD}$$
Moreover, if~$\mathbf K$ is short and  $\res\mathbf K$ is archimedean, then we can choose such a pair~$(\mathbf K^*$, $\hat{\mathbf K}^*)$ where $\mathbf K^*$ is also short.
\end{prop}

\noindent
As with Proposition~\ref{1641} we adapt the proof of  [ADH, 16.4.1].  We assume 
$(\mathbf K,\hat{\mathbf K})$ is an immediate pre-$\HLO$-pair and $\mathbf K=(K,I,\Lambda,\Omega)$, $\hat{\mathbf K}=(\hat K,\hat I,\hat \Lambda,\hat \Omega)$.
We identify~$H(K)$ in the usual way with an $H$-subfield of $H(\hat K)$, and for ungrounded $K$ we tacitly use
 that the sequences $(\upl_\rho)$, $(\upo_\rho)$ in $K$ also serve for~$\hat K$. (See [ADH, 11.5--11.7] for the definition and basic properties of  $(\upl_\rho)$, $(\upo_\rho)$.)

\begin{lemma}
Suppose $K$ is grounded, or there exists $b\asymp 1$ in $K$ such that $v(b')$ is a gap in $K$. Then $(\mathbf K,\hat{\mathbf K})$ extends to an immediate $\upo$-free $\HLO$-pair $(\mathbf K^*,\hat{\mathbf K}^*)$  such that $\res\mathbf K^* = \res\mathbf K$ and any embedding
of $\mathbf K$ into a $\HLO$-field $\mathbf L$ closed under logarithms extends to an embedding $\mathbf K^*\to\mathbf L$.
\end{lemma} 
\begin{proof}
Note that if $K$ is grounded, then so is $\hat K$, and any gap in $K$ remains a gap in~$\hat K$.
Put $E:=H(K)$ and $F:=H(\hat K)$, and note that the $H$-field extension $E\subseteq F$ is immediate by [ADH, 10.3.2 and remark preceding it]. Next take $e\in E$ with~$e\succ 1$ and~$v(e^\dagger)=\max\Psi_E=\max\Psi_F$.  
We now construct $K^*:=E_\upo$ and $\hat K^*:=F_\upo$ as in~[ADH, 11.7] with $E$,~$e$ and $F$,~$e$ in the role of $F$,~$f$ there, so $E_\upo=\bigcup_n E_n$, $F_{\upo}=\bigcup_n F_n$, $E_0=E$, $F_0=F$. We take care to do that in such a way that
by induction on $n$ using [ADH, 10.2.3 and its proof] we have for all $n$ an immediate extension $E_n\subseteq F_n$ of grounded $H$-fields with a distinguished element $e_n\in E_n^\times$ such that $e_0=e$,  
${e_n\succ 1}$, $v(e_n^\dagger)=\max\Psi_{E_n}=\max\Psi_{F_n}$,  and
$$E_{n+1}\ =\ E_n(e_{n+1}),\quad  F_{n+1}\ =\ F_n(e_{n+1}), \quad e'_{n+1}\ =\ e_n^\dagger.$$ 
This yields an immediate extension $E_{\upo}\subseteq F_{\upo}$ of $\upo$-free $H$-fields. Expanding $K^*, \hat{K}^*$ uniquely to  $\HLO$-fields gives a pair $(\mathbf K^*,\hat{\mathbf K}^*)$ with the required properties.
\end{proof}


\begin{lemma}
Suppose $K$ has gap $\beta$ and $v(b')\neq\beta$ for all $b\asymp 1$ in $K$. Then $(\mathbf K,\hat{\mathbf K})$ extends to
an immediate grounded $\HLO$-pair $(\mathbf K_1,\hat{\mathbf K}_1)$  
such that $\res\mathbf K_1 = \res\mathbf K$ and any embedding of $\mathbf K$ into a $\HLO$-field $\mathbf L$ closed under
integration extends to an embedding~$\mathbf K_1\to\mathbf L$.
\end{lemma}
\begin{proof}
Note that $\beta$ is a gap in $\hat K$, and  $v(b')\neq\beta$ for all $b\asymp 1$ in $\hat K$. By [ADH, 10.3.2 and remark preceding it], $H(K)$ is an immediate extension of $K$, and  $H(\hat K)$ of $\hat K$, so
$H(\hat K)$ is an immediate extension of $H(K)$.

Take~$s\in K$ with $vs=\beta$ and follow the proof of [ADH, 16.4.3].
Suppose~$s\notin I$. Then also~$s\notin\hat I$. Take $\hat K_1=H(\hat K)(y)$ as in Case~1 of that proof
applied to~$\hat{\mathbf K}$ in place of $\mathbf K$.   We have the $H$-subfield~$K_1:=H(K)(y)$ of~$\hat K_1$, and~$\hat K_1$ is an immediate extension of~$K_1$ by~[ADH, 10.2.2  and its proof]. Expanding 
~$K_1$,~$\hat K_1$ uniquely to
pre-$\HLO$-fields gives a pair~$(\mathbf K_1,\hat{\mathbf K}_1)$ with the required property. 
If~${s\in I}$, proceed as before, but following instead Case~2 of the proof of [ADH, 16.4.3] and with $H(K)$ and $H(\hat K)$ instead of 
$K$ and $\hat K$, using~[ADH, 10.2.1 and its proof].
\end{proof}

\begin{cor}\label{cor:16.4.4 imm}
Suppose $K$ does not have asymptotic integration. Then 
$(\mathbf K,\hat{\mathbf K})$ extends to an immediate $\upo$-free $\HLO$-pair 
$(\mathbf K^*,\hat{\mathbf K}^*)$ 
 such that  
$\res\mathbf K^* = \res\mathbf K$, and any embedding of~$\mathbf K$
into a $\HLO$-field $\mathbf L$ closed under integration extends to an embedding~$\mathbf K^*\to\mathbf L$.
\end{cor}

\noindent
In the next three lemmas we treat the case where $K$ has asymptotic integration. For the first we adapt the proof of  [ADH, 16.4.5] and use parts of it:
 
\begin{lemma}\label{lem:16.4.5 imm}
Assume $K$ has asymptotic integration and is not $\upl$-free. Then
$(\mathbf K,\hat{\mathbf K})$ extends to an immediate $\upo$-free $\HLO$-pair 
$(\mathbf K^*,\hat{\mathbf K}^*)$ 
 such that   
$\res\mathbf K^* = (\res\mathbf K)^{\operatorname{rc}}$, and every embedding
of $\mathbf K$ into a Liouville closed $\HLO$-field $\mathbf L$ extends to an embedding $\mathbf K^*\to\mathbf L$.
\end{lemma}
\begin{proof}
By Corollary~\ref{cor:16.4.4 imm} it is enough to show that~$(\mathbf K,\hat{\mathbf K})$ extends to an immediate $\HLO$-pair 
$(\mathbf K_1,\hat{\mathbf K}_1)$ with a gap such that~$\res\mathbf K_1 = (\res\mathbf K)^{\operatorname{rc}}$ and every embedding
of~$\mathbf K$ into a Liouville closed $\HLO$-field $\mathbf L$ extends to an embedding $\mathbf K_1\to\mathbf L$.
Let
$$E\ :=\ H(K)^{\operatorname{rc}}\ \subseteq\ F:=H(\hat K)^{\operatorname{rc}}.$$
Then $\Gamma_E=\Q\Gamma$, and $F$ is an immediate   $H$-field extension of $E$.
We distinguish two cases:

\case[1]{$E$ has a gap.} Take $s\in E^\times$ and $n\geq 1$ such that $vs$ is a gap in~$E$ and~$s^n\in K$.
Then $E$ has exactly two $\HLO$-cuts $(I_1,\Lambda_1,\Omega_1)$, $(I_2,\Lambda_2,\Omega_2)$,
where~$I_1=\{{y\in E:y\prec s}\}$,
$I_2=\{y\in E:y\preceq s\}$, and 
$F$ has 
exactly two $\HLO$-cuts~$(\hat I_1,\hat \Lambda_1,\hat \Omega_1)$ and $(\hat I_2,\hat \Lambda_2,\hat \Omega_2)$,
with $\hat I_1=\{y\in F:y\prec s\}$, $\hat I_2=\{{y\in F:y\preceq s}\}$ (so $I_j=\hat I_j\cap E$ for $j=1,2$).
Take $\mathbf K_1$ as in Case~1 of the proof of~[ADH, 16.4.5]:
if~$-s^\dagger\in\Lambda$, then $\mathbf K_1:=(E,I_1,\Lambda_1,\Omega_1)$,
and if $-s^\dagger\notin\Lambda$, then $\mathbf K_1:=(E,I_2,\Lambda_2,\Omega_2)$.
Similarly, if~$-s^\dagger\in \Lambda$, then $\hat{\mathbf K}_1:=(F,\hat I_1,\hat \Lambda_1,\hat \Omega_1)$,
and if $-s^\dagger\notin \Lambda$, then  $\hat{\mathbf K}_1:=(F,\hat I_2,\hat \Lambda_2,\hat \Omega_2)$.
Then $(\mathbf K_1,\hat{\mathbf K}_1)$ is an immediate $\HLO$-pair  with the desired property.

\case[2]{$E$ has no gap.} Then $E$, $F$ have asymptotic integration, and the sequence~$(\upl_\rho)$ for $K$ also serves for $E$ and for $F$.
Take $\upl\in K$ such that $\upl_\rho\leadsto\upl$. Then~$-\upl$ creates a gap over $E$ and over $F$ by [ADH, 11.5.14].
Take an element~$f\neq 0$ in some Liouville closed $H$-field extension of $F$ such that $f^\dagger=-\upl$.
Then $F(f)$ is an $H$-field  and $E(f)$ is an $H$-subfield of $F(f)$ with~$\res E(f)=\res E=\res F=\res F(f)$.
Moreover, $vf$ is a gap in $F(f)$ and in $E(f)$, and $F(f)$ is an immediate extension of~$E(f)$, by the remark after [ADH 11.5.14]
and the uniqueness part of [ADH, 10.4.5].
Now $E(f)$ has
exactly two $\HLO$-cuts $(I_1,\Lambda_1,\Omega_1)$ and $(I_2,\Lambda_2,\Omega_2)$,
where
$$I_1=\big\{{y\in E(f):y\prec f}\big\},\qquad I_2=\big\{y\in E(f):y\preceq f\big\},$$ and $F(f)$ has 
exactly two $\HLO$-cuts~$(\hat I_1,\hat \Lambda_1,\hat \Omega_1)$ and $(\hat I_2,\hat \Lambda_2,\hat \Omega_2)$,
with 
$$\hat I_1=\big\{y\in F(f):y\prec f\big\},\qquad \hat I_2=\big\{{y\in F(f):y\preceq f}\big\}.$$   
Therefore~$I_j=\hat I_j\cap E(f)$ for $j=1,2$.
We set $\mathbf K_1:=\big(E(f),I_1,\Lambda_1,\Omega_1\big)$ and $\hat{\mathbf K}_1:=\big(F(f),\hat I_1,\hat\Lambda_1,\hat\Omega_1\big)$ 
if $\upl\in\Lambda$, and $\mathbf K_1:=\big(E(f),I_2,\Lambda_2,\Omega_2\big)$, $\hat{\mathbf K}_1:=\big(F(f),\hat I_2,\hat \Lambda_2,\hat \Omega_2\big)$
 if $\upl\notin\Lambda$.
Then $\mathbf K_1\subseteq\hat{\mathbf K}_1$, and the immediate $\HLO$-pair $(\mathbf K_1,\hat{\mathbf K}_1)$ 
  is as required.
\end{proof}

\begin{lemma}
Suppose $K$ is not $\upo$-free and $\hat K$ is $\upl$-free. Then $(\mathbf K,\hat{\mathbf K})$ extends to an immediate $\upo$-free $\HLO$-pair
$(\mathbf K^*,\hat{\mathbf K}^*)$ such that $\res\mathbf K^*$ is algebraic over~$\res\mathbf K$ 
 and any embedding   $\mathbf K\to\mathbf L$ into
a Schwarz closed $\HLO$-field $\mathbf L$ extends to an embedding~$\mathbf K^*\to\mathbf L$.
\end{lemma}
\begin{proof} We adapt and use the proof of [ADH, 16.4.6]. 
Take $\upo\in K$ with $\upo_\rho\leadsto\upo$. Then
$\omega\big(\Upl(K)\big){}^\downarrow<\upo<\sigma\big(\Upg(K)\big){}^\uparrow$ and
  either   $\Omega=\omega(K)^\downarrow$ or $\Omega=K\setminus \sigma\big(\Upg(K)\big){}^\uparrow$.
Likewise with $\hat K$ in place of $K$. Also $\upo\notin \omega(K)^\downarrow,\ \upo\notin \omega(\hat K)^\downarrow$. 
There are two cases:

\case[1]{$\Omega=\omega(K)^{\downarrow}$.}
Then $\upo\notin\hat\Omega$ and so $\hat\Omega=\omega(\hat K)^{\downarrow}$.
Take a pre-$H$-field extension~$\hat K_{\upg}$ of $\hat K$
as in Case~1 of the proof of [ADH, 16.4.6] with $\hat K$ in place of~$K$.
Then $\res \hat K_{\upg}=\res \hat K=\res K$.
Put~$K_{\upg}:=K\langle\upg\rangle$, a pre-$H$-subfield of $\hat K_{\upg}$ with~$\res K_{\upg}=\res K$. Then $v\upg$ is a gap in $K_{\upg}$ and in $\hat K_{\upg}$, so 
by [ADH, 13.7.6],  $\hat K_{\upg}$ is an immediate extension of $K_{\upg}$.
Expanding $K_{\upg}$ to a pre-$\HLO$-field $\mathbf K_{\upg}$ as in Case~1 of
the proof of [ADH, 16.4.6], and similarly expanding~$\hat K_{\upg}$ to a pre-$\HLO$-field~$\hat{\mathbf K}_{\upg}$, we 
thus obtain the immediate  pre-$\HLO$-pair~$(\mathbf K_{\upg},\hat{\mathbf K}_{\upg})$ extending~$(\mathbf K,\hat{\mathbf K})$. Take an immediate $\upo$-free $\HLO$-pair~$(\mathbf K^*,\hat{\mathbf K}^*)$ extending~$(\mathbf K_{\upg},\hat{\mathbf K}_{\upg})$ as in Corollary~\ref{cor:16.4.4 imm} applied to $(\mathbf K_{\upg},\hat{\mathbf K}_{\upg})$ in place of~$(\mathbf K,\hat{\mathbf K})$. Then $(\mathbf K^*,\hat{\mathbf K}^*)$ has the required property.  

\case[2]{$\Omega=K\setminus \sigma\big(\Upg(K)\big){}^\uparrow$.}
Then $\upo\in\Omega\subseteq \hat \Omega$, so $\hat\Omega=\hat K\setminus \sigma\big(\Upg(\hat K)\big){}^\uparrow$.
As in the proof of [ADH, 16.4.6] we obtain  an immediate pre-$H$-field extension~$\hat K_{\upl}:=\hat K(\upl)$ of $\hat K$ with $\upl_\rho\leadsto\upl$ and $\omega(\upl)=\upo$. Put $K_{\upl}:=K(\upl)$, an 
immediate pre-$H$-field extension of $K$.
Expand $K_\upl$ to a pre-$\HLO$-field $\mathbf K_{\upl}$ as in Case~2 of the proof of~[ADH, 16.4.6],
and similarly expand $\hat K_{\upl}$  to a pre-$\HLO$-field $\hat{\mathbf K}_{\upl}$.
Then~$\mathbf K_{\upl}\supseteq\mathbf K$ and~${\hat{\mathbf K}_{\upl}\supseteq\hat{\mathbf K}}$,
and from $\upl\notin\Upl(K_\lambda)^\downarrow$ and $\upl\in \big(\hat K_{\upl}\setminus\Upd(\hat K_{\upl})^\uparrow\big)\cap K_{\upl}$ we obtain~$\hat{\mathbf K}_{\upl}\supseteq\mathbf K_{\upl}$.
Thus~$(\mathbf K_{\upl},\hat{\mathbf K}_{\upl})$ is an immediate  pre-$\HLO$-pair and extends~$(\mathbf K,\hat{\mathbf K})$. Take an immediate $\upo$-free $\HLO$-pair~$(\mathbf K^*,\hat{\mathbf K}^*)$ extending~$(\mathbf K_{\upl},\hat{\mathbf K}_{\upl})$  obtained from Lemma~\ref{lem:16.4.5 imm} applied to $(\mathbf K_{\upl},\hat{\mathbf K}_{\upl})$ in place of~$(\mathbf K,\hat{\mathbf K})$. Then $(\mathbf K^*,\hat{\mathbf K}^*)$ has the required property.
\end{proof}

\begin{lemma}
Suppose $\hat K$ is $\upo$-free.  Then 
$(\mathbf K,\hat{\mathbf K})$ extends to an immediate $\upo$-free $\HLO$-pair 
$(\mathbf K^*,\hat{\mathbf K}^*)$ such that  any embedding of $\mathbf K$ into
a $\HLO$-field $\mathbf L$ extends to an embedding of~$\mathbf K^*$ into $\mathbf L$.
\end{lemma}
\begin{proof}
As $\hat K$ is $\upo$-free, so is $K$. 
  By [ADH, 13.6.1],     $H(K)$ is $\upo$-free, and by [ADH, 10.3.2 and   remark~(a) before it], $H(K)$ is an immediate extension of $K$, and likewise with
 $\hat K$ in place of~$K$. Let $\mathbf K^*$, $\hat{\mathbf K}^*$ be the unique expansions of $H(K)$, $H(\hat K)$, respectively,
 to  $\HLO$-fields. Then $(\mathbf K^*,\hat{\mathbf K}^*)$ has the required properties, by the proof of [ADH, 16.4.7].
\end{proof}

\noindent
The first claim of Proposition~\ref{prop:upo-free imm} now follows.  
As to the shortness  part, one checks that $\operatorname{rank}_{\Q}(\Gamma_{K^*}/\Gamma_K)\le \aleph_0$ for $(\mathbf K^*,\hat{\mathbf K}^*)$  as constructed above,    hence if $K$ is short and $\res K$ is archimedean, then $K^*$ is short by Corollary~\ref{cor:short oag} and Lemma~\ref{lem:short ofield}. \qed

\subsection*{Immediate extensions and $\HLO$-cuts}
Let  $K\subseteq\hat K$ be an extension of  pre-$H$-fields. Given a $\HLO$-cut 
$(\hat I,\hat\Lambda,\hat\Omega)$ in $\hat K$, we obtain the $\HLO$-cut
$$(\hat I,\hat\Lambda,\hat\Omega)\cap K:=(\hat I\cap K,  \hat\Lambda\cap K, \hat\Omega\cap K)$$ in $K$.
Recall from [ADH, remark before 16.3.19] that a pre-$H$-field has at least one and at most two $\HLO$-cuts.
In the rest of this subsection we assume that  $K\subseteq\hat K$ is immediate  and~$(I,\Lambda,\Omega)$ is a
 $\HLO$-cut in $K$, and we ask when there is a 
 $\HLO$-cut~$(\hat I,\hat\Lambda,\hat\Omega)$ in $\hat K$   such that~$(I,\Lambda,\Omega)=(\hat I,\hat\Lambda,\hat\Omega)\cap K$.
 
\begin{prop}\label{prop:HLOimm}
The following are equivalent:
\begin{enumerate}
\item[(i)]
There is a $\HLO$-cut $(\hat I,\hat\Lambda,\hat\Omega)$ in $\hat K$  with 
$(I,\Lambda,\Omega)=(\hat I,\hat\Lambda,\hat\Omega)\cap K$;
\item[(ii)] $K$ is not $\upl$-free, or $K$ is $\upo$-free, or $\hat K$ is $\upl$-free, or
$\Omega\neq\omega(K)^\downarrow$.
\end{enumerate}
\end{prop}

\noindent
This is a consequence of Lemmas~\ref{lem:HLOimm, 1}--\ref{lem:HLOimm, 6} below, which also address the uniqueness of the 
$\HLO$-cut  in $\hat K$  in part~(i) of the proposition. For the next two labeled displays, let $K$ be ungrounded. 
Then by [ADH, 11.8.14] we have
 \begin{equation}\label{eq:HLOim} \Upl(K)^\downarrow\ =\  \Upl(\hat K)^\downarrow\cap K, \qquad
\Upd(K)^\uparrow\ =\  \Upd(\hat K)^\uparrow \cap K, \quad
\Upg(K)^\uparrow\ =\  \Upg(\hat K)^\uparrow \cap K,\end{equation}
and by [ADH, 11.8.14, remark before 11.8.21, and 11.8.29]:
\begin{equation}\label{eq:HLOimm}
\omega\big(\Upl(\hat K)\big){}^\downarrow\cap K\ =\ \omega\big(\Upl(K)\big){}^\downarrow, \qquad
\big(\hat K\setminus\sigma\big(\Upg(\hat K)\big){}^\uparrow\big)\cap K\ =\ K\setminus\sigma\big(\Upg(K)\big){}^\uparrow.\end{equation}
By [ADH, 11.8.2] we also have $\I(K)=\I(\hat K)\cap K$ if $K$ has asymptotic integration. 
 
\begin{lemma}\label{lem:HLOimm, 1}
Suppose $K$ does not have asymptotic integration or $\hat K$ is $\upo$-free.  Then there is a unique  $\HLO$-cut $(\hat I,\hat\Lambda,\hat\Omega)$ in $\hat K$ with~$(I,\Lambda,\Omega)=(\hat I,\hat\Lambda,\hat\Omega)\cap K$.  
\end{lemma}
\begin{proof}
Note that   $K$ has asymptotic integration iff $\hat K$ has,
and if $K$ has a gap $\beta$ and~$v(a')\neq\beta$ for all $a\asymp 1$ in $K$, then $\beta$ remains a gap in $\hat K$ and
$v(b')\neq\beta$ for all~$b\asymp 1$ in $\hat K$. If $\hat K$ is $\upo$-free, then so is $K$.
Now use~[ADH, 16.3.11--16.3.14].
\end{proof}

\begin{lemma}\label{lem:HLOimm, 4}
Suppose $K$ is $\upo$-free, but $\hat K$ is not.
Then there are exactly two  $\HLO$-cuts $(\hat I,\hat\Lambda,\hat\Omega)$ in $\hat K$ with~$(I,\Lambda,\Omega)=(\hat I,\hat\Lambda,\hat\Omega)\cap K$. 
\end{lemma}
\begin{proof}
By [ADH, 16.3.14], $(I,\Lambda,\Omega)$ is the unique $\HLO$-cut in $K$, so 
$(I,\Lambda,\Omega)=(\hat I,\hat\Lambda,\hat\Omega)\cap K$ for every $\HLO$-cut $(\hat I,\hat\Lambda,\hat\Omega)$ in $\hat K$. 
Moreover, $\hat K$  has exactly two $\HLO$-cuts, by  [ADH, 16.3.16]
if $\hat K$ is $\upl$-free, and by [ADH, 16.3.17, 16.3.18]  if not.
\end{proof}

\noindent
In particular, if $K$ has no asymptotic integration or $K$ is $\upo$-free then we have
a $\HLO$-cut $(\hat I,\hat\Lambda,\hat\Omega)$ in $\hat K$  with 
$(I,\Lambda,\Omega)=(\hat I,\hat\Lambda,\hat\Omega)\cap K$. The next lemmas deal
with the case where $K$ has asymptotic integration and $K$ is not $\upo$-free.

\begin{lemma}\label{lem:HLOimm, 2}
Suppose $K$ has asymptotic integration and  is not $\upl$-free. 
Then there is 
exactly one  $\HLO$-cut $(\hat I,\hat\Lambda,\hat\Omega)$ in $\hat K$ with~$(I,\Lambda,\Omega)=(\hat I,\hat\Lambda,\hat\Omega)\cap K$. 
\end{lemma}
\begin{proof}
Suppose first that $2\Psi$ has no supremum in $\Gamma$. 
Then by [ADH, 16.3.17] there are exactly two $\HLO$-cuts 
$(I_1,\Lambda_1,\Omega_1)$, $(I_2,\Lambda_2,\Omega_2)$
in~$K$, with $\Lambda_1=\Upl(K)^\downarrow$, $\Lambda_2=K\setminus\Upd(K)^\uparrow$, and
$\Lambda_1\neq\Lambda_2$.
Similarly there are exactly two $\HLO$-cuts 
$(\hat I_1,\hat \Lambda_1,\hat \Omega_1)$, $(\hat I_2,\hat \Lambda_2,\hat \Omega_2)$
in $\hat K$, with $\hat \Lambda_1=\Upl(\hat K)^\downarrow$, $\hat \Lambda_2=\hat K\setminus\Upd(\hat K)^\uparrow$.
Now use that by \eqref{eq:HLOim} we have  $\Upl(K)^\downarrow = \Upl(\hat K)^\downarrow\cap K$
and $K\setminus\Upd(K)^\uparrow = \big(\hat K\setminus\Upd(\hat K)^\uparrow\big)\cap K$.
The case where $2\Psi$ has a supremum in $\Gamma$ is similar, using  [ADH, 16.3.18]
instead of  [ADH, 16.3.17].
\end{proof}

\begin{lemma}\label{lem:HLOimm, 5}
Suppose $K$ is not $\upo$-free and $\hat K$ is $\upl$-free. Then 
there is 
exactly one  $\HLO$-cut $(\hat I,\hat\Lambda,\hat\Omega)$ in $\hat K$ such that~$(I,\Lambda,\Omega)=(\hat I,\hat\Lambda,\hat\Omega)\cap K$.
\end{lemma}
\begin{proof}
By [ADH, 16.3.16],  $K$ being $\upl$-free, but not $\upo$-free, it has exactly two $\HLO$-cuts, namely
$ \big(\!\I(K), \Upl(K)^{\downarrow}, \omega\big(\Upl(K)\big){}^\downarrow\big)$ and $\big(\!\I(K), \Upl(K)^{\downarrow}, K\setminus \sigma\big(\Upg(K)\big){}^\uparrow\big)$,
and similarly with $\hat K$ in place of $K$.
Now use \eqref{eq:HLOim} and \eqref{eq:HLOimm}.
\end{proof}

\begin{lemma}\label{lem:HLOimm, 6}
Suppose $K$ is $\upl$-free, but not $\upo$-free, and $\hat K$ is not $\upl$-free.
Then there is a $\HLO$-cut $(\hat I,\hat\Lambda,\hat\Omega)$ in $\hat K$ such that~$(I,\Lambda,\Omega)=(\hat I,\hat\Lambda,\hat\Omega)\cap K$ iff $\Omega\neq\omega(K)^\downarrow$, and in this case
there are exactly two such $\HLO$-cuts in $\hat K$.
\end{lemma}
\begin{proof}
By [ADH, 16.3.16]
 $K$ has exactly two $\HLO$-cuts $(I_1,\Lambda_1,\Omega_1)$, $(I_2,\Lambda_2,\Omega_2)$~
where 
$$I_1\ =\ I_2\ =\ \I(K),\quad \Lambda_1\ =\ \Lambda_2\ =\Upl(K)^{\downarrow},\quad
\Omega_1\ =\ K\setminus \sigma\big(\Upg(K)\big){}^\uparrow\ \ne\  \Omega_2\ =\  \omega(K)^{\downarrow}.$$
Now $K$ is $\upl$-free, so $2\Psi$ has no supremum in $\Gamma$ by [ADH, 9.2.17, 11.6.8], hence
by~[ADH, 16.3.17], $\hat K$ has exactly two $\HLO$-cuts $(\hat I_1,\hat\Lambda_1,\hat\Omega_1)$, 
$(\hat I_2,\hat\Lambda_2,\hat\Omega_2)$, where
$$\hat I_1= \hat I_2 = \I(\hat K),\quad \hat\Lambda_1 = \Upl(\hat K)^{\downarrow}, \quad
\hat\Lambda_2 = \hat K\setminus \Upd(\hat K)^{\uparrow}, \quad
\hat\Omega_1 = \hat\Omega_2 = \hat K\setminus \sigma\big(\Upg(\hat K)\big){}^{\uparrow}.$$
Thus $(\hat I_j,\hat\Lambda_j,\hat\Omega_j)\cap K=(I_1,\Lambda_1,\Omega_1)$ for $j=1,2$ by \eqref{eq:HLOimm}. This yields the lemma.
\end{proof}

\subsection*{Proof of Theorem~\ref{thm:embedd imm}}
Let $K$, $\hat K$, $M$ be as in the statement of the theorem,
and let $i\colon K\to M$ be an embedding.  
 If $\hat K$ is $\upo$-free, then Lemma~\ref{lem:L->M} and [ADH, 10.5.8] give an extension of $i$ to an embedding~${\hat{K}\to  M}$ as required.

In the rest of the proof we therefore assume that $\hat K$ is not $\upo$-free.
 If $\hat K$ is $\upl$-free, then, taking $\upo\in\hat K$ with~$\upo_\rho\leadsto\upo$, [ADH, 11.7.13] yields an immediate
 pre-$H$-field extension~$\hat K_{\upl}:=\hat K(\upl)$ of $\hat K$ with $\upl_\rho\leadsto\upl$ and $\omega(\upl)=\upo$, so that replacing $\hat K$ by~$\hat K_{\upl}$ we arrange  that $\hat K$ is  not even $\upl$-free.

Suppose $K$ is not $\upl$-free.
Let $\mathbf M$ be the unique expansion of $M$ to a $\HLO$-field, and expand $K$ to a pre-$\HLO$-field $\mathbf K$ such that $i$ is an embedding~$\mathbf K\to\mathbf M$ of pre-$\HLO$-fields. Proposition~\ref{prop:HLOimm} yields an expansion of $\hat K$ to a pre-$\HLO$-field $\hat{\mathbf K}$ such
that~${\mathbf K\subseteq\hat{\mathbf K}}$, and
then Proposition~\ref{prop:upo-free imm} gives an immediate $\upo$-free short $\HLO$-pair~$(\mathbf K^*,\hat{\mathbf K}^*)$ extending $(\mathbf K,\hat{\mathbf K})$ with $\res\mathbf K^*$ algebraic over $\res\mathbf K$ and an extension of $i$ to an embedding $i^*:\mathbf K^*\to\mathbf M$.  The case of $\upo$-free $\hat K$ treated earlier applied instead to
 $\hat{K}^*$ now yields an extension  of $i^*$ to an embedding $\hat{K}^*\to M$.
 
Next, suppose $K$ is $\upo$-free. 
 Then the pc-sequence $(\upl_\rho)$ in $K$ is of $\d$-tran\-scen\-den\-tal type over $K$, by [ADH, 13.6.3]. 
  Take $\upl\in\hat K$ such that~$\upl_\rho\leadsto\upl$. 
 Now $K$ is short and $M$ is countably spherically complete, so by Lemma~\ref{lem:short}  we have
$\upl^*\in M$ with~${i(\upl_\rho)\leadsto\upl^*}$.
 By [ADH, 11.4.7, 11.4.13, 10.5.8] we obtain a unique extension of~$i$
 to an embedding $j\colon K\langle\upl\rangle\to M$ such that $j(\upl)=\upl^*$.
 The case of non-$\upl$-free $K$ applied instead to $K\langle\upl\rangle$ yields an extension of $j$ to
 an embedding~$\hat K\to M$. \qed

\section{Embeddings into Analytic Hardy Fields}\label{sec:embedd}

\noindent
In this section we use Theorem~\ref{aneta} to derive results about
back-and-forth equivalence, $\infty\omega$-elementary equivalence, and isomorphism for maximal analytic Hardy fields, as was done in   \cite[Section~10]{ADHfgh} for maximal Hardy fields.
(For the relation of back-and-forth equivalence to infinitary logic,  see \cite{Barwise}.)
We also strengthen Corollary~\ref{cor:Tlog} by showing in Corollary~\ref{cor:embed T} that the ordered differential field $\T$ embeds into every maximal analytic Hardy field. 

Let~$\No$ be the ordered field of surreal numbers     equipped with
the derivation $\der_{\BM}$ of Berarducci and Mantova~\cite{BM}. Then $\No$ is a closed $H$-field, by \cite{ADH1+}. Moreover, given an  uncountable cardinal $\kappa$,  the surreal
numbers of   length~$<\kappa$ form an ordered differential subfield~$\No(\kappa)$ of~$\No$
with $\No(\kappa)\preceq\No$, by~\cite[Corollary~4.6]{ADH1+}.
As in the argument leading up to \cite[Corollary~10.4]{ADHfgh}, combining Theorem~\ref{aneta} 
and
 \cite[Corollary~10.3]{ADHfgh} yields:
 
\begin{cor}
Let $M$ be a maximal analytic or maximal smooth Hardy field. Then the ordered differential fields $M$ and  $\No(\omega_1)$ are back-and-forth equivalent. Hence~$M \equiv_{\infty\omega} \No(\omega_1)$, and
assuming~$\operatorname{CH}$,   $M\cong \No(\omega_1)$. 
\end{cor}

\noindent
The ordered field $\No(\omega_1)$
is not complete:  Set $a_{\nu}:=\sum_{\mu< \nu} \omega^{-\mu}$ with $\mu$, $\nu$ ranging over countable ordinals. Then $(a_{\nu})$ is a cauchy sequence in $\No(\omega_1)$ without a limit in
$\No(\omega_1)$. Thus, assuming CH, no real closed $\eta_1$-ordered  field extension of $\R$ of cardinality $\mathfrak{c}$ is complete, 
in particular, no maximal Hardy field is complete.  
(This also follows from \cite[Theorem~3.12(ii)]{DW}: if $G$ is a complete $\eta_1$-ordered abelian group, then $\abs{G}>\aleph_1$.)

\medskip\noindent
Let $K$ be an    $H$-field  with small derivation and constant field~$\R$.
Then \cite[Theorem~3]{ADH1+} yields an embedding $K\to\No$ 
of ordered differential fields.
The argument in the proof of~\cite[Theorem~3]{ADH1+} shows that if $\kappa>\abs{K}$ is a regular  cardinal, then we can choose~$\iota$ so that~$\iota(K)\subseteq\No(\kappa)$.
If  $\operatorname{trdeg}(K|\R)$ is countable, then~$K$
actually embeds into~$\No(\omega_1)$.
This is a consequence of the next lemma, a variant of \cite[Lemma~10.1]{ADHfgh}.

\begin{lemma}\label{iso1 variant, 1}
Let $K$ be a pre-$H$-field with very small derivation, archimedean residue field, and $\operatorname{trdeg}(K|C)\le \aleph_0$. Let
 $L$ be a closed $\eta_1$-ordered $H$-field with small derivation and $C_L=\R$.
Then $K$ embeds into~$L$.
\end{lemma}
  \begin{proof} 
Passing to $H(K)$ we arrange that $K$ is an $H$-field. Without loss, $C_K$ is an ordered subfield of $\R$, and then adjoining new constants if necessary, we arrange~$C_K=\R$.
Take a closed $H$-field $\hat K$ extending $K$. Next, take  a countable set~$S\subseteq K$ such that~$K=\R\langle S\rangle$ and then
a countable closed $H$-subfield $K_0\supseteq S$ of~$\hat K$.   Let $E$ be a copy  of the prime model  of the theory of
  closed $H$-fields with small derivation  inside $K_0$.
  Applying  \cite[Lemma~10.1]{ADHfgh} to an $H$-field embedding~$E\to L$  with $K_0$ in place of $K$
      yields an   $H$-field embedding~$i\colon K_0\to L$. Then~$i$ is the identity on $C_{K_0}\subseteq \R$, and then   [ADH, 10.5.15, 10.5.16]  yield an extension
of $i$ to an $H$-field embedding $K_0(\R) \to L$ that is the
identity on~$\R$, and the restriction of this embedding to $K$ is a pre-$H$-field embedding $K\to L$.
\end{proof}


\noindent
Using also Theorem~\ref{aneta} and its smooth version we obtain from Lemma~\ref{iso1 variant, 1}:

\begin{cor} 
Let  $K$ be as in Lemma~\ref{iso1 variant, 1} and let $M$ be a maximal Hardy field. Then~$K$  embeds into~$M$.
Likewise if $M$ is a maximal analytic Hardy field or a maximal smooth Hardy field.
\end{cor}

\noindent
The following immediate consequence of the last corollary is worth recording:

\begin{cor}
Every Hardy field of countable transcendence degree over its constant field is
isomorphic to an analytic Hardy field.
\end{cor}

\noindent
The next corollary strengthens \cite[Corollary~12.4]{ADH4}:

\begin{cor}\label{cor:inftyomegaelemsub}
Let $M$ be a maximal analytic or maximal smooth Hardy field, and let $N$ be a maximal Hardy field with $M\subseteq N$. Then $M \preceq_{\infty\omega} N$.
\end{cor}
\begin{proof}
By Theorem~\ref{aneta} and its smooth version, $M$ is $\eta_1$, and by Theorem~A of \cite{ADHfgh}, $N$~is~$\eta_1$. It remains to use \cite[Lemma~10.5]{ADHfgh}.
\end{proof}

\noindent
At the heart of the proof of \cite[Lem\-ma~10.1]{ADHfgh} is [ADH, 16.2.3] of which we now give a version
with the cofinality hypothesis replaced by a shortness assumption:

\begin{prop}\label{prop:16.2.3 short}
Let $E$ be an $\upo$-free $H$-field and $K$ be a closed short $H$-field extending~$E$ such that~$C_{E}=C_K$.
Let $i\colon E\to L$ be an embedding where $L$ is a closed $\eta_1$-ordered $H$-field. Then $i$ extends to an embedding~$K\to L$.
\end{prop}
\begin{proof}
Suppose $E\neq K$;
it is enough to show that $i$ extends to an embedding of some $\upo$-free $H$-subfield $F$ of $K$ into $L$, where $F$ properly contains $E$.
 
  Consider first the case~$\Gamma_E^<$ is not cofinal in $\Gamma^<$.
 Then we have $y\in K^>$ such that~$\Gamma_E^< < vy < 0$. 
 Now $E$ is short, so we have
$y^*\in L^>$ such that~$\Gamma^{<}_{iE} < vy^* < 0$.
 As in the proof of [ADH, 16.2.3] we then obtain an $\upo$-free $H$-subfield $F$ of $K$ with~$F\supseteq  E\langle y\rangle$ and 
 an extension of $i$ to  an embedding~$F\to L$.
 
 For the rest of the proof we assume $\Gamma_E^<$ is cofinal in $\Gamma^<$. Then every differential subfield of $K$ containing $E$ is
 an $\upo$-free $H$-subfield of $K$.
 
 \subcase[1]{$E$ is not closed.}
 This goes like Subcase~1 in the proof of~[ADH, 16.2.3].
 
 \subcase[2]{$E$ is closed, and $E\langle y\rangle$ is an immediate extension of $E$ for some $y\in K\setminus E$.}
 For such $y$, Lemma~\ref{lem:L->M} yields an extension of $i$ to an embedding $E\langle y\rangle\to L$.
 
\subcase[3]{$E$ is closed, and there is no $y\in K\setminus E$ such that  $E\langle y\rangle$ is an immediate extension of $E$.}
Take any $f\in K\setminus E$.
Since $E$ is short  and $L$ is $\eta_1$ we have 
$g\in L$ such that for all $a\in E$, $a<f\Leftrightarrow i(a)<g$. Now [ADH, 16.1.5] 
gives  an $H$-field embedding $E\langle f\rangle\to L$ extending $i$ which sends $f$ to $g$.
\end{proof}

\begin{cor}\label{cor:16.2.3 short}
Let $E$, $K$, $L$, $i$ be as in Proposition~\ref{prop:16.2.3 short}, with ``$C_E=C_K$'' replaced by
``$C_K$ is archimedean and $C_L=\R$".  Then $i$ extends to an embedding~$K\to L$.
\end{cor}
\begin{proof}
The ordered field embedding $i|_{C_E}\colon C_E\to C_L=\R$ extends uniquely to an ordered field embedding
$j\colon C_K\to C_L$. Now argue as in the proof of [ADH, 16.2.4], using Proposition~\ref{prop:16.2.3 short} in place of [ADH, 16.2.3].
\end{proof}

\noindent
Proposition~\ref{prop:16.2.3 short} leads to a version of   Lemma~\ref{iso1 variant, 1} for short closed $H$-fields:

\begin{lemma}\label{lem:embed short H-fields} 
Let $K$ be a  closed short $H$-field with small derivation and archimedean constant field, and $L$  a closed $\eta_1$-ordered $H$-field with small derivation and $C_L=\R$. Then $K$ embeds into $L$.
\end{lemma}
\begin{proof}
Take $x\in K$ with $x'=1$. Now $K$ has small derivation,  so~${x\succ 1}$,  the $H$-field~$C_K(x)$ is grounded, and we have an
embedding $i\colon C_K(x)\to L$ extending the unique ordered field embedding $C_K\to C_L$.  By  [ADH, 10.6.23]  we have a Liouville closure $E$ of~$C_K(x)$ in $K$ and $i$ extends to an embedding $E\to L$. Moreover, $E$ is $\upo$-free, by \cite[Lemma~1.3.18]{ADH6}, so we can use
Proposition~\ref{prop:16.2.3 short}.
\end{proof}

\noindent
With $K=\T$ and $L$ a maximal analytic Hardy field in Lemma~\ref{lem:embed short H-fields}  we conclude:

\begin{cor}\label{tan}
The ordered differential field $\T$ is isomorphic over $\R$ to an analytic Hardy field containing $\R$.
\end{cor}  

\noindent 
We upgrade this as follows:

\begin{cor}\label{cor:embed T}
Let $E$ be a pre-$H$-subfield of $\T$, $M$ be a maximal Hardy field, and~${i\colon E\to M}$ be an embedding.
Then $i$ extends to an embedding $\T\to M$. Likewise with ``maximal analytic'' and with ``maximal smooth'' instead of ``maximal''.
\end{cor}
\begin{proof} 
Expand $E$, $M$ (uniquely) to pre-$\HLO$-fields $\mathbf E$, $\mathbf M$, respectively, such that $i$ is an embedding $\mathbf E\to\mathbf M$, and expand $\T$ (uniquely) to a pre-$\HLO$-field $\mathbf T$. Then~${\mathbf E\subseteq\mathbf T}$ by \cite[Lemma~12.1, Corollary~12.9]{ADH4}. Now [ADH, 16.4.1] yields an $\upo$-free $\HLO$-field~$\mathbf E^*$ with~$\mathbf E\subseteq\mathbf E^*\subseteq\mathbf T$ and an extension of $i$ to an embedding $E^*\to M$,
which in turn extends  to an embedding $\T\to M$ by
Corollary~\ref{cor:16.2.3 short}.
\end{proof}

\begin{remark} If $\hat\T$ is an immediate $H$-field extension of~$\T$, then any
embedding of $\T$ into a maximal Hardy field $M$ extends
to an embedding~$\hat\T\to M$, by Theorem~\ref{thm:embedd imm}.  Likewise for $M$ a maximal smooth or maximal analytic Hardy field. With $\No$ in place of $M$ we can also take strong additivity into account.
To see this   recall from~\cite[Proposition~5.1 and subsequent remarks]{ADH1+}  that the unique strongly additive embedding~$\iota\colon\T\to\No$ over $\R$ of exponential ordered fields which sends $x\in\T$ to~$\omega\in\No$
is also an embedding of differential fields, with~${\iota(\T)\subseteq\No(\omega_1)}$ by
\cite[Pro\-po\-si\-tion~5.2(3)]{ADH1+}.  
By \cite[Proposition~5.2(1)]{ADH1+}, $\iota(G^{\operatorname{LE}})=\fM\cap\iota(\T)$, where~$G^{\operatorname{LE}}$ is the group of LE-monomials (cf.~[ADH, p.~718]) and
$\fM$ is the class of monomials in~$\No$ (cf.~\cite[\S{}1]{ADH1+}), hence
$\iota$ extends uniquely to a strongly additive ordered field embedding~$\hat\iota\colon\R[[G^{\operatorname{LE}}]]\to\No$. The derivation of $\No$ is strongly additive, so~${\hat\iota(\R[[G^{\operatorname{LE}}]])=\R[[\iota(G^{\operatorname{LE}})]]}$ is
a differential subfield of~$\No$. The derivation on~$\R[[G^{\operatorname{LE}}]]$
that makes~$\hat\iota$ a differential field embedding is 
then the  unique strongly additive  derivation on~$\R[[G^{\operatorname{LE}}]]$ extending the derivation of~$\T$. It also makes $\R[[G^{\operatorname{LE}}]]$ a spherically complete immediate $H$-field extension of $\T$, so the result stated at the beginning of this extended remark applies to the $H$-field $\R[[G^{\operatorname{LE}}]]$ in the role of $\hat\T$. 
\end{remark}



\section{Some Set-Theoretic Issues}  

\noindent
We finish with some questions of a set-theoretic nature that others might be better prepared to answer. We assume our base theory ZFC is consistent, and these are questions about relative consistency with ZFC.

\begin{enumerate}
\item Is it consistent that there are non-isomorphic maximal Hardy fields?
\item Is it consistent that no maximal Hardy field is isomorphic to $\No(\omega_1)$?
\item Is it consistent that there is a complete maximal Hardy field? 
\end{enumerate}
Positive answers would mean (at least) that we cannot drop the assumption CH in some results we proved under this hypothesis. 
Note also that with CH we have~$\cf(H)=\ci(H^{>\R})=\omega_1$ for all maximal Hardy fields. This suggests:

\begin{enumerate}
\item[(4)]   Is it consistent that $\cf(H_1)\ne \cf(H_2)$ for some maximal Hardy fields~$H_1$,~$H_2$? 
Same with $\ci(H_i^{>\R})$ instead of $\cf(H_i)$.  
\item[(5)]  Is it consistent that $\cf(H)\ne \ci(H^{>\R})$ for some maximal Hardy field $H$?
\item[(6)] Is it consistent that there is a maximal Hardy field $H$ and a gap in $H$
of character $(\alpha,\beta^*)$ with $\alpha,\beta\geq\omega$, not equal to one of
 $(\omega,\kappa^*)$, $(\kappa,\omega^*)$,   $(\kappa,\kappa^*)$, $(\lambda,\lambda^*)$,
where $\kappa:=\ci(H^{>\R})$, $\lambda:=\cf(H)$?
\end{enumerate}
One can also ask these questions for maximal analytic Hardy fields and maximal smooth Hardy fields instead of maximal Hardy fields. We can even ask them for maximal Hausdorff fields (containing at least $\R$, say) instead of maximal Hardy fields.
As with Corollary~\ref{CHcof}, might some weaker assumption like $\mathfrak{b}=\mathfrak{d}$ be enough for some results where we assumed CH?

If $H$ is a Hardy field with $H^{>\R}$ closed under compositional inversion, then 
$$h\mapsto h^{\inv}\colon H^{>\R}\to H^{>\R}$$ is a strictly decreasing bijection, so $\operatorname{cf}(H)=\operatorname{ci}(H^{>\R})$. However, we don't know if there is a maximal Hardy field $H$ with $H^{>\R}$ closed under
compositional inversion.


\addappendix[A Proof of Whitney's Approximation Theorem]

\noindent
For the convenience of the reader, we include here a proof of Theorem~\ref{thm:WAP dim 1, general},
adapting the exposition in \cite[\S{}1.6]{Nara}. 
Throughout this appendix
  $r\in\N\cup\{\infty\}$ and~$a,b\in \R$. 

\medskip
\noindent
Recall that the support $\supp f$ of a function~$f\colon \R\to \R$ is the closure in
$\R$ of the set~${\big\{t\in \R:\, f(t)\ne0\big\}}$. 
We begin with two lemmas, 
where $f\in\Cc^m(\R)$ is such that~$\supp f$ is bounded; let also $\lambda$ range over $\R^>$.
From the Gaussian integral~$\int_{-\infty}^\infty \ex^{-s^2}\,ds=\pi^{1/2}$ 
we get~$(\lambda/\pi)^{1/2}\int_{-\infty}^\infty \ex^{-\lambda s^2}\,ds=1$.
Consider~${f_\lambda\colon\R\to\R}$ given by
\begin{equation}\label{eq:flambda}
f_\lambda(t) := (\lambda/\pi)^{1/2}\int_{-\infty}^\infty f(s)\ex^{-\lambda (s-t)^2}\,ds.
\end{equation}
Note that  we could have replaced here the bounds $-\infty$, $\infty$ in this integral by any~$a$,~$b$    such that $\supp(f)\subseteq [a,b]$. A change of variables gives
$$f_\lambda(t)\  =\  (\lambda/\pi)^{1/2}\int_{-\infty}^\infty f(t-s)\ex^{-\lambda s^2}\,ds.$$
As in \cite[(8.12), Exercise~2(b)]{D} one   obtains that 
$f_\lambda\in\Cc^\infty(\R)$   and
for $k\leq m$:
$$f_\lambda^{(k)}(t)= 
(\lambda/\pi)^{1/2}\int_{-\infty}^\infty f^{(k)}(s) \ex^{-\lambda (s-t)^2}\,ds=
(\lambda/\pi)^{1/2}\int_{-\infty}^\infty f^{(k)}(t-s) \ex^{-\lambda s^2}\,ds.$$
Moreover:

\begin{lemma}\label{lem:WAP aux 1}
$f_\lambda$ extends to an entire function; in particular,
$f_\lambda\in\Cc^\omega(\R)$.
\end{lemma}
\begin{proof}
Take $a< b$  such that $\supp f\subseteq[a,b]$ and consider $g\colon [a,b]\times\mathbb C\to \mathbb C$
given by~$g(s,z):=f(s)\ex^{-\lambda(s-z)^2}$. Then $g$ is continuous,  for each $s\in [a,b]$ the function~$g(s,{-})\colon\mathbb C\to\mathbb C$ is analytic,  and $\partial g/\partial z\colon [a,b]\times\mathbb C\to\mathbb C$ is
continuous. Hence~$z\mapsto \int_a^b g(s,z)\,ds\colon\mathbb C\to\mathbb C$ is analytic by \cite[(9.10), Exercise~3]{D}.
\end{proof}

\begin{lemma}\label{lem:WAP aux 2}
$\dabs{f_\lambda-f}_m\to 0$ as $\lambda\to\infty$.
\end{lemma}
\begin{proof}
For $k\leq m$ we have
\begin{align*}
f_\lambda^{(k)}(t)-f^{(k)}(t)\  &=\  (\lambda/\pi)^{1/2}\int_{-\infty}^\infty \big(f^{(k)}(t-s)-f^{(k)}(t)\big)\ex^{-\lambda s^2}\,ds \\
&=\ (\lambda/\pi)^{1/2}\int_{-\infty}^\infty \big(f^{(k)}(s)-f^{(k)}(t)\big)\ex^{-\lambda (s-t)^2}\,ds.
\end{align*}
Let $\varepsilon\in\R^>$ be given, and choose $\delta>0$ such that
$$\abs{f^{(k)}(s)-f^{(k)}(t)} \leq \varepsilon/2\quad\text{whenever $\abs{s-t}\leq \delta$ and $k\leq m$.}$$
For $s\le t-\delta$ and for $s\ge t+\delta$ we have  $\ex^{-\lambda(s-t)^2}\le \ex^{-(\lambda/2)\delta^2} \ex^{-(\lambda/2)(s-t)^2}$, so
\begin{align*}
\int_{-\infty}^{t-\delta} \ex^{-\lambda(s-t)^2}\,ds + \int_{t+\delta}^{\infty} \ex^{-\lambda(s-t)^2}\,ds &\ \leq\ 
\ex^{-(\lambda/2)\delta^2}\int_{-\infty}^\infty  \ex^{-(\lambda/2)(s-t)^2}\,ds \\ &\ =\ \ex^{-(\lambda/2)\delta^2} (2\pi/\lambda)^{1/2}.
\end{align*}
Set $M:=\dabs{f}_{m}\in\R^\geq$. For $k\leq m$ we have
$$
\int_{-\infty}^\infty \big(f^{(k)}(s)-f^{(k)}(t)\big)\ex^{-\lambda (s-t)^2}\,ds = 
\int_{-\infty}^{t-\delta}(\dots)\,ds + 
\int_{t-\delta}^{t+\delta} (\dots)\,ds + 
\int^{\infty}_{t+\delta}  (\dots)\,ds,$$
hence
\begin{align*}
\abs{f_\lambda^{(k)}(t)-f^{(k)}(t)} &\ \leq\  (\lambda/\pi)^{1/2} \left( 
M \int_{-\infty}^{t-\delta} \ex^{-\lambda(s-t)^2}\,ds + {} \right. \\ 
 &\qquad\qquad\qquad\quad \left.  (\varepsilon/2)\int_{-\infty}^\infty \ex^{-\lambda(s-t)^2}\,ds + M \int^{\infty}_{t+\delta} \ex^{-\lambda(s-t)^2}\,ds\right) \\
 &\ \leq \ (\varepsilon/2)+\sqrt{2} M \ex^{-(\lambda/2)\delta^2}.
 \end{align*}
Thus if $\lambda$ is so large that $\sqrt{2} M \ex^{-(\lambda/2)\delta^2}\leq \varepsilon/2$, then $\dabs{f_\lambda-f}_m\leq\varepsilon$.
 \end{proof}

\noindent
In the next lemma we let   $U\subseteq\R$ be nonempty and open and let
$K$ range over nonempty compact subsets of $U$ and $m$ over the natural numbers~$\leq r$.

\begin{lemma}\label{lem:WAP, cauchy}
Let  $(f_n)$ be a sequence in $\Cc^r(U)$ which,
for all $K$, $m$,  is a cauchy sequence with respect to $\dabs{\,\cdot\,}_{K;\,m}$. Then there exists $f\in\Cc^r(U)$ such that for all~$K$,~$m$
we have $\dabs{f_n-f}_{K;\,m}\to 0$ as $n\to \infty$.
\end{lemma}
\begin{proof}
For all $K$, $m$, $(f_n^{(m)})$ is a cauchy sequence with respect to $\dabs{\,\cdot\,}_K$. Hence for each $m$ we obtain an $f^m\in\Cc(U)$ such that for all $K$, $\dabs{f_n^{(m)}-f^m}_K\to 0$ as~${n\to\infty}$; cf.~\cite[(7.2.1)]{D}.
Set $f:=f^0$. By induction on $m\leq r$ we show that~${f\in\Cc^m(U)}$ and~$f^{(m)}=f^m$. This is clear for $m=0$,
so suppose $0<m\leq r$ and $f\in\Cc^{m-1}(U)$, $f^{(m-1)}=f^{m-1}$.
Let $a\in U$, and take~$\varepsilon>0$   such that~$K:=[a-\varepsilon,a+\varepsilon]\subseteq U$. 
Let~$t\in K\setminus\{a\}$.
Then for each $n$ we have $s_n$ with $\abs{a-s_n}\leq \abs{a-t}$   such that
$$f_n^{(m-1)}(t)-f_n^{(m-1)}(a)\ =\ f_n^{(m)}(s_n)\cdot (t-a).$$
Take a subsequence $(s_{n_k})$ of $(s_n)$ and $s=s(t)$ with~$\lim\limits_{k\to\infty} s_{n_k} = s$.
Then~${\abs{a-s}}\leq\abs{a-t}\le \varepsilon$ and   
$$\lim_{k\to\infty} \big(f_{n_k}^{(m-1)}(t)-f_{n_k}^{(m-1)}(a)\big) = f^{m-1}(t)-f^{m-1}(a) = f^{(m-1)}(t)-f^{(m-1)}(a)$$
and $\lim\limits_{k\to\infty} f_{n_k}^{(m)}(s_{n_k})=f^m(s)$, since $\lim\limits_{n\to\infty}\dabs{f_n^{(m)}-f^m}_{K} = 0$. Hence
$$f^{(m-1)}(t)-f^{(m-1)}(a)=f^m(s)\cdot (t-a)$$
where $f^m\big(s(t)\big)\to f^m(a)$   as $t\to a$, since $f^m$ is continuous at $a$.
\end{proof}

\noindent
We   now prove Theorem~\ref{thm:WAP dim 1, general}.
Let $(a_n)$, $(b_n)$, $(\varepsilon_n)$ be sequences in $\R$ and $(r_n)$
in~$\N$ such that~$a_0= b_0$, $(a_n)$ is strictly decreasing, $(b_n)$ is strictly increasing, and $\varepsilon_n>0$, $r_n\le r$ for all~$n$. Set $I:=\bigcup_n K_n$, where $K_n:=[a_n, b_n]$, and let $f\in \Cc^r(I)$.
\textit{We need to show the existence of
a~$g\in\Cc^\omega(I)$ such that 
$\dabs{{f-g}}_{K_{n+1}\setminus K_n;\,r_n}<\varepsilon_n$ for each~$n$.}\/
Replacing  $\varepsilon_n$ by $\min\{\varepsilon_n, \frac{1}{n+1}\}$ and  $r_n$ by $\max\{r_0,\dots,r_n\}$ 
we first arrange that $\varepsilon_n\to 0$
as~$n\to\infty$ and  $r_n\leq r_{n+1}$ for all $n$. Set
$$L_n\ :=\ K_{n+1}\setminus K_n\ =\ [a_{n+1},a_n)\cup(b_n,b_{n+1}],$$
and take $\varphi_n\in\Cc^\infty(\R)$   such that $\varphi_n=0$ on a neighborhood of $K_{n-1}$ (satisfied automatically for $n=0$, by convention),  
$\varphi_n=1$ on a neighborhood of~$\operatorname{cl}(L_n)=[a_{n+1},a_n]\cup[b_n,b_{n+1}]$, and~$\supp\varphi_n\subseteq K_{n+2}$. For example, for $n\ge 1$, $\alpha_{a,b}\in\Cc^\infty(\R)$  as in
\cite[(3.4)]{ADHfgh},  and sufficiently small positive~$\varepsilon=\varepsilon(n)$, set
$$\alpha_n(t):=\begin{cases}
\alpha_{a_{n+2}+\varepsilon,a_{n+1}-\varepsilon}(t) &\text{if $t\leq a_n$,} \\
1-\alpha_{a_{n}+\varepsilon,a_{n-1}-\varepsilon}(t) &\text{otherwise.}
\end{cases}$$
and
$$\beta_n(t):=\begin{cases}
\alpha_{b_{n-1}+\varepsilon,b_n-\varepsilon}(t) &\text{if $t\leq b_n$,} \\
1-\alpha_{b_{n+1}+\varepsilon,b_{n+2}-\varepsilon}(t) &\text{otherwise.}
\end{cases}$$
and put $\varphi_n:=\alpha_n+\beta_n$.  (See Figure~\ref{fig:beta_n}.)

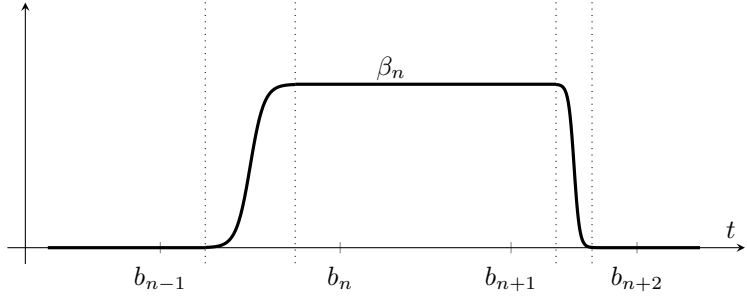
\begin{figure}[ht]
\begin{tikzpicture}
\def\anmone{1.5};\def\an{3.5};\def\anplone{5.4};\def\anpltwo{6.8};\def\epsl{0.5};
  \begin{axis} [axis lines=center, xmin=-0.2, xmax=8, ymin = -0.01, ymax = 0.15, width=0.9\textwidth, height = 0.4\textwidth, xlabel={$t$}, xtick={\anmone,\an,\anplone,\anpltwo}, ytick=\empty, xticklabels={\strut $b_{n-1}$, \strut $b_n$, \strut $b_{n+1}$, \strut $b_{n+2}$},
   ]
   \addplot [domain=0.25:\anmone+\epsl, smooth, very thick] {0}; 

    \draw[dotted] (axis cs: \anmone+\epsl,-0.25) -- (axis cs: \anmone+\epsl,2);
    \draw[dotted] (axis cs: \an-\epsl,-0.25) -- (axis cs: \an-\epsl,2);
    \draw[dotted] (axis cs: \anplone+\epsl,-0.25) -- (axis cs: \anplone+\epsl,2);
    \draw[dotted] (axis cs: \anpltwo-\epsl,-0.25) -- (axis cs: \anpltwo-\epsl,2);
      
    \addplot [domain=\an-\epsl:\anplone+\epsl, smooth, very thick] {0.1};  
   \addplot [domain=\anplone+2*\epsl:7.5, smooth, very thick] {0}; 

    \def\transanmonean{(6.5*(((2*(x-\an+\epsl)/((\an)-(\anmone)-2*\epsl))+1)))};
    \def\transanplone{(6.5*(((2*(x-\anpltwo+\epsl)/((\anpltwo)-(\anplone)-2*\epsl))+1)))};

    \addplot [domain=\anmone+\epsl:\an-\epsl, smooth, very thick] { 0.1/(1+exp(-\transanmonean)) };
    \addplot [domain=\anplone+\epsl:\anplone+2*\epsl, smooth, very thick] { 0.1/(1+exp(\transanplone)) };
        \node[right] at (axis cs:\an+0.3,0.11) {$\beta_n$};

 \end{axis}
\end{tikzpicture}
\caption{The hump function $\beta_n$}\label{fig:beta_n}
\end{figure}

\medskip
\noindent
 With $M_n:=1+2^{r_n}\dabs{\varphi_n}_{r_n}$,   choose~$\delta_n\in\R^>$ so that for all $n$,
\begin{equation}\label{eq:WAP, 0}
2\delta_{n+1} \leq \delta_n,\qquad \sum_{m = n}^\infty \delta_m M_{m+1}\leq \varepsilon_n/4
\end{equation}
Given $g\in\Cc(\R)$ with  bounded support and $\lambda\in\R^>$,  let $I_\lambda(g):=g_\lambda$
be as in \eqref{eq:flambda}, with $g$ in place of $f$. Next, let $f\in \Cc^r(I)$ be given. Then
we inductively define   sequences $(\lambda_n)$ in~$\R^>$ and~$(g_n)$ in~$\Cc^\omega(\R)$ as follows: Let $\lambda_m\in \R^{>}$ and $g_m\in \Cc^{\omega}$ for~$m<n$; then
consider the function $h_n\in\Cc^r(\R)$ given by
$$h_n(t)\ :=\ \begin{cases} \varphi_n(t)\cdot \big(f(t)-\big(g_0(t)+\cdots+g_{n-1}(t)\big)\big)
& \text{if $t\in I$,} \\
0 & \text{otherwise.}\end{cases}$$ 
Thus    $\supp h_n\subseteq\supp\varphi_n$ is bounded. Put
  $g_n:=I_{\lambda_n}(h_n)\in\Cc^\omega(\R)$ where we take~$\lambda_n\in\R^>$ such that~$\dabs{g_n-h_n}_{r_n}<\delta_n$ (any sufficiently large $\lambda_n$ will do, by Lemma~\ref{lem:WAP aux 2}).
So $\dabs{g_{n+1}-h_{n+1}}_{K_{n};\,r_{n+1}}<\delta_{n+1}$, and since $\varphi_{n+1}$ and thus also $h_{n+1}$ vanish on a neighborhood of~$K_{n}$, this yields
\begin{equation}\label{eq:WAP, 1}
\dabs{g_{n+1}}_{K_{n};\, r_{n+1}}<\delta_{n+1}.
\end{equation}
Likewise, since $\varphi_n=1$ on a neighborhood of $\operatorname{cl}(L_n)$,
\begin{equation}\label{eq:WAP, 2}
\dabs{f-(g_0+\cdots+g_n)}_{L_n;\, r_n}<\delta_n.
\end{equation}
Also $$\dabs{g_{n+1}-h_{n+1}}_{L_n;\,r_{n}}\leq\dabs{g_{n+1}-h_{n+1}}_{r_{n+1}}<\delta_{n+1},$$ and thus by \eqref{eq:prod norm} and
\eqref{eq:WAP, 2}:
\begin{align*}
\dabs{g_{n+1}}_{L_n;\,r_{n}}\ 	&\ \leq\ \dabs{g_{n+1}-h_{n+1}}_{L_n;\,r_{n}}+\dabs{\varphi_{n+1}\cdot(f-(g_0+\cdots+g_n))}_{L_n;\,r_{n}}
 \\
&\ \leq\  \delta_{n+1} + 2^{r_n}\dabs{\varphi_{n+1}}_{L_n;\,r_{n}} \cdot \dabs{f-(g_0+\cdots+g_n)}_{L_n;\,r_{n}} \\
&\ \leq\ \delta_{n+1}+ M_{n+1}\delta_n.
\end{align*}
Moreover, by \eqref{eq:WAP, 1} and $r_n\leq r_{n+1}$ we have $\dabs{g_{n+1}}_{K_n;\, r_n}<\delta_{n+1}$.
Hence by \eqref{eq:WAP, 0}:
\begin{equation}\label{eq:WAP, 3}
\dabs{g_{n+1}}_{K_{n+1};\,r_n}\ \leq\ \delta_{n+1}+ M_{n+1}\delta_n+\delta_{n+1}\ \leq\ M_{n+1}\delta_n+\delta_n\ \leq\ 2\delta_n M_{n+1}.
\end{equation}
Let $K\subseteq I$ be nonempty and compact, and let $m\leq r_n$ for some $n$.
We  claim that~$({g_0+\cdots+g_i})$ is a cauchy sequence with respect to $\dabs{\,\cdot\,}_{K;\,m}$.
To see this, 
let~$\varepsilon\in\R^>$ be given, and
take~$n$ such that $K\subseteq K_{n+1}$, $m\leq r_n$, and $\varepsilon_n\le 2\varepsilon$. 
Then by
\eqref{eq:WAP, 0} and
\eqref{eq:WAP, 3} we have  for~$j > i\geq n$:
\begin{align*}
\dabs{g_{i+1}+\cdots+g_j}_{K;\,m} &\ \leq\ 
\dabs{g_{i+1}}_{K;\,m}+\cdots+\dabs{g_j}_{K;\,m} \\ 
&\ \leq\ \dabs{g_{i+1}}_{K_{i+1};\,r_i}+\cdots+\dabs{g_j}_{K_{j};\,r_{j-1}} \\
&\ \leq\ 2\delta_iM_{i+1}+\cdots+2\delta_{j-1}M_j\ \leq\ \varepsilon_i/2\ \leq\ \varepsilon.
\end{align*}
So Lemma~\ref{lem:WAP, cauchy}  yields a function $g\colon I\to\R$ such that
$g(t)=\sum\limits_{i=0}^\infty g_i(t)$ for all~$t\in I$ and $g\in \Cc^{r_n}(I)$ for all $n$. 
In the same way, using \eqref{eq:WAP, 0} and \eqref{eq:WAP, 3} and denoting the restriction of $g_i$ to $I$ also by $g_i$, we obtain
$$\dabs{g-(g_0+\cdots+g_n)}_{L_n;\,r_n}\ =\ \left\| \sum_{i=n+1}^\infty g_i\right\|_{L_n;\,r_n}\ \leq\ \varepsilon_n/2$$
and hence by \eqref{eq:WAP, 0} and \eqref{eq:WAP, 2}:
\begin{align*}
\dabs{f-g}_{L_n;\,r_n} &\ \leq\ \dabs{f-(g_0+\cdots+g_n)}_{L_n;\,r_n} + 
\dabs{g-(g_0+\cdots+g_n)}_{L_n;\,r_n}\\
&\ \leq\ \delta_n+\textstyle\frac{1}{2}\varepsilon_n\ <\ \varepsilon_n.
\end{align*}
 To complete the proof we are going to choose sequences $(g_n)$ and~$(\lambda_n)$ as above
so that $g$ is analytic.
Now for $t\in\R$ we have
$$g_n(t)\ =\ (\lambda_n/\pi)^{1/2}\int_{-\infty}^{\infty} h_n(s)\ex^{-\lambda_n(s-t)^2}ds\ =\ (\lambda_n/\pi)^{1/2}\int_{a_{n+2}}^{b_{n+2}} h_n(s)\ex^{-\lambda_n(s-t)^2}ds$$
and $g_n$ is the restriction to $\R$ of the entire function $\hat g_n$ given by
$$\hat g_n(z)\ =\ (\lambda_n/\pi)^{1/2}\int_{a_{n+2}}^{b_{n+2}} h_n(s)\ex^{-\lambda_n(s-z)^2}ds\qquad (z\in\mathbb C).$$ 
(See the proof of Lemma~\ref{lem:WAP aux 1}.)
Put $$\rho_n\ :=\ \textstyle\frac{1}{2}\min\!\big\{ (a_n-a_{n+1})^2,(b_{n+1}-b_{n})^2\big\}\in\R^>$$ and
$$U_n\ :=\  \big\{z\in \mathbb C:\, a_{n+1}<\Re z<b_{n+1},\ 
\Re\!\big((z-a_{n+1})^2\big),\  \Re\!\big((z-b_{n+1})^2\big)>\rho_n
  \big\},$$
an open subset of $\mathbb C$ containing $K_n$ such that $\Re\!\big((s-z)^2\big)>\rho_n$
for all~$s\in \R\setminus K_{n+1}$ and~$z\in U_n$. (Cf.~Figure~\ref{fig:Un}.)

\begin{figure}[ht]
\begin{tikzpicture}
\def\anplone{1.5};\def\an{3.5};\def\bn{5.4};\def\bnplone{6.8};\def\rhon{((1.4^2)/2)};
  \begin{axis} [axis lines=center, xmin=-0.2, xmax=8, ymin = -3.15, ymax = 3.15, width=0.9\textwidth, height = 0.5\textwidth, xtick={\anplone,\an,\bn,\bnplone}, ytick=\empty, xticklabels={\strut $a_{n+1}$, \strut $a_{n}$, \strut $b_{n}$, \strut $b_{n+1}$}, xticklabel style = {xshift=0.04\textwidth, yshift=-0.125\textwidth, xlabel={$\Re$}, ylabel={$\Im$}}
   ]    
    \draw[dotted] (axis cs: \anplone,-3.15) -- (axis cs: \anplone,3.15);
    \draw[dotted] (axis cs: \an,-3.15) -- (axis cs: \an,3.15);
    \draw[dotted] (axis cs: \bn,-3.15) -- (axis cs: \bn,3.15);
    \draw[dotted] (axis cs: \bnplone,-3.15) -- (axis cs: \bnplone,3.15);

    \addplot [name path=A, domain=\anplone:((\anplone+\bnplone)/2), smooth, dashed, thick] {  ( (x-\anplone)^2 - \rhon )^(1/2) };
    \addplot [name path=B, domain=\anplone:((\anplone+\bnplone)/2), smooth, dashed, thick] { - ( (x-\anplone)^2 - \rhon )^(1/2) };

    \addplot [name path=C, domain=((\anplone+\bnplone)/2):\bnplone, smooth,  dashed, thick] {   ( (x-\bnplone)^2 - \rhon )^(1/2) };
    \addplot [name path=D, domain=((\anplone+\bnplone)/2):\bnplone, smooth, dashed, thick] { - ( (x-\bnplone)^2 - \rhon )^(1/2) };

        \addplot[pattern=north west lines] fill between[of=A and B];
        \addplot[pattern=north west lines] fill between[of=C and D];

    \node[right] at (axis cs:\bnplone+0.3,2.75) {$\mathbb C$};
    \node[right] at (axis cs:\anplone+1.5,2.11) {$U_n$};
 \end{axis}
\end{tikzpicture}
\caption{The domain $U_n$}\label{fig:Un}
\end{figure}

\medskip
\noindent
We also set
$$H_m\ :=\ 2(\lambda_m/\pi)^{1/2}\,\dabs{h_m}_{K_{m+2}} \,(b_{m+2}-a_{m+2})\in\R^\ge.$$
Recall that $h_m$ only depends on the $g_j$ with $j<m$. Fix a sequence $(c_m)$ of positive reals such that $\sum_m c_m <\infty$. Then we can and do choose the sequences $(g_m)$, $(\lambda_m)$ so that in addition
$$H_m \exp(-\lambda_m/m)\ \leq\  c_m\quad\text{for all $m\ge 1$.}$$
Then  
\begin{equation}\label{eq:cn}
\sum_m H_m\exp(-\lambda_m\rho)\ <\infty\ \qquad\text{for all~${\rho\in\R^>}$.}
\end{equation}
It is enough  that 
for each~$n$ the series~$\sum_m \hat g_m$ converges uniformly on 
compact subsets of $U_n$, because then by \cite[(9.12.1)]{D} we have a holomorphic function
 $$z\mapsto \sum_m \hat g_m(z)\ :\  U:=\bigcup_n U_n\to\mathbb C$$
whose restriction to $I$ is $g$.
To prove such convergence, fix $n$ and
let $m\geq n+2$. Then $\supp h_m\subseteq   K_{m+2}\setminus K_{m-1}\subseteq
K_{m+2}\setminus K_{n+1}$. Hence $\abs{\hat g_m(z)} \leq H_m\ex^{-\lambda_m\rho_n}$
for~${z\in U_n}$. 
Together with \eqref{eq:cn} 
this now yields that $\sum_m \hat g_m$ converges uniformly on compact subsets of $U_n$. \qed

\medskip

\noindent
In the remainder of this appendix we discuss how to control the domain of the holomorphic function~$\hat g$
 in the proof of Theorem~\ref{thm:WAP dim 1, general}; this leads to   improvements of Corollaries~\ref{cor:WAP dim 1} and~\ref{cor:WAP dim 1, 1} which might be useful elsewhere: Corollaries~\ref{cor:WAP dim 1, controlled} and~\ref{cor:WAP dim 1, 1, controlled} below. 
For the next corollary we are in the setting of that theorem and~$f\in \Cc^r(I)$. With~$\alpha\in \R\cup\{-\infty\}$ and~$\beta\in \R\cup\{+\infty\}$ such that~$I=(\alpha,\beta)$,   put
$$V\ :=\ \big\{z\in\mathbb C:\  \Re(z)\in I,\ \abs{\Im z} < \Re(z)-\alpha,\,\beta-\Re(z)\big\},$$
an open subset of $\mathbb C$ containing $I$.

\begin{cor}\label{cor:WAP dim 1, general, controlled}
Suppose   $a_n-a_{n+1}\to 0$ and $ b_{n+1}-b_n\to 0$ as $n\to \infty$. 
Then there is a holomorphic $\hat g\colon V\to\mathbb C$, real-valued on $\R$,   
such that $g:=\hat g|_I\in\C^\omega(I)$ satisfies 
$$\dabs{{f-g}}_{K_{n+1}\setminus K_n;\,r_n}<\varepsilon_n, \text{  for all~$n$}.$$
\end{cor}
\begin{proof}
It suffices to show that the open set $U\subseteq\mathbb C$ in   the proof of Theorem~\ref{thm:WAP dim 1, general} contains $V$. Note that $\rho_n\to 0$ as $n\to \infty$.
Let $z=x+y\imag\in V$ ($x,y\in\R$). Then 
$$ (x-a_{n+1})^2-y^2-\rho_n \to (x-\alpha)^2-y^2 > 0 \text{ as }n\to \infty,$$
and thus $\Re\!\big((z-a_{n+1})^2\big)=(x-a_{n+1})^2-y^2 > \rho_n$ for all sufficiently large $n$. Likewise, ${\Re\!\big((z-b_{n+1})^2\big) > \rho_n}$  for all sufficiently large $n$. Therefore $z\in U_n$ for sufficiently large $n$.
\end{proof}

\begin{cor}\label{cor:Carleman}
Suppose $r\in\N$, $f\in\C^r(\R)$, $\varepsilon\in\C(\R)$, and $\varepsilon>0$ on $\R$. Then 
there is an entire function $g\colon\mathbb C\to\mathbb C$ such that~${\abs{(f-g)^{(k)}}\le\varepsilon}$ on $\R$ for all $k\leq r$.
\end{cor}
\begin{proof}
Set $b_n:=\log(n+1)$, $a_n:=-b_n$, and $K_n:=[a_n,b_n]$. Then $\bigcup_n K_n=\R$ and~$a_n-a_{n+1}\to 0$ and $ b_{n+1}-b_n\to 0$ as $n\to \infty$. Set $\varepsilon_n:=\min\big\{\varepsilon(t): t\in K_{n+1}\big\}$ and~$r_n:=r$. Then $V=\mathbb C$ and we apply Corollary~\ref{cor:WAP dim 1, general, controlled}.
\end{proof}

\begin{remark}
Corollary~\ref{cor:Carleman} is due to Carleman~\cite{Carleman} for $r=0$, 
to Kaplan~\cite{Kaplan} for~${r=1}$, and to Hoischen~\cite[Satz~2]{Hoischen} in general; see~\cite[Chapter~VIII, pp.~273--276, 291]{Burckel}. In a similar way, Corollary~\ref{cor:WAP dim 1, general, controlled}  also yields the $\C^\infty$-version of Corollary~\ref{cor:Carleman}  in
\cite[Satz~1]{Hoischen}. For a multivariate version of these facts, see \cite{AADC}.
\end{remark}

\noindent
Given any $a$  we now consider the open sector $V_a$ in the complex plane given by
$$V_a\ :=\ \big\{z\in\mathbb C: \abs{\Im (z)} < \Re(z)-a\big\}\ =\ a+\big\{z\in\mathbb C^\times: -\textstyle\frac{\pi}{4}<\arg z<\frac{\pi}{4}  \big\}.$$

\begin{cor} \label{cor:WAP dim 1, controlled}
Let  $f$, $(b_n)$, $(\varepsilon_n)$,  $(r_n)$ be as in  Corollary~\ref{cor:WAP dim 1}. Then there are $a<~b$ and a holomorphic
function $\hat g\colon V_a\to\mathbb C$, real-valued on $\R$, 
such that~$g:=\hat g|_{\R^{\geq b}}\in\C^\omega_b$ satisfies~$\dabs{{f-g}}_{[b_n,b_{n+1}];\,r_n}<\varepsilon_n$ for all~$n$.
\end{cor}
\begin{proof}
We first arrange that  $b_{n+1}-b_n\to 0$ as $n\to\infty$.  For this, let  $(b_m^*)$ be the strictly increasing sequence in $\R$ such that
$$\{b_0^*,b_1^*,\dots\}\ =\ \{b_0,b_1,\dots\}\cup\big\{b+\log 1,b+\log 2,\dots\big\},$$
for each $m$, set $\varepsilon_m^*:=\varepsilon_n$, $r_m^*:=r_n$ with $n$ such that $[b_{m}^*,b_{m+1}^*]\subseteq [b_n,b_{n+1}]$, and replace $(b_n)$, $(\varepsilon_n)$,  $(r_n)$  by $(b_m^*)$, $(\varepsilon_m^*)$,  $(r_m^*)$.
Now argue as in the proof of Corollary~\ref{cor:WAP dim 1}, using
Corollary~\ref{cor:WAP dim 1, general, controlled} instead of Theorem~\ref{thm:WAP dim 1, general}.
\end{proof}

\noindent
Now the proof of  Corollary~\ref{cor:WAP dim 1, 1}, using 
Corollary~\ref{cor:WAP dim 1, controlled} instead of
 Corollary~\ref{cor:WAP dim 1}, gives:
 
\begin{cor} \label{cor:WAP dim 1, 1, controlled}
Let     $f$, $\varepsilon$ be as in Corollary~\ref{cor:WAP dim 1, 1}.
Then there are $a<b$ and a holomorphic~$\hat g\colon V_a\to\mathbb C$, real-valued on $\R$,
such that~$g:=\hat g|_{\R^{\geq b}}\in\C^\omega_b$ 
satisfies~${\abs{(f-g)^{(k)}(t)}} < \varepsilon(t)$ for all $t\ge b$ and $k\leq \min\!\big\{r,1/\varepsilon(t)\big\}$.
\end{cor}

\newlength\templinewidth
\setlength{\templinewidth}{\textwidth}
\addtolength{\templinewidth}{-2.25em}

\patchcmd{\thebibliography}{\list}{\printremarkbeforebib\list}{}{}

\let\oldaddcontentsline\addcontentsline
\renewcommand{\addcontentsline}[3]{\oldaddcontentsline{toc}{section}{References}}

\def\printremarkbeforebib{\bigskip\hskip1em The citation [ADH] refers to our book \\

\hskip1em\parbox{\templinewidth}{
M. Aschenbrenner, L. van den Dries, J. van der Hoeven,
\textit{Asymptotic Differential Algebra and Model Theory of Transseries,} Annals of Mathematics Studies, vol.~195, Princeton University Press, Princeton, NJ, 2017.
}

\bigskip

}


\begin{thebibliography}{99}

\bibitem{AADC}
C. Andradas, E. Aneiros,  A. D\'\i{}az-Cano, 
\textit{An extension of Whitney approximation theorem,} in:
M. Castrill\'on L\'opez et al. (eds.),
\textit{Contribuciones Matem\'aticas en Honor a Juan Tarr\'es,}
pp.~17--25, Universidad Complutense de Madrid, Facultad de Ciencias Matem\'aticas, Madrid, 2012. 



\bibitem{VDF}
M. Aschenbrenner,  L. van den Dries, J. van der Hoeven, {\em Maximal immediate extensions of valued differential fields,}
Proc. London Math. Soc. {\bf 117} (2018), no.~2, 376--406.
 
\bibitem{ADH1+} M. Aschenbrenner, L. van den Dries, J. van der Hoeven, \textit{The surreal numbers as a universal $H$-field,}  J. Eur. Math. Soc. {\bf 21} (2019), 1179--1199.

\bibitem{ADH3} M. Aschenbrenner, L. van den Dries, J. van der Hoeven, \textit{Revisiting closed asymptotic couples},  Proc. Edinb. Math. Soc. (2) {\bf 65} (2022), 530--555.



\bibitem{ADHfgh}
M. Aschenbrenner, L. van den Dries, J. van der Hoeven,  
 {\it Filling gaps in Hardy fields}\/,  J. Reine Angew. Math. {\bf 815} (2024), 107--172.  


\bibitem{ADH5} 
M. Aschenbrenner,  L. van den Dries, J. van der Hoeven, {\em  Constructing $\upo$-free Hardy fields,} J. Anal. Math., to appear,  {\tt 	arXiv:2404.03695}, 2024.

\bibitem{ADH6} 
M. Aschenbrenner,  L. van den Dries, J. van der Hoeven, {\em A Normalization Theorem in Asymptotic Differential Algebra,} preprint, 
 {\tt 	arXiv:2403.19732}, 2024.

\bibitem{ADH4} M. Aschenbrenner, L. van den Dries, J. van der Hoeven, 
\textit{The theory of maximal Hardy fields},   preprint, {\tt arXiv:2408.05232}, 2024.  
 


\bibitem{ADHdim+} M. Aschenbrenner, L. van den Dries, J. van der Hoeven, \textit{More on dimension in transseries}, in preparation.

\bibitem{Barwise} J. Barwise, \textit{Back and forth through infinitary logic,} in:  M. Morley (ed.), \textit{Studies in Model Theory,} pp.~5--34,  MAA Studies in Mathematics, vol. 8, Mathematical Association of America, Buffalo, N.Y., 1973.

\bibitem{BM} A. Berarducci, V. Mantova, \textit{Surreal numbers, derivations, and transseries,} J. Eur. Math. Soc.  {\bf 20} (2018), 339--390. 

\bibitem{Besi}
A. Besikowitsch, \textit{\"Uber analytische Funktionen mit vorgeschriebenen Werten ihrer Ableitungen,}
 Math. Z. {\bf 21} (1924), 111--118.

\bibitem{Blass} A. Blass, \textit{Combinatorial cardinal characteristics of the continuum,} in: M. Foreman, A.  Ka\-na\-mori (eds.), \textit{Handbook of Set Theory,} Vol. 1, pp.~395--489, Springer, Dordrecht, 2010. 

 
\bibitem{Borel}
E. Borel,   {\em Sur quelques points de la th\'eorie des fonctions,}
Ann. Sci. \'Ecole Norm. Sup. (3) {\bf 12} (1895), 9--55. 



\bibitem{Burckel}
R. Burckel, \textit{An Introduction to Classical Complex Analysis. Vol. 1,}
Pure and Applied Mathematics, vol. 82, Academic Press, Inc., New York-London, 1979.

\bibitem{Carleman}
T. Carleman, \textit{Sur un th\'eor\`eme de Weierstrass,} Ark. Mat. Astr. Fys. {\bf 20B} (1927), no. 4, 1--5.

\bibitem{Ci1}
K. Ciesielski, \textit{A short ordered commutative domain whose quotient field is not short,} Algebra Universalis {\bf 25} (1988), no. 1, 1--6.

\bibitem{Ci2}
K. Ciesielski, \textit{$2^{2^\omega}$ nonisomorphic short ordered commutative domains whose quotient fields are long,} Proc. Amer. Math. Soc. {\bf 113} (1991), no. 1, 217--227.  

\bibitem{DW}
H. G. Dales, W. H. Woodin, \textit{Super-Real Fields,}
London Mathematical Society Monographs, New Series, vol. 14,  Oxford University Press, New York, 1996.

\bibitem{D} J. Dieudonn\'e, \textit{Foundations of Modern Analysis,}
 Pure Appl. Math., vol.~10, Academic Press, New York-London, 1960. 


\bibitem{vdDMM}
L. van den Dries, A. Macintyre, D. Marker, {\em Logarithmic-exponential series,} Ann. Pure Appl. Logic {\bf 111} (2001), 61--113. 


\bibitem{Esterle} J. Esterle, {\it Solution d'un probl\`eme d'Erd\"{o}s, Gillman et Henriksen et application \`{a} l'\'{e}tude des homomorphismes de $C(K)$,} Acta Math. (Hungarica) {\bf 30} (1977), 113--127.

\bibitem{Forn}
A. Fornasiero, {\it Dimensions, matroids, and dense pairs of first-order structures,}
Ann. Pure Appl. Logic {\bf 162} (2011), 514--543.

\bibitem{Gokhman}
D. Gokhman, \textit{Functions in a Hardy field not ultimately {$C^\infty$}}, 
Complex Variables Theory Appl. {\bf 32} (1997), no. 1, 1--6. 

\bibitem{Grelowski}
K. Grelowski, {\it 
Extending Hardy fields by non-$\mathcal C^\infty$-germs,}
Ann. Polon. Math. {\bf 93} (2008), no.~3, 281--297.


\bibitem{HarringtonShelah}
L. Harrington, S. Shelah, \textit{Counting equivalence classes for co-$\kappa$-Souslin equivalence relations,}
in:  D. van Dalen et al. (eds.), \textit{Logic Colloquium '80,} pp. 147--152, 
 Studies in Logic and the Foundations of Mathematics, vol. 108, North-Holland Publishing Co., Amsterdam-New York, 1982.


\bibitem{Hau07}
F. Hausdorff, {\em Untersuchungen \"uber Ordnungstypen,} IV, V, Ber. K\"onigl. S\"achs. Gesell. Wiss. Leipzig Math.-Phys. Kl. {\bf 59} (1907), 84--159.

\bibitem{Harzheim}
E. Harzheim, \textit{Ordered Sets,}
Advances in Mathematics, vol. 7, Springer, New York, 2005. 


\bibitem{Hoischen}
L. Hoischen, \textit{Eine Versch\"arfung eines Approximationssatzes von Carleman,} J. Approximation Theory {\bf 9} (1973), 272--277. 

\bibitem{Kaplan}
W. Kaplan, \textit{Approximation by entire functions,}
Michigan Math. J.  {\bf 3} (1955), 43--52. 


\bibitem{Nara}
R. Narasimhan, \textit{Analysis on Real and Complex Manifolds,} 2nd ed., Advanced Studies in Pure Mathematics, vol. 1, Masson \& Cie, \'Editeurs, Paris; North-Holland Publishing Co., Amsterdam-London; American Elsevier Publishing Co., Inc., New York, 1973.  

\bibitem{Nigel19}
N. Pynn-Coates, {\em Newtonian valued differential fields with arbitrary value group}, Comm. Algebra {\bf 47} (2019), no.~7, 2766--2776.

\bibitem{RSW}
J.-P. Rolin, P. Speissegger, A. J. Wilkie, {\it
Quasianalytic Denjoy-Carleman classes and o-min\-i\-mal\-i\-ty,}
J. Amer. Math. Soc. {\bf 16} (2003), no.~4, 751--777. 




\bibitem{Rosenstein}
J. Rosenstein, \textit{Linear Orderings,} Pure and Applied Mathematics, vol. 98, Academic Press, Inc., New York-London, 1982.

\bibitem{Rudin}
M. E. Rudin, \textit{Martin's axiom,} in:  J. Barwise (ed.), \textit{Handbook of Mathematical Logic,} pp.~491--501,
 Studies in Logic and the Foundations of Mathematics, vol. 90, North-Holland Publishing Co., Amsterdam, 1977.
 
\bibitem{S} G. Sj\"odin, \textit{Hardy-fields}, Ark. Mat. {\bf 8} (1970), no. 22, 217--237.

\bibitem{ST}
R. M. Solovay, S. Tennenbaum, \textit{Iterated Cohen extensions and Souslin's problem,}
Ann. of Math. (2) {\bf 94} (1971), 201--245.


\bibitem{Urysohn23}
P. Urysohn, \textit{Un th\'eor\`eme sur la puissance des ensembles ordonn\'es,} Fund. Math. {\bf 5}
(1923), 14--19.

\bibitem{Urysohn24}
P. Urysohn, \textit{Remarque sur ma note: ``Un th\'eor\`eme sur la puissance des ensembles
ordonn\'es'',} Fund. Math. {\bf 6} (1924), 278.


\bibitem{Whitney}
H.~Whitney, \textit{Analytic extension of differentiable functions defined on closed sets,} Trans. Amer. Math. Soc. {\bf 36} (1934), 63--89.

\end{thebibliography}
\end{document}